\documentclass[letterpaper, pdftex,11pt]{amsart} 
\usepackage[pdftex]{hyperref} 
\usepackage{amsmath,amssymb}
\usepackage{mathtools}
\usepackage{amscd}
\usepackage{mathrsfs}
\usepackage[foot]{amsaddr}
\usepackage{amsthm}
\usepackage{setspace}
\usepackage[all]{xy}
\usepackage{comment}
\usepackage[pdftex]{graphicx} 
\usepackage{tikz}
\usetikzlibrary{arrows}
\usetikzlibrary{shapes.geometric,fit}
\usepackage{tikz-cd}
\usepackage{fancyhdr}

\theoremstyle{definition}
\newtheorem{thm}{Theorem}[subsection]
\newtheorem*{thm*}{Theorem}
\newtheorem{defi}[thm]{Definition}
\newtheorem*{defi*}{Definition}
\newtheorem*{acknowledge}{Acknowledgement}
\newtheorem{cor}[thm]{Corollary}
\newtheorem*{cor*}{Corollary}
\newtheorem{prop}[thm]{Proposition}
\newtheorem*{prop*}{Proposition}
\newtheorem{lem}[thm]{Lemma}
\newtheorem*{lem*}{Lemma}
\newtheorem{ex}[thm]{Example}
\newtheorem*{ex*}{Example}

\newtheorem{rem}[thm]{Remark}
\newtheorem*{rem*}{Remark}

\newtheorem*{claim*}{Claim}

\newtheorem*{hw*}{Homework}

\DeclareMathOperator{\Ind}{Ind}
\newcommand{\C}{\mathbb{C}}

\newcommand{\R}{\mathbb{R}}

\newcommand{\Z}{\mathbb{Z}}
\newcommand{\N}{\mathbb{N}}
\newcommand{\T}{\mathbb{T}}
\newcommand{\ad}{\mathrm{ad}}

\newcommand{\Bis}{\mathrm{Bis}}
\DeclareMathOperator{\Iso}{\mathrm{Iso}}

\DeclareMathOperator{\Span}{\mathrm{span}}

\DeclareMathOperator{\dom}{\mathrm{dom}}
\DeclareMathOperator{\ab}{\mathrm{ab}}
\DeclareMathOperator{\id}{\mathrm{id}}
\DeclareMathOperator{\Aut}{\mathrm{Aut}}

\def\i<#1>{\langle #1 \rangle}
\def\l<#1>{\left\langle #1 \right\rangle}

\makeatletter
\renewenvironment{proof}[1][\proofname]{\par
  \normalfont
  \topsep6\p@\@plus6\p@ \trivlist
  \item[\hskip\labelsep{\bfseries #1}\@addpunct{.}]\ignorespaces
}{
  \endtrivlist
}
\renewcommand{\proofname}{\sc{Proof}}
\makeatother
\makeatletter
\newcommand*{\defeq}{\mathrel{\rlap{%
                     \raisebox{0.3ex}{$\m@th\cdot$}}%
                     \raisebox{-0.3ex}{$\m@th\cdot$}}%
                     =}
\makeatother

\makeatletter
\@namedef{subjclassname@2020}{\textup{2020} Mathematics Subject Classification} 
\makeatother

\title[]{Weyl groups of groupoid C*-algebras}
\author{Fuyuta Komura}
\address{Center for Advanced Intelligence Project, RIKEN,
	3-14-1 Hiyoshi, Kohoku-ku, Yokohama, 223-8522, Japan}
\address{Phone: +81-45-566-1641+42706}
\address{Fax: +81-45-566-1642}
\email{fuyuta.k@keio.jp}

\begin{document}
\maketitle
\begin{abstract}
	
	In the theory of C*-algebras,
	the Weyl groups were defined for the Cuntz algebras and graph algebras by Cuntz and Conti et al.\ respectively.	
	In this paper,
	we introduce and investigate the Weyl groups of groupoid C*-algebras as a natural generalization of the existing Weyl groups.
	Then we analyse several groups of automorphisms on groupoid C*-algebras.
	Finally,
	we apply our results to Cuntz algebras,
	graph algebras and C*-algebras associated with Deaconu-Renault systems.

\end{abstract}

\tableofcontents

\section{Introduction}

	In the theory of C*-algebras,
	the study of Weyl groups was initiated by Cuntz for the Cuntz algebras $\mathcal{O}_n$ in \cite{Cuntz1980}.
	In \cite{Cuntz1980},
	Cuntz investigated automorphisms on $\mathcal{O}_n$ which preserve the canonical MASA $D_n\subset \mathcal{O}_n$.
	Cuntz computed the group of $D_n$-preserving automorphisms on $\mathcal{O}_n$ and analysed its algebraic and topological properties.
	For example,
	Cuntz proved that the Weyl group
	\[
	\Aut(\mathcal{O}_n;D_n)/\Aut_{D_n}(\mathcal{O}_n)
	\]
	becomes a discrete group,
	where $\Aut(\mathcal{O}_n;D_n)$ and $\Aut_{D_n}(\mathcal{O}_n)$ denote the groups of automorphisms on $\mathcal{O}_n$ which preserve $D_n$ globally and pointwisely respectively.
	In \cite{ContiHongSzymaski},
	the authors investigated automorphisms on $\mathcal{O}_n$ which preserve both of $D_n$ and the gauge invariant subalgebra $\mathcal{O}_n^{\T}$,
	which is a problem proposed by Cuntz in \cite{Cuntz1980}.
	The celebrated result in \cite{ContiHongSzymaski} asserts that the restricted Weyl group
	\[
	\Aut(\mathcal{O}_n;\mathcal{O}_n^{\T}, D_n)/\Aut_{D_n}(\mathcal{O}_n)
	\]
	is isomorphic to the group of homeomorphisms on $\{1,2,\cdots,n\}^{\N}$ which eventually commute with the shift on $\{1,2,\cdots,n\}^{\N}$,
	where \[\Aut(\mathcal{O}_n;\mathcal{O}_n^{\T}, D_n)\] denotes the group of automorphisms on $\mathcal{O}_n$ which preserves $\mathcal{O}_n^{\T}$ and $D_n$ globally.
	We remark that $\Aut_{D_n}(\mathcal{O}_n)$ is automatically contained in \[\Aut(\mathcal{O}_n;\mathcal{O}_n^{\T}, D_n)\]
	as a normal subgroup.
	In \cite{CONTI20122529},
	these results were generalized to graph algebras.
	Under some assumptions,
	the authors showed in \cite[Theorem 4.13]{CONTI20122529} that the restricted Weyl group of a graph algebra $C^*(E)$ can be embedded into the group of homeomorphisms on $E^{(\infty)}$ which eventually commute with the shift on $E^{(\infty)}$,
	where $E^{(\infty)}$ denotes the infinite path space on $E$.
	The authors in \cite{CONTI20122529} ask if this embedding is an isomorphism and this seems to be still an open problem.
	In any case,
	the authors in \cite{ContiHongSzymaski} and \cite{CONTI20122529} characterized automorphisms which preserve some subalgebras in terms of symbolic dynamical systems.
	These results revealed the new relation between C*-algebras and symbolic dynamical systems.
	
	We explain other existing researches which deal with both of C*-algebras and dynamical systems.
	In \cite{Giordano1995},
	the authors investigated the relation between Cantor minimal systems and associated crossed product C*-algebras.
	Since then,
	the works in \cite{Giordano1995} are generalized in various ways.
	For example,
	in \cite[Theorem 8.2]{CARLSEN2021107923},
	the authors characterised the continuous orbit equivalence between Deaconu-Renault systems,
	which are kinds of dynamical systems,
	in terms of C*-algebras associated with Deaconu-Renault systems.
	In addition,
	in \cite[Theorem 8.10]{CARLSEN2021107923},
	the authors obtained a characterization of the eventual conjugacy between Deaconu-Renault systems.
	In \cite[Theorem 3.1]{ARMSTRONG_BRIX_CARLSEN_EILERS_2023},
	the authors characterised the conjugacy between Deaconu-Renault systems.
	We note that the eventual conjugacy is a stronger notion than the continuous orbit equivalence and 
	the conjugacy is a stronger notion than the eventual conjugacy.
	One of the key steps in \cite{CARLSEN2021107923} and \cite{ARMSTRONG_BRIX_CARLSEN_EILERS_2023} is to recover the information about underlying Deaconu-Renault systems from associated C*-algebras.
	Renault's celebrated work about Cartan C*-subalgebras in \cite{renault} provides us a technique to overcome this step and allows us to recover the information about the continuous orbit equivalence of Deaconu-Renault systems.
	To recover the eventual conjugacy,
	the authors in \cite{CARLSEN2021107923} utilized *-isomorphisms which are compatible with gauge actions on associated C*-algebras.
	Similarly,
	the authors in \cite{ARMSTRONG_BRIX_CARLSEN_EILERS_2023} utilized *-isomorphisms compatible with more various actions on associated C*-algebras to recover the conjugacy of Deaconu-Renault systems.
	As we can see from their works,
	we can recover more detailed information by analysing *-homomorphisms which are compatible with additional structures like gauge actions.
	
	In this paper,
	we aim to generalize the above results to groupoid C*-algebras.
	Since graph algebras and C*-algebras associated with Deaconu-Renault systems can be realized as groupoid C*-algebras by \cite{Paterson2002} and \cite{Deaconu1995groupoidendomorphism} respectively,
	it is natural to expect that groupoid C*-algebras provide us a natural framework to generalize the above results.
	Indeed,
	we investigate automorphisms on groupoid C*-algebras which are compatible with additional structures like actions on groupoid C*-algebras in this paper.
	Now,
	we explain our purpose more precisely.
	
	Let $G$ be a locally compact Hausdorff \'etale groupoid.
	Following \cite{renault1980groupoid},
	we associate an inclusion of C*-algebras $C_0(G^{(0)})\subset C^*_r(G)$,
	where $C^*_r(G)$ denotes the reduced groupoid C*-algebra of $G$.
	In the similar way as \cite{Cuntz1980},
	we define the Weyl group of $G$ in Definition \ref{def: Weyl groups of groupoid C*-algebras} as
	\[
	\mathfrak{W}_G\defeq \Aut(C^*_r(G); C_0(G^{(0)}))/\Aut_{C_0(G^{(0)})}C^*_r(G).
	\]
	This definition generalizes the existing Weyl groups as explained in Remark \ref{remark: Weyl group generalize existing one}.
	The main aim in Section \ref{Section: weyl groups of groupoid C*-algebras} is to investigate the algebraic and topological properties of the topological groups
	\[
	\Aut(C^*_r(G); C_0(G^{(0)})), \Aut_{C_0(G^{(0)})}C^*_r(G)\text{ and } \mathfrak{W}_G
	\]
	for an effective \'etale groupoid $G$.
	First,
	we point out that the Weyl group $\mathfrak{W}_G$ is nothing but the automorphism group $\Aut(G)$ if $G$ is effective.
	Indeed,
	in \cite[Proposition 5.7]{Matui2011} and \cite[Corollary 2.2.2]{komura2023homomorphisms},
	it is shown that there exists a group isomorphism
	\begin{align*}
	\Aut(C^*_r(G); C_0(G^{(0)}))&\simeq \Aut(G)\ltimes Z(G)
	\end{align*}
	which induces 
	\begin{align*}
	\Aut_{C_0(G^{(0)})}C^*_r(G)& \simeq Z(G),
	\end{align*}
	where $Z(G)$ denotes the group of $\T$-valued cocycles on $G$ and $\Aut(G)\ltimes Z(G)$ denotes the semidirect product with respect to the natural action (see Theorem \ref{thm: varphi is a group hom} for details).
	Then we investigate the Weyl group using this isomorphism.
	Our main results in Section \ref{Section: weyl groups of groupoid C*-algebras} are Theorem \ref{thm: Z(G) is maximal abelian group} and Theorem \ref{thm: Weyl group is discrete countable if G is effective and expansive}.
	Theorem \ref{thm: Z(G) is maximal abelian group} asserts that $\Aut_{C_0(G^{(0)})}C^*_r(G)$ is a maximal abelian subgroup in $\Aut(C^*_r(G))$ under the mild assumptions.
	Theorem \ref{thm: Weyl group is discrete countable if G is effective and expansive} asserts that the Weyl group $\mathfrak{W}_G$ is a discrete countable group if $G$ is expansive in the sense of Nekrashevych \cite[Definition 5.3]{NEKRASHEVYCH_2019} (see also Definition \ref{def: expansive}).
	We note that these results were known for Cuntz algebras in \cite{Cuntz1980} and we generalize to groupoid C*-algebras.
	
	In Section \ref{Section: Restricted Weyl groups of groupoid C*-algebras},
	we introduce the restricted Weyl groups of groupoid C*-algebras.
	For an \'etale effective groupoid $G$ and open subgroupoid $G^{(0)}\subset H\subset G$,
	we have an inclusion $C_0(G^{(0)})\subset C^*_r(H)\subset C^*_r(G)$.
	Then the restricted Weyl group of $H\subset G$ is defined as
	\[
	\mathfrak{RW}_{G,H}\defeq\Aut(C^*_r(G); C^*_r(H),C_0(G^{(0)}))/\Aut_{C_0(G^{(0)})}C^*_r(G).
	\]
	Note that,
	if $G$ is effective,
	an element in $\Aut_{C_0(G^{(0)})}C^*_r(G)$ globally preserves $C^*_r(H)$ by Theorem \ref{thm: varphi is a group hom} and we may take this quotient group.
	As in the case of Weyl groups,
	this definition generalizes the existing restricted Weyl groups investigated in \cite{ContiHongSzymaski} and \cite{CONTI20122529}.
	Our main purpose in Section \ref{Section: Restricted Weyl groups of groupoid C*-algebras} is to analyse the groups
	\[
	\Aut(C^*_r(G); C^*_r(H), C_0(G^{(0)})), \Aut_{C^*_r(H)}C^*_r(G) \text{ and }\mathfrak{RW}_{G,H}.
	\]
	First,
	we characterise these groups in terms of the underlying \'etale groupoids $H\subset G$.	
	Then we treat the case that $H$ is given as the kernel $\ker\sigma$ of a discrete group cocycle $\sigma\colon G\to \Gamma$.
	Our main results in Section \ref{Section: Restricted Weyl groups of groupoid C*-algebras} is Corollary \ref{cor: Galois group is dual of Gamma ab} and Corollary \ref{cor: restricted Weyl isomorphic to restricted groupoid Weyl}.
	Corollary \ref{cor: Galois group is dual of Gamma ab} asserts that $\Aut_{C^*_r(\ker\sigma)}C^*_r(G)$ is isomorphic to $\widehat{\Gamma^{\ab}}$ under some assumptions,
	where $\Gamma^{\ab}$ denotes the abelianization of $\Gamma$ and $\widehat{\Gamma^{\ab}}$ denotes the Pontryagin dual.
	This result indicates that the ``Galois group'' $\Aut_{C^*_r(\ker\sigma)}C^*_r(G)$ only remembers the abelianization of $\Gamma$ and we can rarely recover $\Gamma$ from $\Aut_{C^*_r(\ker\sigma)}C^*_r(G)$.
	Corollary \ref{cor: restricted Weyl isomorphic to restricted groupoid Weyl} asserts that the restricted Weyl group $\mathfrak{RW}_{G,\ker\sigma}$ is isomorphic to the group of automorphisms on $G$ compatible with the cocycle $\sigma$.
	We apply this result to compute the restricted Weyl groups of graph algebras and C*-algebras associated with Deaconu-Renault systems in Subsection\ \ref{subsection: Restricted Weyl group of graph algebras} and \ref{subsection: Restricted Weyl group of Deaconu-Renault system} respectively.
	
	In Section \ref{section: Examples and applications},
	we apply results in Section \ref{Section: weyl groups of groupoid C*-algebras} and \ref{Section: Restricted Weyl groups of groupoid C*-algebras} to examples.
	In Subsection \ref{subsection : Examples of coactions on Cuntz algebras and graph algebras},
	we apply our results to the Cuntz algebra $\mathcal{O}_n$.
	As an application,
	we construct simple C*-subalgebras $B_1, B_2\subset \mathcal{O}_n$ of finite Watatani indices such that $B_1\subsetneq B_2$ and $\Aut_{B_1}\mathcal{O}_n=\Aut_{B_2}\mathcal{O}_n$ holds.
	This shows that the ``Galois group'' $\Aut_B(A)$ rarely woks well for an inclusion of C*-algebras $B\subset A$.
	In Subsection \ref{Subsection: Groupoid model of the Cuntz algebra of infinite degree},
	we show
	\[
	\Aut_{\mathcal{O}_{\infty}^{\T}}\mathcal{O}_{\infty}\simeq \T.
	\]
	While the isomorphism 
	\[
		\Aut_{\mathcal{O}_{n}^{\T}}\mathcal{O}_{n}\simeq \T.
	\] is known for $n\in\N_{\geq 2}$,
	this isomorphism for $n=\infty$ seems to be a new result.
	Indeed,
	the existing proof in \cite[Proposition 4.4]{CONTI20122529} relies on a variant of the Cuntz-Takesaki correspondence,
	which is a one-to-one correspondence between the set of unital endomorphisms on $\mathcal{O}_n$ and the set of unitary elements in $\mathcal{O}_n$ (see \cite{Cuntz1980}, for example).
	Since the Cuntz-Takesaki correspondence is not available for $\mathcal{O}_{\infty}$,
	the existing proof of
	\[
	\Aut_{\mathcal{O}_{n}^{\T}}\mathcal{O}_{n}\simeq \T
	\]
	does not work for $\mathcal{O}_{\infty}$.
	In Subsection \ref{subsection: Deaconu-Renault systems},
	we define the flip eventual conjugacy between Deaconu-Renault systems.
	Then we characterize the flip eventual conjugacy in terms of \'etale groupoids and C*-algebras in Theorem \ref{theorem: characterization of flip eventually conjugate}.
	In Subsection \ref{subsection: Restricted Weyl group of Deaconu-Renault system},
	we investigate the restricted Weyl group $\mathfrak{RW}_{G(X,T),\ker\sigma_X}$ of an \'etale groupoid associated with a topologically free Deaconu-Renault system $(X,T)$.
	We show that $\mathfrak{RW}_{G(X,T),\ker\sigma_X}$ is isomorphic to the group of the eventually conjugate automorphisms on $(X,T)$ under some assumptions in Corollary \ref{cor: eventually conjugate automorphism is isomorphic to restricted Weyl of DR system}.
	It seems an interesting phenomena that ``flip'' cannot occur if a Deaconu-Renault system $T$ is far from injective (see Proposition \ref{prop: flip cannot occur if TU=X}).
	In Subsection \ref{subsection: Restricted Weyl group of graph algebras},
	we apply the results in Subsection \ref{subsection: Restricted Weyl group of Deaconu-Renault system} to graph algebras.
	As a consequence,
	in Corollary \ref{cor: answer to open problem of Conti},
	we obtain an affirmative answer of the open problem mentioned under \cite[Theorem 4.13]{CONTI20122529},
	which asks if the restricted Weyl group of a graph algebra is isomorphic to the group of eventually shift commuting automorphisms on the infinite path space.
	
	We apply our results in this paper to Cuntz algebras,
	C*-algebras associated with Deaconu-Renault systems and graph algebras.
	As a future work,
	the author hopes that the results in this paper are applied to other C*-algebras coming from \'etale groupoids.

	This paper is organized as follows.
	In Section \ref{section: Preliminaries},
	we introduce fundamental notions in this paper.
	In Section \ref{Section: weyl groups of groupoid C*-algebras} and \ref{Section: Restricted Weyl groups of groupoid C*-algebras},
	we introduce and investigate (restricted) Weyl groups of groupoid C*-algebras.
	In Section \ref{section: Examples and applications},
	we apply the results in Section \ref{Section: weyl groups of groupoid C*-algebras} and \ref{Section: Restricted Weyl groups of groupoid C*-algebras} to Cuntz algebras,
	C*-algebras associated with Deaconu-Renault systems and graph algebras.
	
	\begin{acknowledge}
		The author is grateful to Takeshi Katsura, Yuki Arano and Taro Sogabe for fruitful discussions.
		This work was supported by JST CREST Grant Number JPMJCR1913 and RIKEN Special Postdoctoral Researcher Program.
	\end{acknowledge}
		
\section{Preliminaries}\label{section: Preliminaries}
	
	In this section,
	we introduce fundamental notions about \'etale groupoids,
	groupoid C*-algebras and inverse semigroups.
	
	\subsection{\'Etale groupoids}
	
	In this subsection,
	we recall the fundamental notions about \'etale groupoids to fix the notations.
	See \cite{asims} and \cite{renault1980groupoid} for more details.
			
	A groupoid is a set $G$ together with a distinguished subset $G^{(0)}\subset G$,
	domain and range maps $d,r\colon G\to G^{(0)}$ and a multiplication 
	\[
	G^{(2)}\defeq \{(\alpha,\beta)\in G\times G\mid d(\alpha)=r(\beta)\}\ni (\alpha,\beta)\mapsto \alpha\beta \in G
	\]
	such that
	\begin{enumerate}
		\item for all $x\in G^{(0)}$, $d(x)=x$ and $r(x)=x$ hold,
		\item for all $\alpha\in G$, $\alpha d(\alpha)=r(\alpha)\alpha=\alpha$ holds,
		\item for all $(\alpha,\beta)\in G^{(2)}$, $d(\alpha\beta)=d(\beta)$ and $r(\alpha\beta)=r(\alpha)$ hold,
		\item if $(\alpha,\beta),(\beta,\gamma)\in G^{(2)}$,
		we have $(\alpha\beta)\gamma=\alpha(\beta\gamma)$,
		\item\label{inverse} every $\gamma \in G$,
		there exists $\gamma'\in G$ which satisfies $(\gamma',\gamma), (\gamma,\gamma')\in G^{(2)}$ and $d(\gamma)=\gamma'\gamma$ and $r(\gamma)=\gamma\gamma'$.   
	\end{enumerate}
	Since the element $\gamma'$ in (\ref{inverse}) is uniquely determined by $\gamma$,
	$\gamma'$ is called the inverse of $\gamma$ and denoted by $\gamma^{-1}$.
	We call $G^{(0)}$ the unit space of $G$.
	A subgroupoid of $G$ is a subset of $G$ which is closed under the inversion and multiplication. 
	We define $G_x\defeq G_{\{x\}}$, $G^x\defeq G^{\{x\}}$ and $G(x)\defeq G_x\cap G^x$ for $x\in G^{(0)}$.
	Then $G(x)$ is a group and called an isotropy group at $x\in G^{(0)}$.
	A subset $U\subset G^{(0)}$ is said to be invariant if $d(\alpha)\in U$ implies $r(\alpha)\in U$ for all $\alpha\in G$.
	Let $H$ be a groupoid.
	A map $\Phi\colon G\to H$ is called a groupoid homomorphism if for a pair $\alpha$ and $\beta\in G$ with $d(\alpha)=r(\beta)$,
	$\Phi(\alpha)$ and $\Phi(\beta)$ are composable and $\Phi(\alpha\beta)=\Phi(\alpha)\Phi(\beta)$ holds.
	For a groupoid homomorphism $\Phi\colon G\to H$,
	we write $\ker\Phi\defeq \Phi^{-1}(H^{(0)})$.
	
	A topological groupoid is a groupoid equipped with a topology where the multiplication and the inverse are continuous.
	A topological groupoid is said to be \'etale if the domain map is a local homeomorphism.
	Note that the range map of an \'etale groupoid is also a local homeomorphism.
	In this paper,
	we will always assume that \'etale groupoids are locally compact Hausdorff unless otherwise stated in the following sections.	
	A locally compact Hausdorff \'etale groupoid is said to be ample if it is totally disconnected.
	
	A subset $U$ of an \'etale groupoid $G$ is called a bisection if the restrictions $d|_U$ and $r|_U$ are injective.
	It follows that $d|_U$ and $r|_U$ are homeomorphisms onto their images if $U$ is an open bisection since $d$ and $r$ are open maps.
	
	An \'etale groupoid $G$ is said to be effective if $G^{(0)}$ coincides with the interior of $\Iso(G)$,
	where 
	\[
	\Iso(G)\defeq\{\alpha\in G\colon d(\alpha)=r(\alpha)\}
	\]
	is the isotropy of $G$.
	An \'etale groupoid $G$ is said to be topologically principal if 
	\[
	\{x\in G^{(0)}\mid G(x)=\{x\}\}
	\]
	is dense in $G^{(0)}$.
	If $G$ is Hausdorff and topologically principal,
	then $G$ is effective.
	If $G$ is second countable and effective,
	then $G$ is topologically principal (see \cite[Proposition 3.6]{renault}).
	
	An \'etale groupoid $G$ is said to be topologically transitive if $r(d^{-1}(U))$ is dense in $G^{(0)}$ for all non-empty open set $U\subset G^{(0)}$.
	Equivalently,
	$G$ is topologically transitive if and only if each non-empty open invariant subset $U\subset G^{(0)}$ is dense in $G^{(0)}$.
	If there exists $x\in G^{(0)}$ such that $r(d^{-1}(\{x\}))\subset G^{(0)}$ is dense,
	then $G$ is topologically transitive.
	The converse is true if $G$ is second countable by \cite[Lemma 3.4]{STEINBERG20192474}.
	
	Let $G$ be an \'etale groupoid and $\Gamma$ be a topological group.
	A groupoid homomorphism $\sigma \colon G\to \Gamma$ is called a cocycle.
	We let $Z(G)$ denote the set of all continuous cocycles $c\colon G\to \T$,
	where $\T$ denotes the circle group.
	Then $Z(G)$ is an abelian group with respect to the pointwise product.
	We let $Z(G^{(0)})$ denote the set of all continuous functions $f\colon G^{(0)}\to \T$.
	Then $Z(G^{(0)})$ is also an abelian group with respect to the pointwise product.
	We have a group homomorphism $\partial\colon Z(G^{(0)})\to Z(G)$ defined by
	\[
	\partial f(\alpha)\defeq f(r(\alpha))\overline{f(d(\alpha))}
	\]
	for $f\in Z(G^{(0)})$ and $\alpha\in G$.
	
	\subsection{Inverse semigroup actions}

	We recall the basic notions about inverse semigroups.
	See \cite{lawson1998inverse} or \cite{paterson2012groupoids} for more details.
	An inverse semigroup $S$ is a semigroup where for every $s\in S$ there exists a unique $s^*\in S$ such that $s=ss^*s$ and $s^*=s^*ss^*$.
	An element $s^*$ is called a generalized inverse of $s\in S$.
	By a subsemigroup of $S$,
	we mean a subset of $S$ which is closed under the product and generalized inverse of $S$.
	We denote the set of all idempotents in $S$ by $E(S)\defeq\{e\in S\mid e^2=e\}$.
	It is known that $E(S)$ is a commutative subsemigroup of $S$.
	An inverse semigroup which consists of idempotents is called a (meet) semilattice of idempotents.
	A zero element is a unique element $0\in S$ such that $0s=s0=0$ holds for all $s\in S$.
	An inverse semigroup with a unit is called an inverse monoid.
	A map $\varphi\colon S\to T$ between inverse semigroups $S$ and $T$ is called a semigroup homomorphism if $\varphi(st)=\varphi(s)\varphi(t)$ holds for all $s,t\in S$.
	Note that a semigroup homomorphism automatically preserves generalised inverses (i.e.\ $\varphi(s^*)=\varphi(s)^*$ holds for all $s\in S$).
	
	The set of all open bisections in a \'etale groupoid is an inverse semigroup as the following:
	
	\begin{ex}[{\cite[Proposition 2.2.4]{paterson2012groupoids}}]
		Let $G$ be a locally compact Hausdorff \'etale groupoid.
		The set of all open bisections in $G$ is denoted by $\mathrm{Bis}(G)$.
		For $U, V\in \mathrm{Bis}(G)$,
		their product is defined by
		\[
		UV\defeq \{\alpha\beta\in G\mid \alpha\in U, \beta\in V, d(\alpha)=r(\beta)\}.
		\]
		Then $UV\in\mathrm {Bis}(G)$ and $\mathrm{Bis}(G)$ is an inverse semigroup with respect to this product.
		Note that $U^*\in \mathrm{Bis}(G)$ is given by
		\[
		U^{-1}\defeq \{\alpha^{-1}\in G\mid \alpha\in U\}.
		\]
		In addition,
		let $\Bis^c(G)$ denote the set of all compact open bisections in $G$.
		Then $\Bis^c(G)$ is a subsemigroup of $\Bis(G)$.
	\end{ex}

	For a topological space $X$,
	we denote by $I_X$ the set of all homeomorphisms between open subsets in $X$.
	Then $I_X$ is an inverse semigroup with respect to the product defined by the composition of maps.
	For an inverse semigroup $S$,
	an inverse semigroup action $\alpha\colon S\curvearrowright X$ is a semigroup homomorphism $S\ni s\mapsto \alpha_s\in I_X$.
	In this paper, we always assume that every action $\alpha$ satisfies $\bigcup_{e\in E(S)}\dom(\alpha_e)=X$.
	If $S$ has a zero element,
	we assume that $\dom(\alpha_0)=\emptyset$.
	
	Next,
	we recall how to construct an \'etale groupoid from an inverse semigroup action.
	Let $X$ be a locally compact Hausdorff space.
	For an action $\alpha\colon S\curvearrowright X$,
	we associate an \'etale groupoid $S\ltimes_{\alpha}X$ as the following.
	First we put the set $S*X\defeq \{(s,x) \in S\times X \mid x\in \dom(\alpha_{s^*s})\}$.
	Then we define an equivalence relation $\sim$ on $S*X$ by declaring that $(s,x)\sim (t,y)$ holds if
	\[
	\text{$x=y$ and there exists $e\in E(S)$ such that $x\in \dom(\alpha_e)$ and $se=te$}.  
	\]
	Set $S\ltimes_{\alpha}X\defeq S*X/{\sim}$ and denote the equivalence class of $(s,x)\in S*X$ by $[s,x]$.
	The unit space of $S\ltimes_{\alpha}X$ is $X$, where $X$ is identified with the subset of $S\ltimes_{\alpha}X$ via the injective map
	\[
	X\ni x\mapsto [e,x] \in S\ltimes_{\alpha}X, x\in \dom(\alpha_e).
	\]
	The domain map and range maps are defined by
	\[
	d([s,x])=x, r([s,x])=\alpha_s(x)
	\]
	for $[s,x]\in S\ltimes_{\alpha}X$.
	The product of $[s,\alpha_t(x)],[t,x]\in S\ltimes_{\alpha}X$ is $[st,x]$.
	The inverse is $[s,x]^{-1}=[s^*,\alpha_s(x)]$.
	Then $S\ltimes_{\alpha}X$ is a groupoid in these operations.
	For $s\in S$ and an open set $U\subset \dom(\alpha_{s^*s})$,
	define 
	\[[s, U]\defeq \{[s,x]\in S\ltimes_{\alpha}X\mid x\in U\}.\]
	These sets form an open basis of $S\ltimes_{\alpha}X$.
	In these structures,
	$S\ltimes_{\alpha}X$ is a locally compact \'etale groupoid,
	although $S\ltimes_{\alpha}X$ is not necessarily Hausdorff.
	In this paper,
	we only treat inverse semigroup actions $\alpha\colon S\curvearrowright X$ such that $S\ltimes_{\alpha}X$ become Hausdorff.

	Let $S$ be an inverse semigroup with $0$ and $\Gamma$ be a discrete group.
	Put $S^{\times}\defeq S\setminus\{0\}$.
	A map $\theta\colon S^{\times}\to \Gamma$ is called a partial homomorphism if $\theta(st)=\theta(s)\theta(t)$ holds for all $s,t\in S^{\times}$ with $st\not=0$.
	Assume that $\theta\colon S^{\times}\to \Gamma$ is a partial homomorphism and $\alpha\colon S\curvearrowright X$ is an action on a topological space $X$.
	Then we associate a continuous cocycle $\widetilde{\theta}\colon S\ltimes_{\alpha} X \to \Gamma$ defined by
	\[
	\widetilde{\theta}([s,x])\defeq \theta(s)
	\]
	for all $[s,x]\in S\ltimes_{\alpha}X$.
	\subsection{Groupoid C*-algebras}
	
	We recall the definition of groupoid C*-algebras.
	
	Let $G$ be a locally compact Hausdorff \'etale groupoid.
	Then $C_c(G)$, the vector space of compactly supported continuous $\C$-valued functions on $G$, is a *-algebra with respect to the multiplication and the involution defined by
	\[
	f*g(\gamma)\defeq\sum_{\alpha\beta=\gamma}f(\alpha)g(\beta), f^*(\gamma)\defeq\overline{f(\gamma^{-1})},
	\]
	where $f,g\in C_c(G)$ and $\gamma\in G$.
	The left regular representation $\lambda_x\colon C_c(G)\to B(\ell^2(G_x))$ at $x\in G^{(0)}$ is defined by
	\[
	\lambda_x(f)\delta_{\alpha}\defeq \sum_{\beta\in G_{r(\alpha)}}f(\alpha)\delta_{\alpha\beta},
	\]
	where $f\in C_c(G)$, $\alpha\in G_x$ and $\{\delta_{\alpha}\}_{\alpha\in G_x}\subset \ell^2(G_x)$ denotes the standard complete orthonormal system of $\ell^2(G_x)$.
	The reduced norm $\lVert\cdot\rVert_{r}$ on $C_c(G)$ is defined by
	\[
	\lVert f\rVert_r\defeq \sup_{x\in G^{(0)}} \lVert \lambda_x(f)\rVert
	\]
	for $f\in C_c(G)$.
	We often omit the subscript `$r$' of $\lVert\cdot\rVert_r$ if there is no chance to confuse.
	The reduced groupoid C*-algebra $C^*_r(G)$ is defined to be the completion of $C_c(G)$ with respect to the reduced norm.
	Note that $C_c(G^{(0)})\subset C_c(G)$ is a *-subalgebra and this inclusion extends to the inclusion  $C_0(G^{(0)})\subset C^*_r(G)$.
	In addition,
	we have a faithful conditional expectation $E\colon C^*_r(G)\to C_0(G^{(0)})$ defined by
	\[
	E(f)=f|_{G^{(0)}}
	\]
	for all $f\in C_c(G)$ (see \cite[Proposition 10.2.6]{asims} for example).
	
	The reduced groupoid C*-algebra $C^*_r(G)$ can be embedded into $C_0(G)$ as in the following,
	which was originally proved by Renault.
	See also \cite[Proposition 9.3.3]{asims} for the proof.

	\begin{prop}[{\cite[Proposition II 4.2]{renault1980groupoid}}] \label{prop evaluation}
		Let $G$ be a locally compact Hausdorff \'etale groupoid.
		For $a\in C^*_r(G)$,
		$j(a)\in C_0(G)$ is defined by
		\[
		j(a)(\alpha)\defeq\i<\delta_{\alpha}|\lambda_{d(\alpha)}(a)\delta_{d(\alpha)}>
		\]
		for $\alpha\in G$\footnote{In this paper, inner products of Hilbert spaces are linear with respect to the right variables.}.
		Then $j\colon C^*_r(G)\to C_0(G)$ is a norm decreasing injective linear map.
		Moreover, $j$ is an identity map on $C_c(G)$.
		
	\end{prop}
	
	\begin{rem}
		Since $j\colon C^*_r(G)\to C_0(G^{(0)})$ is injective,
		we may identify $j(a)$ with $a$.
		Hence, we often regard $a$ as a function on $G$ and simply denote $j(a)$ by $a$.
		We often use the inequality
		\[\lVert a\rVert_{\infty}\leq \lVert a\rVert_{r}\]
		for $a\in C^*_r(G)$,
		where $\lVert\cdot\rVert_{\infty}$ and $\lVert\cdot\rVert_r$ denote the supremum and reduced norm respectively. 
	\end{rem}

	\subsection{Automorphism groups}
		
	For a C*-algebra $A$,
	$\Aut(A)$ denotes the group of *-automorphism on $A$.
	We equip $\Aut(A)$ with the strong norm topology.
	Namely,
	$\Aut(A)$ is equipped with the weakest topology where the map
	\[
	\Aut(A)\ni \varphi\mapsto \varphi(a)\in A
	\]
	is continuous for all $a\in A$.
	
	Let $H$ be a group and $A$ be a C*-algebra.
	An action $\tau\colon H\curvearrowright A$ is a group homomorphism $\tau\colon H\to \Aut(A)$.
	We denote the invariant subalgebra of $\tau$ by
	\[
	A^{\tau}\defeq \{a\in A\mid \tau_{\chi}(a)=a\text{ for all $\chi\in H$}\}.
	\]
	If there is no chance to confuse,
	we simply represent $A^{\tau}$ as $A^H$.
	For a topological group $H$,
	we say that an action $\tau\colon H\curvearrowright A$ is strongly continuous if $\tau\colon H\to \Aut(A)$ is continuous,
	where $\Aut(A)$ is equipped with the strong norm topology.

	\begin{defi}\label{defi: definitions of automorphism groups}
		Let $A$ be a C*-algebra,
		$B\subset A$ and $B_i\subset A$ be C*-subalgebras of $A$ for $i=1,2$.
		We define
		\begin{align*}
		\Aut(A; B)&\defeq \{\varphi\in\Aut(A)\mid \varphi(B)=B\},\\
		\Aut(A; B_1, B_2)&\defeq \{\varphi\in\Aut(A)\mid \varphi(B_i)=B_i \text{ for $i=1,2$}\},\\
		\Aut_B(A)&\defeq \{\varphi\in \Aut(A)\mid \varphi(b)=b \text{ for all $b\in B$}\}.
		\end{align*}
		We equip these groups with the relative topology of the strong norm topology of $\Aut(A)$.
	\end{defi}
	
	Note that $\Aut_B(A)$ is a normal subgroup of $\Aut(A;B)$.
	
	\begin{rem}\label{remark: conjugate action on AutBA}
		Let $A$ be a C*-algebra and $B\subset A$ be a C*-subalgebra.
		Then $\Aut(A,B)$ acts on $\Aut_BA$ by the conjugation since $\Aut_BA$ is a normal subgroup in $\Aut(A; B)$.
		Namely,
		we have the action $\ad\colon \Aut(A;B)\curvearrowright \Aut_BA$ defined by
		\[
		\ad_{\varphi}(\psi)\defeq \varphi\circ\psi\circ \varphi^{-1}
		\]
		for all $\varphi\in\Aut(A;B)$ and $\psi\in\Aut_BA$.
		Note that we have
		\[
		\varphi\circ\psi=\ad_{\varphi}(\psi)\circ\varphi
		\]
		for all $\varphi\in\Aut(A;B)$ and $\psi\in\Aut_BA$.
		
	\end{rem}
		
	Although the proof of the following proposition is straightforward,
	we include a proof for the reader's convenience.
	
	\begin{prop}\label{proposition: centralizer of AutBA is Aut(A,B)}
		Let $A$ be a C*-algebra and $B\subset A$ be a C*-subalgebra.
		Assume that $\varphi\in \Aut(A)$ and there exists $\tau_{\varphi}\in\Aut(\Aut_BA)$ such that
		\[
		\varphi\circ \psi=\tau_{\varphi}(\psi)\circ \varphi
		\]
		for all $\psi\in\Aut_BA$.
		In addition,
		assume that $B=A^{\Aut_BA}$ holds.
		Then $\varphi\in \Aut(A;B)$.
	\end{prop}
	\begin{proof}
		
		Take $b\in B$.
		For all $\psi\in\Aut_BA$,
		we have $\psi\circ \varphi=\varphi\circ\tau_{\varphi}^{-1}(\psi)$ and
		\[
		\psi(\varphi(b))=\varphi(\tau_{\varphi}^{-1}(\psi)(b))=\varphi(b).
		\]
		Hence we obtain $\varphi(b)\in A^{\Aut_BA}=B$ and $\varphi\in\Aut_BA$.
		\qed
	\end{proof}

	\subsection{Miscellaneous facts}
	
	In this subsection,
	we collect facts about groupoid C*-algebras which we will use in the following sections.
	We include proofs of propositions which the author can not find in literatures.
	
	We characterize the primeness of groupoid C*-algebras in terms of the underlying \'etale groupoids.
	Recall that a (two-sided closed) ideal $I\subset A$ of a C*-algebra $A$ is essential if $aI=\{0\}$ implies $a=0$ for all $a\in A$.
	A C*-algebra $A$ is said to be prime if every nonzero ideal of $A$ is essential.
	
\begin{prop}\label{prop: topologically transitive and prime}
	Let $G$ be a locally compact Hausdorff \'etale groupoid.
	If $C^*_r(G)$ is prime,
	then $G$ is topologically transitive.
	Conversely,
	if $G$ is topologically principal and topologically transitive,
	then $C^*_r(G)$ is prime.
\end{prop}

\begin{proof}
	First, assume that $C^*_r(G)$ is prime.
	Let $U\subset G^{(0)}$ be a non-empty open invariant subset.
	Suppose that $U$ is not dense in $G^{(0)}$.
	Then there exists a non-empty open subset $V\subset G^{(0)}$ such that $U\cap V=\emptyset$.
	Put $I\defeq C^*_r(G_{U})$.
	Then $I$ is a non-zero ideal of $C^*_r(G)$.
	Take a non-zero element $f\in C_c(V)$.
	Then $fI=\{0\}$ holds.
	This contradicts to the assumption that $C^*_r(G)$ is prime.
	
	Second,
	assume that $G$ is topologically principal and topologically transitive.
	Assume that $I\subset C^*_r(G)$ is a non-zero ideal.
	Then,
	by \cite[Remark 10.2.8, Lemma 10.3.1]{asims}\footnote{We may apply \cite[Theorem 10.2.7, Remark 10.2.8]{asims} since we assume that $G$ is topologically principal.
	Note that we do not need the second countability of $G$ as one can observe in the proof of \cite[Theorem 10.2.7]{asims}.},
	there exists a non-empty open invariant subset $U\subset G^{(0)}$ such that $I\cap C_0(G^{(0)})=C_0(U)$.	
	Since $G$ is  topologically transitive,
	$U\subset G^{(0)}$ is dense in $G^{(0)}$.
	Now,
	assume that $a\in C^*_r(G)$ satisfies $aI=\{0\}$.
	Since $a^*a C_0(U)=\{0\}$,
	we have $a^*a(x)=0$ for all $x\in U$. 
	Since $U$ is dense in $G^{(0)}$,
	we obtain $E(a^*a)=0$,
	where $E\colon C^*_r(G)\to C_0(G^{(0)})$ denotes the standard conditional expectation.
	Since $E$ is faithful,
	we obtain $a=0$ and this completes the proof.
	\qed
\end{proof}

\begin{prop}\label{prop: relative commutant is trivial if H is topo transitive}
	Let $G$ be a locally compact Hausdorff \'etale groupoid and $H\subset G$ be an open subgroupoid with $G^{(0)}\subset H$.
	Assume that $G$ is effective and $H$ is topologically transitive.
	Then the relative commutant $C^*_r(H)'\cap C^*_r(G)$ of $C^*_r(H)$ in $C^*_r(G)$ is trivial in the sense that
	\begin{align*}
	C^*_r(H)'\cap C^*_r(G)=
	\begin{cases}
	\C1_{C^*_r(G)} & ( \text{$G^{(0)}$ is compact}) \\
	\{0\} & (\text{otherwise})
	\end{cases}
	\end{align*}
	holds.
\end{prop}

\begin{proof}
	Take $a\in C^*_r(H)'\cap C^*_r(G)$.
	It suffices to show that $a$ is a constant function on $G^{(0)}$.
	Since $C_0(G^{(0)})$ is maximal abelian in $C^*_r(G)$ by \cite[Proposition 11.1.14]{asims},
	we have
	\[a\in C^*_r(H)'\cap C^*_r(G)\subset C_0(G^{(0)})'\cap C^*_r(G)=C_0(G^{(0)}). \]
	Suppose that $a$ is not a constant function on $G^{(0)}$.
	Then there exists disjoint non-empty open sets $U,V\subset G^{(0)}$ such that $a(U)\cap a(V)=\emptyset$.
	Since $H$ is topologically transitive,
	there exists $\alpha\in H$ such that $d(\alpha)\in U$ and $r(\alpha)\in V$.
	Take $f\in C_c(H)$ such that $f(\alpha)=1$.
	We have
	\[
	a(r(\alpha))=a(r(\alpha))f(\alpha)=af(\alpha)=f a(\alpha)=f(\alpha)a(d(\alpha))=a(d(\alpha)).
	\]
	This contradicts to $a(U)\cap a(V)=\emptyset$.
	Hence $a$ is a constant function on $G^{(0)}$ and this completes the proof.
	\qed
\end{proof}

\begin{prop}\label{prop: all we need is effectiveness}
	Let $G$ be an \'etale groupoid and $\Phi\in\Aut(G)$.
	Assume that $G$ is effective.
	If $\Phi|_{G^{(0)}}=\id_{G^{(0)}}$,
	then $\Phi=\id_G$ holds.
\end{prop}

\begin{proof}
	
	Take an open bisection $U\subset G$ arbitrarily.
	We claim that $U\Phi(U)^{-1}\subset \Iso(G)$ holds.
	Take $\delta\in U\Phi(U)^{-1}$.
	Then there exists $\alpha,\beta\in U$ such that $\delta=\alpha\Phi(\beta)^{-1}$.
	Since $\alpha$ and $\Phi(\beta)^{-1}$ are composable,
	we have
	\[
	d(\alpha)=d(\Phi(\beta))=\Phi(d(\beta))=d(\beta).
	\]
	Since $U$ is bisection,
	we have $\alpha=\beta$.
	Then we obtain $\delta=\alpha\Phi(\alpha)^{-1}$ and
	\[
	r(\delta)=r(\alpha)=\Phi(r(\alpha))=r(\Phi(\alpha))=d(\delta).
	\]
	Hence $\delta\in \Iso(G)$ and $U\Phi(U)^{-1}\subset \Iso(G)$ holds.
	Since $U\Phi(U)^{-1}$ is open and $G$ is effective,
	we obtain $U\Phi(U)^{-1}\subset G^{(0)}$.
	Therefore,
	$\alpha=\Phi(\alpha)$ holds for all $\alpha\in G$.
	\qed
\end{proof}

\begin{rem}
	In \cite[Proposition 2.2.4]{komura2023homomorphisms},
	the author proved the same assertion for a topologically principal \'etale groupoid.
	We relax the assumption from ``topologically principal'' to ``effective'' in Proposition \ref{prop: all we need is effectiveness}.
	Note that topologically principal Hausdorff \'etale groupoid is effective.
\end{rem}

\begin{defi}\label{defi: def of varphi_Phi,c}
	Let $G$ be a locally compact Hausdorff \'etale groupoid,
	$\Phi\in\Aut(G)$ and $c\in Z(G)$.
	For $f\in C_c(G)$,
	define a function $\varphi_{\Phi,c}(f)\colon G\to\C$ by
	\[
	\varphi_{\Phi,c}(f)(\alpha)\defeq c(\Phi^{-1}(\alpha))f(\Phi^{-1}(\alpha))
	\]
	for $\alpha\in G$.
\end{defi}
	
	The proof of the next Proposition is straightforward and hence omitted.
\begin{prop} \label{prop: var_Phi,c give rise to Aut}
	Let $G$ be a locally compact Hausdorff \'etale groupoid,
	$\Phi\in\Aut(G)$ and $c\in Z(G)$.
	For all $f\in C_c(G)$,
	$\varphi_{\Phi,c}(f)$ in Definition \ref{defi: def of varphi_Phi,c} belongs to $C_c(G)$.
	In addition,
	the map 
	\[\varphi_{\Phi,c}\colon C_c(G)\ni f\mapsto \varphi_{\Phi,c}(f)\in C_c(G)\]
	is a *-isomorphism on $C_c(G)$ and extended to the automorphism $\varphi_{\Phi,c}\in \Aut(C^*_r(G))$.
	Moreover,
	$\varphi_{\Phi,c}\in \Aut_{C_0(G^{(0)})}C^*_r(G)$ holds.
\end{prop}
	
	Note that $\Aut(G)$ acts on $Z(G)$ by
	\[
	\Phi.c(\alpha)\defeq c(\Phi^{-1}(\alpha)),
	\]
	where $\Phi\in \Aut(G)$, $c\in Z(G)$ and $\alpha\in G$.
	The semidirect product $\Aut(G)\ltimes Z(G)$ in the following proposition is taken with respect to this action.
	
\begin{thm}[{\cite[Proposition 5.7]{Matui2011}},{\cite[Corollary 2.2.2, 2.2.6]{komura2023homomorphisms}}] \label{thm: varphi is a group hom}
	Let $G$ be a locally compact Hausdorff \'etale groupoid.
	The map
	\[
	\Psi\colon \Aut(G)\ltimes Z(G)\ni (\Phi, c)\mapsto \varphi_{\Phi,c}\in \Aut (C^*_r(G); C_0(G^{(0)}))
	\]
	is a group homomorphism,
	where $\varphi_{\Phi,c}$ is the *-isomorphism appeared in Proposition \ref{prop: var_Phi,c give rise to Aut}.
	In addition,
	assume that $G$ is effective.
	Then $\Psi$ is surjective and therefore a group isomorphism.
	Moreover,
	\[\Psi(Z(G))=\Aut_{C_0(G^{(0)})}(C^*_r(G))\]
	holds.
\end{thm}

\begin{proof}
	In \cite[Corollary 2.2.2]{komura2023homomorphisms},
	it is shown that
	\begin{itemize}
		\item $\Psi$ is an injective group homomorphism,
		\item if $G$ is effective,
		then $\Psi$ is surjective.
	\end{itemize}
	It remains to show $\Psi(Z(G))=\Aut_{C_0(G^{(0)})}C^*_r(G)$ under the assumption that $G$ is effective.
	We show \[\Psi(Z(G))\supset \Aut_{C_0(G^{(0)})}C^*_r(G)\] since the reverse inclusion is obvious.
	Take $\varphi\in\Aut_{C_0(G^{(0)})}C^*_r(G)$.
	Then there exists $\Phi\in\Aut(G)$ and $c\in Z(G)$ such that $\varphi=\varphi_{\Phi,c}$.
	Since we have $\varphi(f)=f$ for all $f\in C_c(G^{(0)})$,
	we obtain $\Phi(x)=x$ for all $x\in G^{(0)}$.
	By Proposition \ref{prop: all we need is effectiveness},
	we obtain $\Phi=\id_G$.
	Hence, we obtain $\varphi=\varphi_{\id_G,c}\in \Psi(Z(G))$.
	This completes the proof.
	\qed
\end{proof}

\begin{rem}
	In \cite[Corollary 2.2.6]{komura2023homomorphisms},
	the author assumed that $G$ is topologically principal to show $\Psi(Z(G))=\Aut_{C_0(G^{(0)})}C^*_r(G)$.
	By Proposition \ref{prop: all we need is effectiveness},
	we may relax the assumption from ``topologically principal'' to ``effective''.
\end{rem}

\begin{rem}
	Remark that the adjoint action
	\[\ad\colon\Aut_{C_0(G^{(0)})}(C^*_r(G)) \curvearrowright \Aut (C^*_r(G); C_0(G^{(0)}))\]
	in Remark \ref{remark: conjugate action on AutBA} is conjugate to the action $\Aut(G)\curvearrowright Z(G)$ via the isomorphism in Theorem \ref{thm: varphi is a group hom} if $G$ is effective.
\end{rem}

\section{Weyl groups of groupoid C*-algebras} \label{Section: weyl groups of groupoid C*-algebras}

\subsection{Basic properties of Weyl groups}

In this section,
we introduce Weyl groups of groupoid C*-algebras.
Then we investigate the both of topological and algebraic properties of Weyl groups.

\begin{defi}\label{def: Weyl groups of groupoid C*-algebras}
	Let $G$ be a locally compact Hausdorff \'etale groupoid.
	The Weyl group $\mathfrak{W}_G$ of $G$ is defined to be
	\[
	\mathfrak{W}_G\defeq \Aut(C^*_r(G); C_0(G^{(0)}))/\Aut_{C_0(G^{(0)})}C^*_r(G).
	\]
\end{defi}

\begin{rem} \label{remark: Weyl group generalize existing one}
	We justify our definition of the Weyl groups of groupoid C*-algebras (Definition \ref{def: Weyl groups of groupoid C*-algebras}) here.
	In the context of C*-algebras,
	the study of the Weyl groups was initiated by Cuntz in \cite{Cuntz1980}.
	In \cite{Cuntz1980},
	Cuntz defined and investigated the Weyl groups for Cuntz algebras.
	Then,
	in \cite{CONTI20122529},
	the authors defined the Weyl groups for graph C*-algebras as a natural generalization of Cuntz's definition in  \cite{Cuntz1980}.
	Note that the class of graph C*-algebras includes the Cuntz algebras.	
	Now,
	we observe that our definition of the Weyl groups of groupoid C*-algebras (Definition \ref{def: Weyl groups of groupoid C*-algebras}) can be seen as a natural generalization of the Weyl groups of graph C*-algebras studied in \cite{CONTI20122529}.
	Let $E$ be a directed graph and $C^*(E)$ denotes its graph C*-algebra.
	In \cite{CONTI20122529},
	the authors defined the Weyl group $\mathfrak{W}_E$ of $C^*(E)$ as
	\[
	\mathfrak{W}_E\defeq \Aut(C^*(E);D_E)/\Aut_{D_E}C^*(E),
	\]
	where $D_E\subset C^*(E)$ denotes the diagonal commutative subalgebra of $C^*(E)$ (see \cite[Section 2.1]{CONTI20122529} for the precise definition of $C^*(E)$, $D_E$ and $\mathfrak{W}_E$).
	Besides,
	by \cite{Paterson2002}, 
	one can construct an \'etale groupoid $G_E$ so that $C^*(E)$ is isomorphic to $C^*_r(G_E)$ via the isomorphism which maps $D_E$ to $C_0(G_E^{(0)})$.
	Hence,
	we have $\mathfrak{W}_E\simeq \mathfrak{W}_{G_E}$ for all directed graph $E$ and this implies that our definition of Weyl groups (Definition \ref{def: Weyl groups of groupoid C*-algebras}) is a generalization of the existing Weyl groups.
	In Subsection \ref{subsection: Restricted Weyl group of graph algebras},
	we will explain these statements more precisely.
\end{rem}

Now,
we point out that the Weyl group $\mathfrak{W}_G$ is nothing but the automorphism group $\Aut(G)$ if $G$ is effective.

\begin{prop}\label{prop: Weyl group is Aut(G)} 
	Let $G$ be a locally compact Hausdorff \'etale groupoid.
	Assume that $G$ is effective.
	Then $\mathfrak{W}_G$ is isomorphic to $\Aut(G)$ as a group.
\end{prop}
\begin{proof}
	This is an immediate consequence of Theorem \ref{thm: varphi is a group hom}.
	\qed
\end{proof}

For a locally compact Hausdorff \'etale groupoid $G$,
we have a group homomorphism
\[
\Psi\colon \Aut(G)\ltimes Z(G)\ni (\Phi, c)\mapsto \varphi_{\Phi,c}\in \Aut(C^*_r(G), C_0(G^{(0)}))
\]
as in Theorem \ref{thm: varphi is a group hom}.
Using this homomorphism $\Psi$,
we equip $\Aut(G)\ltimes Z(G)$ with the initial topology of $\Psi$.
Namely,
$\Aut(G)\ltimes Z(G)$ is equipped with the weakest topology where $\Psi$ becomes continuous.
We investigate topological properties of $\Aut(G)\ltimes Z(G)$.
First,
we begin with the topology of $Z(G)$.

\begin{prop}\label{prop: topology of Z(G) is the compact open topology}
	Let $G$ be a locally compact Hausdorff \'etale groupoid.
	Then the relative topology of $Z(G)$ in $\Aut(G)\ltimes Z(G)$ coincides with the topology of the uniform convergence on compact sets.
\end{prop}

\begin{proof}
	In this proof,
	$\lVert\cdot \rVert_{r}$ and $\lVert\cdot\rVert_{\infty}$ denote the reduced norm and supremum norm of $C_c(G)$ respectively.
	First,
	assume that a net $\{c_{\lambda}\}_{\lambda\in\Lambda}$ converges to $c\in Z(G)$ in the uniform convergence topology on compact sets.
	We show the net $\{\varphi_{\id_G, c_{\lambda}}(f)\}_{\lambda\in\Lambda}\subset C^*_r(G)$ converges to $\varphi_{\id_G, c}(f)$ in the reduced norm of $C^*_r(G)$ for all $f\in C_c(G)$.
	We may assume $f\in C_c(U)$ for some open bisection $U\subset G$ by the partition of unity.
	Then,
	since $\varphi_{\id_G, c_{\lambda}}(f)-\varphi_{\id_G,c}(f)$ is supported on the bisection $U$,
	we have
	\[
	\lVert \varphi_{\id_G, c_{\lambda}}(f)-\varphi_{\id_G,c}(f)\rVert_{r}=\lVert \varphi_{\id_G, c_{\lambda}}(f)-\varphi_{\id_G,c}(f)\rVert_{\infty}
	\]
	by Proposition \cite[Proposition 9.2.1]{asims}.
	Since $f$ has compact support,
	the right hand side converges to $0$.
	Hence $\{\varphi_{\id_G, c_{\lambda}}(f)\}_{\lambda\in\Lambda}$ converges to $\varphi_{\id_G, c}(f)$ for all $f\in C_c(G)$ in the reduced norm.
	By the approximation argument,
	one can check that $\{\varphi_{\id_G, c_{\lambda}}(a)\}_{\lambda\in\Lambda}$ converges to $\varphi_{\id_G, c}(a)$ for all $a\in C^*_r(G)$.
	Therefore,
	$\{c_{\lambda}\}_{\lambda\in\Lambda}$ converges to $c$ in the relative topology of $Z(G)$ in $\Aut(G)\ltimes Z(G)$.
	
	Next,
	assume that a net $\{c_{\lambda}\}_{\lambda\in\Lambda}$ converges to $c\in Z(G)$ in the relative topology of $Z(G)$ in $\Aut(G)\ltimes Z(G)$.
	Take a compact set $K\subset G$.
	By Urysohn's lemma,
	there exists $f\in C_c(G)$ such that $f|_K=1$ holds.
	Then we have
	\begin{align*}
	\sup_{\alpha\in K}\lvert c_{\lambda}(\alpha)-c(\alpha)\rvert&\leq \lVert \varphi_{\id_G, c_{\lambda}}(f)-\varphi_{\id_G,c}(f)\rVert_{\infty} \\
	&\leq \lVert \varphi_{\id_G, c_{\lambda}}(f)-\varphi_{\id_G,c}(f)\rVert_{r}. 
	\end{align*}
	Since the last term converges to $0$,
	we have proved that $\{c_{\lambda}\}_{\lambda}$ uniformly converges to $c$ on any compact sets $K\subset G$.
	Therefore,
	we have finished the proof.
	\qed
\end{proof}

In \cite[Proposition 3.3]{CONTI20122529},
the authors showed that $\Aut_{D_E}C^*(E)$ is a maximal abelian subgroup in $\Aut(C^*(E))$ under some assumptions for a graph algebra $C^*(E)$.
We prove the analogue of this result for groupoid C*-algebras.
For this purpose,
we prepare some propositions and terminologies.
Recall that $Z(G)$ acts on $C^*_r(G)$ by
\[
\varphi_{\id_G,\chi}(f)(\alpha)=\chi(\alpha)f(\alpha)
\]
for $\alpha\in G$, $f\in C_c(G)$ and $\chi\in Z(G)$.

\begin{prop}\label{prop: fixed point algebra of Z(G) is C_0(G0)}
	Let $G$ be a locally compact Hausdorff \'etale groupoid.
	Assume that $G$ is effective.
	Then the fixed point algebra $C^*_r(G)^{\partial Z(G^{(0)})}$ coincides with $C_0(G^{(0)})$.
	In particular,
	$C^*_r(G)^{Z(G)}=C_0(G^{(0)})$ holds.
\end{prop}
\begin{proof}
	Since $Z(G)$ fixes $C_0(G^{(0)})$ pointwisely,
	$C_0(G^{(0)})\subset C^*_r(G)^{\partial Z(G^{(0)})}$ is obvious.
	We show $C^*_r(G)^{\partial Z(G^{(0)})}\subset C_0(G^{(0)})$.
	Assume that $a\in C^*_r(G)^{\partial Z(G^{(0)})}$ and $a\not\in C_0(G^{(0)})$.
	Then there exists $\alpha\in G\setminus G^{(0)}$ such that $a(\alpha)\not=0$.
	Since we assume that $G$ is effective,
	there exists $\alpha'\in a^{-1}(\C\setminus\{0\})$ such that $d(\alpha')\not=r(\alpha')$.
	Indeed,
	if not,
	\[\alpha\in a^{-1}(\C\setminus\{0\})\subset \Iso(G)^{\circ}=G^{(0)}\] holds and this contradicts to $\alpha\not\in G^{(0)}$.
	By Urysohn's lemma,
	there exists a continuous function $f\in C_c(G^{(0)})$ such that $f(r(\alpha'))=1$, $f(d(\alpha'))=0$ and $0\leq f \leq 1$.
	Putting $h\defeq e^{i\pi f}\in Z(G^{(0)})$ and $\chi\defeq \partial h\in Z(G)$,
	we have
	\[
	\varphi_{\id_G,\chi}(a)(\alpha')=\chi(\alpha')a(\alpha')=e^{i\pi(f(r(\alpha'))-f(d(\alpha')))}a(\alpha')=-a(\alpha').
	\]
	This contradicts to $a\in C^*_r(G)^{\partial Z(G^{(0)})}$.
	The last assertion is obvious.
	\qed
\end{proof}

\begin{defi}
	Let $G$ be a locally compact Hausdorff \'etale groupoid.
	Then $G$ is said to have enough arrows if the following property holds:
	for every nonempty open set $U\subset G^{(0)}$,
	there exists $\alpha\in G\setminus G^{(0)}$ with $d(\alpha)\in U$.
	Moreover,
	$G$ is said to have enough exits if the following property holds:
	for every nonempty open set $U\subset G^{(0)}$,
	there exists $\alpha\in G\setminus G^{(0)}$ with $d(\alpha)\in U$ and $d(\alpha)\not=r(\alpha)$.
\end{defi}

\begin{lem}\label{lem: effective and enough arrows imply enough exits}
	Let $G$ be a locally compact Hausdorff \'etale groupoid.
	Assume that $G$ is effective and has enough arrows.
	Then $G$ has enough exits.
\end{lem}

\begin{proof}
	Let $U\subset G^{(0)}$ be a nonempty open set.
	Since we assume that $G$ has enough arrows,
	there exists $\alpha\in G\setminus G^{(0)}$ with $d(\alpha)\in U$.
	Take an open bisection $W\subset G$ so that $\alpha\in W$ and $d(W)\subset U$ hold.
	If $W\subset \Iso(G)$,
	then $\alpha\in \Iso(G)^{\circ}=G^{(0)}$ and this contradicts to $\alpha\not\in G^{(0)}$.
	Therefore there exists $\beta\in W\setminus \Iso(G)$.
	Now we have $d(\beta)\in d(W)\subset U$ and $d(\beta)\not=r(\beta)$.
	\qed
\end{proof}

The following lemma is a key lemma to show that $\Aut_{C_0(G^{(0)})}(C^*_r(G))$ is a maximal abelian subgroup of $\Aut(C^*_r(G))$.

\begin{lem}\label{lem: Z(G) is maximal abelian in semidirect product}
	Let $G$ be a locally compact Hausdorff \'etale groupoid.
	Assume that $G$ is effective and has enough arrows.
	Then $Z(G)$ is a maximal abelian subgroup of $\Aut(G)\ltimes Z(G)$.
\end{lem}

\begin{proof}
	We show that the centralizer of $Z(G)$ in $\Aut(G)\ltimes Z(G)$ coincides with $Z(G)$ itself.
	Assume that $(\Phi, c)\in \Aut(G)\ltimes Z(G)$ commutes with every elements in $Z(G)$.
	It suffices to show $\Phi=\id_G$.
	Since $(\Phi,c)$ commutes with $Z(G)$,
	we have
	\[
	(\Phi, c\times \chi)=(\Phi, c)(\id_G, \chi )=(\id_G, \chi)(\Phi, c)=(\Phi, (\chi\circ \Phi)\times c)
	\]
	for all $\chi\in Z(G)$.
	Hence we have $\chi=\chi\circ \Phi$ for all $\chi\in Z(G)$.
	We observe that this condition implies $\Phi=\id_G$.
	
	Since we assume that $G$ is effective,
	it suffices to show $\Phi|_{G^{(0)}}=\id_G^{(0)}$ by Proposition \ref{prop: all we need is effectiveness}.
	Assume that there exists $x\in G^{(0)}$ with $\Phi(x)\not=x$.
	We claim that there exists $\alpha\in G$ such that $\Phi(d(\alpha))\not=d(\alpha)$ and $d(\alpha)\not= r(\alpha)$.
	Since $\Phi(x)\not=x$,
	there exists an open set $U\subset G^{(0)}$ such that $x\in U$ and $U\cap\Phi(U)=\emptyset$.
	Since $G$ has enough exits by Lemma \ref{lem: effective and enough arrows imply enough exits},
	there exists $\alpha\in U$ such that $d(\alpha)\in U$ and $d(\alpha)\not=r(\alpha)$.
	In addition, it follows $d(\alpha)\not=\Phi(d(\alpha))$ from $U\cap\Phi(U)=\emptyset$.
	
	Now, we obtain an element $\alpha\in G$ such that $d(\alpha)\not=r(\alpha)$ and $\Phi(d(\alpha))\not=d(\alpha)$.
	Since $\Phi$ is injective,
	we have $\Phi(d(\alpha))\not=\Phi(r(\alpha))$.
	In addition,
	we have $r(\alpha)\not=\Phi(d(\alpha))$ or $r(\alpha)\not=\Phi(r(\alpha))$.
	Indeed, if not,
	we obtain $\Phi(d(\alpha))=\Phi(r(\alpha))$ and this is a contradiction.
	First,
	assume that $r(\alpha)\not=\Phi(d(\alpha))$ holds.
	In this case,
	\[F\defeq \{d(\alpha), r(\alpha),\Phi(r(\alpha))\}\] and $\{\Phi(d(\alpha))\}$ are disjoint compact sets in $G^{(0)}$.
	Then,
	by Urysohn's lemma,
	there exists a continuous function $f\colon G^{(0)}\to [0,1]$ such that $f|_F=1$ and $f(\Phi(d(\alpha)))=0$.
	Put $h\defeq e^{i\pi f}\in Z(G^{(0)})$ and $\chi \defeq \partial h\in Z(G)$.
	Then we have
	\[
	\chi(\alpha)=e^{i\pi(f(r(\alpha))-f(d(\alpha)))}=1
	\]
	and
	\[
	\chi(\Phi(\alpha))=e^{i\pi(f(\Phi(r(\alpha)))-f(\Phi(d(\alpha))))}=e^{i\pi}=-1.
	\]
	This contradicts to $\chi=\chi\circ \Phi$ for all $\chi\in Z(G)$.
	Next, assume that $r(\alpha)\not=\Phi(r(\alpha))$ holds.
	Then
	\[
	F\defeq \{d(\alpha), \Phi(r(\alpha))\} \text{ and } H\defeq \{r(\alpha), \Phi(d(\alpha))\}
	\]
	are disjoint.
	By Urysohn's lemma,
	take a continuous function $f\colon G\to \R$ such that $f|_F=1/2$ and $f|_H=0$.
	Put $h\defeq e^{i\pi f}\in Z(G^{(0)})$ and $\chi\defeq \partial h\in Z(G)$.
	Then we have
	\[
	\chi(\alpha)=e^{i\pi(f(r(\alpha))-f(d(\alpha)))}=e^{-i\pi/2}=-i
	\]
	and
	\[
	\chi(\Phi(\alpha))=e^{i\pi(f(\Phi(r(\alpha)))-f(\Phi(d(\alpha))))}=e^{i\pi/2}=i.
	\]
	This also contradicts to $\chi\circ\Phi=\chi$.
	Hence, we obtain $\Phi(x)=x$ for all $x\in G^{(0)}$
	\qed
\end{proof}

Now,
we are ready to show that $\Aut_{C_0(G^{(0)})}(C^*_r(G))$ is a maximal abelian group of $\Aut(C^*_r(G))$ under the assumption that $G$ is effective and has enough arrows.

\begin{thm}\label{thm: Z(G) is maximal abelian group}
	Let $G$ be a locally compact Hausdorff \'etale groupoid.
	Assume that $G$ is effective.
	Then the centralizer of $\Aut_{C_0(G^{(0)})}(C^*_r(G))$ in $\Aut(C^*_r(G))$ is contained in $\Aut(C^*_r(G), C_0(G^{(0)}))$.
	In particular,
	if $G$ is effective and $G$ has enough arrows,
	then $\Aut_{C_0(G^{(0)})}(C^*_r(G))$ is a maximal abelian group of $\Aut(C^*_r(G))$.
\end{thm}
\begin{proof}
	Assume that $\varphi\in \Aut(C^*_r(G))$ commutes with the all elements in $\Aut_{C_0(G^{(0)})}(C^*_r(G))$.
	By Proposition \ref{prop: fixed point algebra of Z(G) is C_0(G0)} and Proposition \ref{proposition: centralizer of AutBA is Aut(A,B)},
	we obtain $\varphi\in\Aut(C^*_r(G); C_0(G^{(0)}))$.
	Now,
	the last assertion follows from Lemma \ref{lem: Z(G) is maximal abelian in semidirect product} and Theorem \ref{thm: varphi is a group hom}. 
	\qed
\end{proof}

In \cite[Proposition 3.5]{CONTI20122529},
it is proved that the Weyl groups of graph C*-algebras associated with finite directed graphs becomes countable groups.
We prove a groupoid analogue of this assertion.
First,
we prepare some terminologies.

\begin{defi}[{\cite[Definition 5.3]{NEKRASHEVYCH_2019}}]\label{def: expansive}
	Let $G$ be a locally compact Hausdorff ample \'etale groupoid.
	Then $G$ is said to be expansive if there exists a finite set $F\subset \Bis^c(G)$ such that the inverse subsemigroup generated by $F$ in $\Bis^c(G)$ forms an open basis of $G$.
\end{defi}

\begin{rem}
	In \cite[Definition 5.3]{NEKRASHEVYCH_2019},
	an \'etale groupoid $G$ is expansive if $G$ admits a finite set $F\subset\Bis^c(G)$ in Definition \ref{def: expansive} which covers a ``generating set'' (see \cite[Definition 5.3]{NEKRASHEVYCH_2019} for the precise definition).
	This definition is equivalent to Definition \ref{def: expansive} since $\bigcup F$ becomes a generating set if $F$ satisfies the condition in Definition \ref{def: expansive}.
\end{rem}

Remark that $G$ is second countable if $G$ is expansive.

\begin{ex}\label{ex: expansivity of finitely generated inverse semiroup action groupoid}
	Let $S$ be a finitely generated inverse semigroup,
	$X$ be a totally disconnected locally compact Hausdorff space and $\sigma\colon S\curvearrowright X$ be an action.
	Assume that $\dom(\sigma_e)$ is a compact open subset of $X$ for each $e\in E(S)$ and $\{\dom(\sigma_e)\}_{e\in E(S)}$ is a basis of $X$.
	Then the transformation groupoid $S\ltimes_{\sigma}X$ is expansive.
	Indeed,
	for $s\in S$,
	put
	\[
	\theta_s\defeq \{[s, x]\in S\ltimes_{\sigma}X\mid x\in \dom(\sigma_{s^*s}) \}.
	\]
	Then $\theta_s\in\Bis^c(S\ltimes_{\sigma}X)$ and the map
	\[
	S\ni s\mapsto \theta_s\in\Bis^c(S\ltimes_{\sigma}X)
	\]
	is a semigroup homomorphism (see \cite[Section 4]{ExelcombinatrialC*algebra} for details).
	In addition,
	it follows that $\{\theta_s\}_{s\in S}$ is a basis of $S\ltimes_{\sigma}X$ from the assumption that $\{\dom(\sigma_e)\}_{e\in E(S)}$ is a basis of $X$.
	Hence,
	if a finite subset $F\subset S$ is a generator of $S$,
	then the inverse semigroup generated by $\{\theta_s\}_{s\in F}$ in $\Bis^c(S\ltimes_{\sigma}X)$ coincides with
	\[
	\{\theta_s\in\Bis^c(S\ltimes_{\sigma}X)\mid s\in S\}
	\]
	and therefore a basis of $S\ltimes_{\sigma}X$.
	In particular,
	graph groupoids associated with finite directed graphs in \cite{Paterson2002} are expansive if the underlying graph has no sink.
	For example,
	the standard groupoid model of the Cuntz algebra $\mathcal{O}_n$ is expansive (see also Subsection \ref{subsection : Examples of coactions on Cuntz algebras and graph algebras}).
\end{ex}

\begin{prop}\label{prop embed AutG to Aut BisG}
	Let $G$ be a locally compact Hausdorff \'etale groupoid.
	For $\Phi\in\Aut(G)$,
	define $\widetilde{\Phi}\in\Aut(\Bis(G))$ by $\widetilde{\Phi}(U)\defeq \Phi(U)$ for $U\in\Bis(G)$.
	Then the map
	\[
	\iota\colon \Aut(G)\ni \Phi\mapsto \widetilde{\Phi}\in \Aut(\Bis(G))
	\]
	is an injective group homomorphism.
	In addition,
	if $G$ is ample,
	\[\iota'\colon \Aut(G)\ni \Phi\mapsto \widetilde{\Phi}|_{\Bis^c(G)}\in \Aut(\Bis^c(G))\]
	is also an injective group homomorphism.
\end{prop}
\begin{proof}
	It is straightforward to check that $\iota$ and $\iota'$ are well-defined and group homomorphisms.
	It follows that $\iota$ is injective from the fact that $\Bis(G)$ is a basis of $G$.
	Indeed,
	assume that $\Phi\not=\Psi$ holds for some $\Phi,\Psi\in\Aut(G)$.
	Then there exists $\alpha\in G$ such that $\Phi(\alpha)\not=\Psi(\alpha)$.
	Since $G$ is Hausdorff and $\Bis(G)$ is a basis of $G$,
	there exists $V,W\in \Bis(G)$ such that $\Phi(\alpha)\in V,
	\Psi(\alpha)\in W$ and $V\cap W=\emptyset$.
	Putting $U\defeq \Phi^{-1}(V)\cap \Psi^{-1}(W)$,
	we have $U\in \Bis(G)$ and $\Phi(U)\cap \Psi(U)=\emptyset$.
	In particular,
	since $U$,
	containing $\alpha$, is not empty,
	we obtain $\widetilde{\Phi}\not=\widetilde{\Psi}$.
	Hence $\iota$ is injective.
	Similarly,
	if $G$ is ample,
	one can check that $\iota'$ is injective from the fact that $\Bis^c(G)$ is a basis of $G$.
	\qed

\end{proof}

Now,
we are ready to show that Weyl groups become discrete countable groups under some finiteness conditions.
This result is an analogue of \cite[Proposition 3.5]{CONTI20122529},
which asserts the Weyl groups of graph C*-algebras associated with finite directed graphs becomes countable groups.
We also remark that we cannot completely recover \cite[Proposition 3.5]{CONTI20122529} from Theorem \ref{thm: Weyl group is discrete countable if G is effective and expansive} as explained in Remark \ref{remark: we cannot recover discreteness of weyl group of graph algebra}.

\begin{thm}\label{thm: Weyl group is discrete countable if G is effective and expansive}
	Let $G$ be a locally compact Hausdorff ample \'etale groupoid.
	Assume that $G$ is effective and expansive.
	Then the Weyl group $\mathfrak{W}_G$ is a countable discrete group.
\end{thm}

\begin{proof}
	Take a finite set $F\subset \Bis^c(G)$ such that the inverse semigroup generated by $F$ is a basis of $G$.
	Then one can see that the map
	\[
	\Aut(\Bis^c(G))\ni \Psi\mapsto (\Psi(U))_{U\in F}\in \Bis^c(G)^{\lvert F\rvert}
	\]
	is an injective map.
	Since $G$ has a countable basis and each elements in $\Bis^c(G)$ are compact,
	$\Bis^c(G)^{\lvert F\rvert}$ is a countable set.
	Hence $\Aut(\Bis^c(G))$ is also countable set and $\mathfrak{W}_G$ is countable by Proposition \ref{prop: Weyl group is Aut(G)} and Proposition \ref{prop embed AutG to Aut BisG}.
	
	Next, we show that $\mathfrak{W}_G=\Aut(G)$ is discrete.
	We let $\lVert\cdot\rVert_r$ and $\lVert\cdot\rVert_{\infty}$ denote the reduced and supremum norm respectively.
	Take $\Phi\in\Aut(G)$ and define
	\[
	W\defeq \bigcap_{U\in F}\{\Psi\in \Aut(G)\mid \lVert \varphi_{\Phi}(\chi_U)-\varphi_{\Psi}(\chi_U)\rVert_r<1/2\},
	\]
	where $\chi_U\in C_c(G)$ denotes the characteristic function on $U$ and $\varphi_{\Phi}\in\Aut(C^*_r(G))$ denotes the *-automorphism defined by
	\[
	\varphi_{\Phi}(f)(\alpha)=f(\Phi^{-1}(\alpha))
	\]
	for $\alpha\in G$ and $f\in C_c(G)$.
	Then $W$ is an open set of $\Aut(G)$.
	We show that $W=\{\Phi\}$.
	Take $\Psi\in W$ and $U\in F$.
	Then we have
	\begin{align*}
	1/2>\lVert \varphi_{\Phi}(\chi_U)-\varphi_{\Psi}(\chi_U)\rVert_r&=\lVert \chi_{\Phi(U)}-\chi_{\Psi(U)}\rVert_r \\
	&\geq \lVert \chi_{\Phi(U)}-\chi_{\Psi(U)}\rVert_{\infty},
	\end{align*}
	where $\lVert \cdot\rVert_{\infty}$ denotes the supremum norm.
	Hence,
	we obtain
	\[
	\lVert \chi_{\Phi(U)}-\chi_{\Psi(U)}\rVert_{\infty}=0
	\]
	and therefore $\Phi(U)=\Psi(U)$ for all $U\in F$.
	Now,
	one can see that $\Phi(U)=\Psi(U)$ holds for all $U\in\Bis^c(G)$ since the inverse semigroup generated by $F$ is a basis of $\Bis^c(G)$.
	Thus,
	we obtain $\Phi=\Psi$ by Proposition \ref{prop embed AutG to Aut BisG}.
	Therefore $\mathfrak{W}_G=\Aut(G)$ is discrete and this completes the proof.
	\qed

\end{proof}

\begin{rem}\label{remark: we cannot recover discreteness of weyl group of graph algebra}
	We remark that we cannot completely recover \cite[Proposition 3.5]{CONTI20122529} from Theorem \ref{thm: Weyl group is discrete countable if G is effective and expansive}.
	This is because we assume that $G$ is effective and expansive in Theorem \ref{thm: Weyl group is discrete countable if G is effective and expansive} while \cite[Proposition 3.5]{CONTI20122529} does not require corresponding assumptions.
\end{rem}

\section{Restricted Weyl groups of groupoid C*-algebras}\label{Section: Restricted Weyl groups of groupoid C*-algebras}

In the previous section,
we investigated the groups of automorphisms on $C^*_r(G)$ which fix $C_0(G^{(0)})$.
In this section,
we study the groups of automorphisms on $C^*_r(G)$ which fix other subalgebras in $C^*_r(G)$.
The study of such groups,
known as restricted Weyl groups,
was initiated by Cuntz for the Cuntz algebras in \cite{Cuntz1980}.
Since then,
the restricted Weyl groups have been widely studied.
For example,
in \cite{ContiHongSzymaski},
the authors revealed many properties of the restricted Weyl groups of the Cuntz algebras.
In \cite{CONTI20122529}
the authors proposed and studied the restricted Weyl groups of graph algebras.
In this section,
we aim to formulate and investigate the restricted Weyl groups of groupoid C*-algebras.
Precisely,
we investigate the groups of automorphisms on groupoid C*-algebras $C^*_r(G)$ which fix C*-subalgebras $C_0(G^{(0)})$ and $C^*_r(H)\subset C^*_r(G)$ arising from open subgroupoids $G^{(0)}\subset H\subset G$.
Our purpose in this section is to characterize such automorphisms on $C^*_r(G)$ in terms of the underlying \'etale groupoids $H\subset G$.

\subsection{Restricted Weyl group (general case)}

Let $G$ be a locally compact Hausdorff \'etale effective groupoid and $H\subset G$ be an open subgroupoid with $G^{(0)}\subset H$.
Then we have the natural inclusion $C_0(G^{(0)})\subset C^*_r(H)\subset C^*_r(G)$ by \cite[Lemma 3.2]{BrownExelFuller2021}.
In this subsection,
we investigate automorphisms which fix $C^*_r(H)$.
More precisely,
we investigate the following groups
\[
\Aut(C^*_r(G); C^*_r(H), C_0(G^{(0)})) \text{ and }\Aut_{C^*_r(H)}C^*_r(G)
\]
in this subsection (see Definition \ref{defi: definitions of automorphism groups} and  for the definitions of these automorphism groups).
Then we investigate the restricted Weyl group $\mathfrak{RW}_{G,H}$.
Following \cite{ContiHongSzymaski} and \cite{CONTI20122529},
we define $\mathfrak{RW}_{G,H}$ for an inclusion of \'etale groupoids.

\begin{defi}\label{definition: restricted weyl group}
	Let $G$ be a locally compact Hausdorff \'etale effective groupoid and $H\subset G$ be an open subgroupoid with $G^{(0)}\subset H$.
	We define the restricted Weyl group $\mathfrak{RW}_{G,H}$ of the inclusion $H\subset G$ as
	\[
	\mathfrak{RW}_{G,H}\defeq \Aut(C^*_r(G); C^*_r(H),C_0(G^{(0)}))/\Aut_{C_0(G^{(0)})}C^*_r(G).
	\]
\end{defi}

\begin{rem}\label{remark: restricted Weyl is generalization of existing one}
	
	First,
	we remark that 
	\[
	\Aut_{C_0(G^{(0)})}C^*_r(G)\subset \Aut(C^*_r(G); C^*_r(H),C_0(G^{(0)}))
	\]
	holds if $G$ is effective.
	One can check this inclusion by Theorem \ref{thm: varphi is a group hom}.
	Hence,
	we may take the quotient group
	\[
	\mathfrak{RW}_{G,H}\defeq \Aut(C^*_r(G); C^*_r(H),C_0(G^{(0)}))/\Aut_{C_0(G^{(0)})}C^*_r(G).
	\]
	
	As in Remark \ref{remark: Weyl group generalize existing one},
	we observe that our definition of the restricted Weyl groups of groupoid C*-algebras (Definition \ref{definition: restricted weyl group}) can be seen as a natural generalization of existing restricted Weyl groups.
	In \cite{CONTI20122529},
	the authors defined the restricted Weyl groups for graph algebras.
	We briefly recall the definition of the restricted Weyl groups for graph algebras here.
	Let $E$ be a directed graph and $C^*(E)$ denotes its graph C*-algebra.
	In addition,
	let $C^*(E)^{\T}$ denote the fixed point subalgebra of the gauge action.
	Assume that $E$ has no sinks and all cycles have exits.
	In \cite{CONTI20122529},
	the authors defined the Weyl group $\mathfrak{RW}_E$ of $C^*(E)$ as
	\[
	\mathfrak{RW}_E\defeq \Aut(C^*(E);C^*(E)^{\T}, D_E)/\Aut_{D_E}C^*(E),
	\]
	where $D_E\subset C^*(E)$ denotes the diagonal commutative subalgebra of $C^*(E)$ (see \cite[Section 2.1]{CONTI20122529} for the precise definition of $C^*(E)$, $D_E$ and $\mathfrak{RW}_E$).
	Note that we may take this quotient group since $\Aut_{D_E}C^*(E)$ is contained in $\Aut(C^*(E);C^*(E)^{\T}, D_E)$ by \cite[Proposition 3.2]{CONTI20122529} if $E$ has no sinks and all cycles have exits.
	Besides,
	by \cite{Paterson2002}, 
	one can construct a locally compact Hausdorff \'etale groupoid $G_E$ and its open subgroupoid $H\subset G_E$ so that $C^*(E)$ is isomorphic to $C^*_r(G_E)$ via the isomorphism which maps $D_E$ and $C^*(E)^{\T}$ to $C_0(G_E^{(0)})$ and $C^*_r(H)$ respectively.
	In addition,
	this $G_E$ is effective by \cite[Proposition 2.3]{BROWNLOWE_CARLSEN_WHITTAKER_2017} (see also Proposition \ref{prop: characterization that graph groupoid is topologically principal}).
	Hence,
	we have $\mathfrak{RW}_E\simeq \mathfrak{RW}_{G_E,H}$ and this implies that our definition of restricted Weyl groups (Definition \ref{def: Weyl groups of groupoid C*-algebras}) is a generalization of the existing restricted Weyl groups.
	In Subsection \ref{subsection: Restricted Weyl group of graph algebras},
	we will explain these statements more precisely.
	
\end{rem}

\begin{defi}
	Let $G$ a locally compact Hausdorff \'etale groupoid and $H\subset G$ be an open subgroupoid with $G^{(0)}\subset H$.
	We define a subgroup $\Aut(G;H)$ of $\Aut(G)$ as
	\[
	\Aut(G;H)\defeq \{\Phi\in \Aut(G)\mid \Phi(H)=H\}.
	\]
	In addition,
	we define a subgroup $Z_H(G)$ of $Z(G)$ as
	\[
	Z_H(G)\defeq \{c\in Z(G)\mid c|_H=1\}.
	\]
\end{defi}

As the Weyl groups are isomorphic to $\Aut(G)$ under some assumptions by Proposition \ref{prop: Weyl group is Aut(G)},
the restricted Weyl groups are isomorphic to $\Aut(G;H)$ if $G$ is effective.

\begin{prop}\label{prop: restricted Weyl group is restricted Weyl groupoid automorphisms}
	Let $G$ be a locally compact Hausdorff \'etale groupoid and $H\subset G$ be an open subgroupoid with $G^{(0)}\subset H$.
	Assume that $G$ is effective.
	Then the restriction of the isomorphism 
	\[
	\Aut(G)\ltimes Z(G)\ni (\Phi, c)\mapsto \varphi_{\Phi,c}\in \Aut(C^*_r(G); C_0(G^{(0)}))
	\]
	in Theorem \ref{thm: varphi is a group hom} induces an isomorphism
	\[
	\Aut(G;H)\ltimes Z(G)\simeq \Aut(C^*_r(G); C_0(G^{(0)}), C^*_r(H)).
	\]
\end{prop}
\begin{proof}
	It is straightforward to check that \[\varphi_{\Phi, c}\in \Aut(C^*_r(G); C^*_r(H), C_0(G^{(0)}))\] holds for $\Phi\in\Aut(G;H)$ and $c\in Z(G)$.
	Now, it suffices to show that the map
	\[
	\Aut(G,H)\ltimes Z(G)\ni (\Phi, c)\mapsto \varphi_{\Phi, c}\in	\Aut(C^*_r(G); C^*_r(H), C_0(G^{(0)}))
	\]
	is surjective.
	Take $\varphi\in \Aut(C^*_r(G); C^*_r(H), C_0(G^{(0)}))$.
	Then there exists $\Phi\in \Aut(G)$ and $c\in Z(G)$ such that $\varphi=\varphi_{\Phi,c}$ by Theorem \ref{thm: varphi is a group hom}.
	First, we show $\Phi(H)\subset H$.
	Take $\alpha\in H$.
	By Urysohn's lemma,
	there exists $f\in C_c(H)$ such that $f(\alpha)\not=0$.
	Since we have $\varphi_{\Phi,c}(f)\in C_c(H)$ and
	\[
	\varphi_{\Phi,c}(f)(\Phi(\alpha))=c(\alpha)f(\alpha)\not=0,
	\]
	it follows $\Phi(\alpha)\in H$.
	Hence we obtain $\Phi(H)\subset H$.
	Since $\varphi_{\Phi,c}^{-1}=\varphi_{\Phi^{-1},\overline{\Phi.c}}$ is contained in $\Aut(C^*_r(G);C^*_r(H), C_0(G^{(0)}))$,
	the same argument implies $\Phi^{-1}(H)\subset H$.
	Therefore we obtain $\Phi(H)=H$ and this completes the proof.
	\qed
\end{proof}

\begin{cor}\label{cor Restricted Weyl group is restricted groupoid automorphisms}
Let $G$ be a locally compact Hausdorff \'etale groupoid and $H\subset G$ be an open subgroupoid with $G^{(0)}\subset H$.
Assume that $G$ is effective.
Then the restricted Weyl group $\mathfrak{RW}_{G,H}$ is isomorphic to $\Aut(G;H)$ as a topological group.
\end{cor}

\begin{proof}
	Recall that the isomorphism
	\[
	\Aut(G;H)\ltimes Z(G)\simeq \Aut(C^*_r(G); C_0(G^{(0)}), C^*_r(H))
	\]
	in Proposition \ref{prop: restricted Weyl group is restricted Weyl groupoid automorphisms} maps $Z(G)$ to $\Aut_{C_0(G^{(0)})}C^*_r(G)$.
	Now the assertion follows from the definition of the restricted Weyl group
	\[
	\mathfrak{RW}_{G,H}\defeq \Aut(C^*_r(G); C^*_r(H),C_0(G^{(0)}))/\Aut_{C_0(G^{(0)})}C^*_r(G).
	\]
	in Definition \ref{definition: restricted weyl group}.
	\qed
\end{proof}

\begin{prop}\label{proposition: automorphisms which fix C*r(H)}
	Let $G$ be a locally compact Hausdorff \'etale groupoid and $H\subset G$ be an open subgroupoid with $G^{(0)}\subset H$.
	Assume that $G$ is effective.
	Then the restriction of the isomorphism 
	\[
	Z(G)\ni c\mapsto \varphi_{\id_G, c}\in \Aut_{C_0(G^{(0)})}C^*_r(G)
	\]
	induces an isomorphism 
	\[
	Z_H(G)\simeq \Aut_{C^*_r(H)}C^*_r(G).
	\]
\end{prop}

\begin{proof}
	It follows that $c\in Z_H(G)$ implies $\varphi_{\id_G,c}\in\Aut_{C^*_r(H)}C^*_r(G)$ from straightforward calculations.
	It suffices to show that the map
	\[
	Z_H(G)\ni c\mapsto \varphi_{\id_G, c}\in \Aut_{C^*_r(H)}C^*_r(G)
	\]
	is surjective.
	Take $\varphi\in \Aut_{C^*_r(H)}C^*_r(G)$.
	Then there exists $c\in Z(G)$ such that $\varphi=\varphi_{\id_G,c}$ by Theorem \ref{thm: varphi is a group hom}.
	To show $c\in Z_H(G)$,
	take $\alpha\in H$.
	By Urysohn's lemma,
	there exists $f\in C_c(H)$ such that $f(\alpha)=1$.
	Since $f\in C_c(H)$ and $\varphi\in\Aut_{C^*_r(H)}C^*_r(G)$,
	we have
	\[
	1=f(\alpha)=\varphi_{\id_G, c}(f)(\alpha)=c(\alpha)f(\alpha)=c(\alpha).
	\]
	Hence we obtain $c\in Z_H(G)$.
	\qed
\end{proof}

\subsection{Restricted Weyl group (discrete group cocycle kernel case)}

In the previous subsection,
we considered an inclusion of \'etale groupoids $H\subset G$.
In this subsection,
we consider the kernel of a discrete group cocycle $\sigma\colon G\to \Gamma$ as a subgroupoid $H$.
Our purpose is to describe the groups $Z_{\ker\sigma}(G)$ and $\Aut(G;\ker\sigma)$ in terms of the underlying groupoid $G$ and cocycle $\sigma$.

First,
we prepare a key lemma in this subsection.

\begin{lem}\label{lemma: key lemma when groupoid hom factors discrete group cocycle }
	Let $G$ be a locally compact Hausdorff \'etale groupoid,
	$K$ be a Hausdorff topological group,
	$\Gamma$ be a discrete group and $\sigma\colon G\to \Gamma$ be a surjective continuous cocycle.
	Assume that a continuous groupoid homomorphism $\Phi\colon G\to K$ satisfies $\ker\sigma\subset \ker\Phi$ and $\ker \sigma$ is topologically transitive.
	Then there exists a group homomorphism $\tau\colon\Gamma\to K$ such that $\Phi=\tau\circ\sigma$.
\end{lem}
\begin{proof}
	First,
	we show that $\Phi|_{\sigma^{-1}(\{s\})}$ is a constant map for all $s\in \Gamma$.
	Take $\alpha,\beta\in \sigma^{-1}(\{s\})$ and suppose $\Phi(\alpha)\not=\Phi(\beta)$.
	Then,
	by the continuity of $\Phi$ and the assumption that $K$ is Hausdorff,
	there exist disjoint open sets $U,V\subset G$ such that $\alpha\in U$, $\beta\in V$ and $\Phi(U)\cap \Phi(V)=\emptyset$ hold.
	Since $\ker\sigma$ is topologically transitive,
	there exists $\gamma\in\ker\sigma$ such that
	\[
	r(\gamma)\in d(U\cap \sigma^{-1}(\{s\})) \text{ and } d(\gamma)\in d(V\cap \sigma^{-1}(\{s\})).
	\]
	Remark that open sets $d(U\cap \sigma^{-1}(\{s\}))$ and $d(V\cap \sigma^{-1}(\{s\}))$ are non-empty since they contain $d(\alpha)$ and $d(\beta)$ respectively.
	Then there exists $\alpha'\in U\cap\sigma^{-1}(\{s\})$ and $\beta'\in V\cap\sigma^{-1}(\{s\})$ such that $r(\gamma)=d(\alpha')$ and $d(\gamma)=d(\beta')$.
	Now,
	$\alpha',\gamma,\beta'^{-1}$ are composable and we have
	\[
	\sigma(\alpha'\gamma\beta'^{-1})=ses^{-1}=e.
	\]
	Since $\ker \sigma\subset \ker\Phi$,
	we obtain $\Phi(\alpha'\gamma\beta'^{-1})=1=\Phi(\gamma)$ and therefore $\Phi(\alpha')=\Phi(\beta')$.
	This contradicts to $\Phi(U)\cap \Phi(V)=\emptyset$.
	Hence,
	$\Phi|_{\sigma^{-1}(\{s\})}$ is constant for all $s\in\Gamma$.
	
	Now,
	by the previous argument,
	the map
	\[
	\tau\colon \sigma(G)\ni \sigma(\alpha)\mapsto \Phi(\alpha)\in K
	\]
	is actually well-defined.
	Since $\tau\circ\sigma=\Phi$ holds and we assume that $\sigma(G)=\Gamma$,
	$\tau\colon\Gamma\to K$ is a group homomorphism such that $\Phi=\tau\circ\sigma$.
	This completes the proof.
	\qed
	
\end{proof}

\begin{rem}\label{remark: remark about surjectiveness of cocycle}
	We give a remark about the assumption that $\sigma$ is surjective.
	In the above situation,
	since $\ker\sigma$ is topologically transitive,
	$\sigma(G)$ is a subgroup of $\Gamma$ by \cite[Lemma 2.2.2]{Komura+2024}.
	Hence,
	even if $\sigma$ is not necessarily surjective,
	we can construct a group homomorphism $\tau\colon \sigma(G)\to K$ so that $\Phi=\tau\circ\sigma$ holds.
	In this sense,
	replacing $\Gamma$ with $\sigma(G)$,
	we may assume that $\sigma$ is surjective without loss of generality.
\end{rem}

Let $G$ be a locally compact Hausdorff \'etale groupoid,
$\Gamma$ be a discrete group and $\sigma\colon G\to\Gamma$ be a continuous cocycle.
Now,
we analyse $Z_{\ker\sigma}(G)$ here.
For a discrete abelian group $\Lambda$,
we denote the Pontryagin dual of $\Lambda$ by $\widehat{\Lambda}$.
Namely,
we put \[\widehat{\Lambda}\defeq \{\chi\colon\Lambda\to\T\mid \text{ $\chi$ is a group homomorphism.}\}.\]
Note that $\widehat{\Lambda}$ is a compact abelian group with respect to the pointwise product and pointwise convergence topology.
Let $\Gamma^{\ab}\defeq\Gamma/[\Gamma,\Gamma]$ denote the abelianization of $\Gamma$ and $q\colon\Gamma\to\Gamma^{\ab}$ be the quotient map.
For $\chi\in\widehat{\Gamma^{\ab}}$,
one can see that $\chi\circ q\circ \sigma$ belongs to $Z_{\ker\sigma}(G)$ and the map 
\[\Psi\colon \widehat{\Gamma^{\ab}}\ni \chi\mapsto \chi\circ q\circ \sigma \in Z_{\ker\sigma}(G)\]
is a group homomorphism.
The following proposition asserts that $\Psi$ is an isomorphism as a topological group if $\sigma$ is surjective and $\ker\sigma$ is topologically transitive.

\begin{prop}\label{proposition: cocycle which =1 on ker factors abelianization of group}
	Let $G$ be a locally compact Hausdorff \'etale groupoid,
	$\Gamma$ be a discrete group and $\sigma\colon G\to\Gamma$ be a surjective continuous cocycle.
	Assume that $\ker\sigma\subset G$ is topologically transitive.
	Then
	\[
	\Psi\colon \widehat{\Gamma^{\ab}}\ni \chi\mapsto \chi\circ q\circ \sigma \in Z_{\ker\sigma}(G)
	\]
	is an isomorphism as a topological group,
	where $\widehat{\Gamma^{\ab}}$ denotes the Pontryagin dual of the abelianization $\Gamma^{\ab}\defeq \Gamma/[\Gamma,\Gamma]$ and $q\colon\Gamma\to \Gamma^{\ab}$ denotes the quotient map.
	In particular,
	\[
	Z_{\ker\sigma^{\ab}}(G)=Z_{\ker\sigma}(G)
	\]
	holds,
	where $\sigma^{\ab}\defeq q\circ \sigma$.
\end{prop}

\begin{proof}
	Since $q\circ\sigma$ is surjective,
	one can check that $\Psi$ is injective.
	We show that $\Psi$ is continuous.
	Note that $Z_{\ker\sigma}(G)$ is equipped with the topology of uniform convergence on compact sets by Proposition \ref{prop: topology of Z(G) is the compact open topology}.
	Take $\varepsilon>0$ and a compact set $K\subset G$ arbitrarily.
	Assume that a net $\{\chi_{\lambda}\}_{\lambda\in\Lambda}\subset \widehat{\Gamma^{\ab}}$ converges to $\chi\in\widehat{\Gamma^{\ab}}$.
	Since $K$ is compact,
	there exists a finite set $F\subset \Gamma$ such that $K\subset \bigcup_{s\in F}\sigma^{-1}(s)$ holds.
	Since $\{\chi_{\lambda}\}_{\lambda\in\Lambda}$ converges to $\chi$ pointwisely,
	there exists $\lambda_0\in\Lambda$ such that
	\[\lvert \chi_{\lambda}(q(s))-\chi(q(s))\rvert<\varepsilon \]
	holds for all $\lambda\geq \lambda_0$ and $s\in F$.
	Now one can see that
	\[
	\sup_{\alpha\in K}\lvert \chi_{\lambda}\circ q\circ \sigma(\alpha)-\chi\circ q\circ \sigma(\alpha)\rvert< \varepsilon.
	\]
	Hence,
	$\{\chi_{\lambda}\circ q \circ \sigma\}_{\lambda\in\Lambda}$ converges to $\chi\circ q\circ \sigma$ and $\Psi$ is continuous.
	
	Next,
	we show that $\Psi$ is surjective.
	Take $c\in Z_{\ker\sigma}(G)$.
	Applying Lemma \ref{lemma: key lemma when groupoid hom factors discrete group cocycle } to $c\colon G\to \T$,
	we obtain a group homomorphism $\chi\colon\Gamma\to \T$ such that $c=\chi\circ\sigma$.
	By the universality of $\Gamma^{\ab}$,
	there exists $\chi'\in\widehat{\Gamma^{\ab}}$ such that $\chi=\chi'\circ q$.
	Hence,
	we obtain $\chi'\in\widehat{\Gamma^{\ab}}$ such that $c=\chi'\circ q\circ \sigma$ and have proved that $\Psi$ is surjective.
	Since $\widehat{\Gamma^{\ab}}$ is compact and $Z_{\ker\sigma}(G)$ is Hausdorff,
	$\Psi$ becomes an isomorphism as a topological group.
	Now, the last assertion
	\[
	Z_{\ker\sigma}(G)=Z_{\ker\sigma^{\ab}}(G)=\{\chi\circ q\circ \sigma \mid \chi\in \widehat{\Gamma^{\ab}}\}.
	\]
	is obvious.
	This completes the proof.
	\qed
\end{proof}

\begin{cor}\label{cor: Galois group is dual of Gamma ab}
	Let $G$ be a locally compact Hausdorff groupoid,
	$\Gamma$ be a discrete group and $\sigma\colon G\to \Gamma$ be a surjective continuous cocycle.
	Assume that $G$ is effective and $\ker\sigma$ is topologically transitive.
	Then $\Aut_{C^*_r(\ker\sigma)}C^*_r(G)$ is isomorphic to $\widehat{\Gamma^{\ab}}$ as a topological group.
	In particular,
	\[
	\Aut_{C^*_r(\ker\sigma)}C^*_r(G)=\Aut_{C^*_r(\ker\sigma^{\ab})}C^*_r(G) (\simeq \widehat{\Gamma^{\ab}})
	\]
	holds,
	where $\sigma^{\ab}\defeq q\circ \sigma$ denotes the composition map of $\sigma$ and the quotient map $q\colon \Gamma\to\Gamma^{\ab}$.
\end{cor}

\begin{proof}
	Just combine Proposition \ref{proposition: automorphisms which fix C*r(H)} and Proposition \ref{proposition: cocycle which =1 on ker factors abelianization of group}.
	\qed
\end{proof}

Next,
we investigate $\Aut(G;\ker\sigma)$.
We begin with the study of general groupoid homomorphisms rather than automorphisms.

\begin{prop}\label{prop; ker preserving groupoid homomorphism induce group hom}
	Let $G_i$ be locally compact Hausdorff \'etale groupoids,
	$\Gamma_i$ be discrete groups and $\sigma_i\colon G_i\to \Gamma_i$ be continuous cocycles for $i=1,2$.
	Assume that $\ker \sigma_1$ is topologically transitive and $\sigma_1$ is surjective.
	Then,
	if a continuous groupoid homomorphism $\Phi\colon G_1\to G_2$ satisfies $\Phi(\ker \sigma_1)\subset \ker \sigma_2$,
	there exists a unique group homomorphism $\tau\colon \Gamma_1\to \Gamma_2$ such that $\tau\circ \sigma_1=\sigma_2\circ \Phi$.
\end{prop}

\begin{proof}
	The uniqueness of $\tau$ is obvious since $\sigma_1$ is surjective.
	Using $\Phi(\ker \sigma_1)\subset \ker \sigma_2$,
	one can check $\ker\sigma_1\subset \ker(\sigma_2\circ\Phi)$.
	Then we may apply Lemma \ref{lemma: key lemma when groupoid hom factors discrete group cocycle } to $\sigma_2\circ\Phi$ and we obtain a group homomorphism $\tau\colon\Gamma_1\to \Gamma_2$ such that $\tau\circ \sigma_1=\sigma_2\circ \Phi$ holds.
	\qed
\end{proof}

\begin{rem}
	Even if $\sigma_1$ is not surjective in the above situation,
	we can construct a group homomorphism $\tau\colon \sigma_1(G_1)\to \Gamma_2$ so that $\tau\circ \sigma_1=\sigma_2\circ \Phi$ holds for the same reason as Remark \ref{remark: remark about surjectiveness of cocycle}.
\end{rem}

Applying Proposition \ref{prop; ker preserving groupoid homomorphism induce group hom} to a groupoid isomorphism,
one may obtain a group isomorphism as the following.

\begin{cor}\label{cor: ker preserving groupoid isom is groupoid equivalent isom}
	Let $G_i$ be locally compact Hausdorff \'etale groupoids,
	$\Gamma_i$ be discrete groups and $\sigma_i\colon G_i\to \Gamma_i$ be surjective continuous cocycles for $i=1,2$.
	Assume that $\ker \sigma_i$ are topologically transitive for $i=1,2$.
	Then,
	if a continuous groupoid isomorphism $\Phi\colon G_1\to G_2$ satisfies $\Phi(\ker \sigma_1)=\ker \sigma_2$,
	there exists a unique group isomorphism $\tau\colon \Gamma_1\to \Gamma_2$ such that $\tau\circ \sigma_1=\sigma_2\circ \Phi$.
\end{cor}

\begin{proof}
	Applying Proposition \ref{prop; ker preserving groupoid homomorphism induce group hom} to $\Phi$ and $\Phi^{-1}$,
	we obtain group homomorphisms $\tau\colon \Gamma_1\to \Gamma_2$ and $\tau'\colon \Gamma_2\to \Gamma_1$ such that $\tau\circ \sigma_1=\sigma_2\circ\Phi$ and $\tau'\circ \sigma_2=\sigma_1\circ\Phi^{-1}$ hold.
	Then we have
	\[
	\tau'\circ\tau\circ\sigma_1=\tau'\circ\sigma_2\circ\Phi=\sigma_1\circ\Phi^{-1}\circ\Phi=\sigma_1.
	\]
	Since $\sigma_1$ is surjective,
	we obtain $\tau'\circ\tau=\id_{\Gamma_1}$.
	In the same way,
	we obtain $\tau\circ\tau'=\id_{\Gamma_2}$ and therefore $\tau\in\Aut\Gamma$.
	\qed
\end{proof}

We rewrite Corollary \ref{cor: ker preserving groupoid isom is groupoid equivalent isom} for a groupoid automorphism case.

\begin{cor}\label{cor: Aut(G;kersigma) is equivalent groupoid automorphisms}
	Let $G$ be a locally compact Hausdorff \'etale groupoid,
	$\Gamma$ be a discrete group and $\sigma\colon G\to\Gamma$ be a surjective continuous cocycle.
	Assume that $\ker \sigma$ is topologically transitive.
	Then
	\[
	\Aut(G;\ker \sigma)=\{\Phi\in\Aut(G)\mid \text{$\tau\circ \sigma=\sigma\circ \Phi$ for some $\tau\in\Aut\Gamma$}\}
	\]
	holds.
\end{cor}
\begin{proof}
	Apply Corollary \ref{cor: ker preserving groupoid isom is groupoid equivalent isom} as $G_1=G_2$ and $\Gamma_1=\Gamma_2$.
	\qed
\end{proof}

\begin{cor}\label{cor: restricted Weyl isomorphic to restricted groupoid Weyl}
	Let $G$ be a locally compact Hausdorff \'etale groupoid,
	$\Gamma$ be a discrete group and $\sigma\colon G\to\Gamma$ be a surjective continuous cocycle.
	Assume that $G$ is effective and $\ker \sigma$ is topologically transitive.
	Then the restricted Weyl group $\mathfrak{RW}_{G,\ker\sigma}$ is isomorphic to
	\[
	\{\Phi\in\Aut(G)\mid \text{$\tau\circ \sigma=\sigma\circ \Phi$ for some $\tau\in\Aut\Gamma$}\}.
	\]
	as a topological group.
\end{cor}
\begin{proof}
	Combine Corollary \ref{cor Restricted Weyl group is restricted groupoid automorphisms} and Corollary \ref{cor: Aut(G;kersigma) is equivalent groupoid automorphisms}.
	\qed
\end{proof}

\subsection{Compact abelian group actions}\label{subsection: Compact abelian group actions}



In this subsection,
we investigate a compact abelian group action $H\curvearrowright C^*_r(G)$ whose fixed point algebra contains $C_0(G^{(0)})$.
The gauge action on the Cuntz algebra is a typical example of such actions as in Example \ref{ex: Cuntz algebra}.

First,
we describe a relation between compact abelian group actions on groupoid C*-algebras and discrete abelian group valued cocycles on \'etale groupoids via Pontryagin duality. 
For a discrete abelian group $\Gamma$,
we denote its Pontryagin dual group by $\widehat{\Gamma}$.
Namely,
we put \[\widehat{\Gamma}\defeq \{\chi\colon\Gamma\to\T\mid \text{ $\chi$ is a continuous group homomorphism.}\}.\]
Note that $\widehat{\Gamma}$ is a compact abelian group with respect to the pointwise product and pointwise convergence topology.
It is well-known that there exists a one-to-one correspondence between compact abelian group actions and discrete abelian group coactions via Pontryagin duality (see \cite[Remark 2.7]{Echterhoff_Quigg_1999}).
In this paper,
we use a one-to-one correspondence between compact abelian group actions on groupoid C*-algebras and discrete abelian group valued cocycles of underlying \'etale groupoids (Proposition \ref{prop: cocycle corresponds to group action on abelian group}).
First,
we recall how to construct a compact abelian group action from a discrete abelian group cocycle.

\begin{defi}\label{defi: compact abelian group action induced from groupoid cocycle}
	Let $G$ be a locally compact Hausdorff \'etale groupoid,
	$\Gamma$ be a discrete abelian group and $\sigma\colon G\to \Gamma$ be a continuous cocycle.
	Then the action $\tau\colon \widehat{\Gamma}\curvearrowright C^*_r(G)$ is defined by
	\[
	\tau_{\chi}(f)(\alpha)\defeq \chi(\sigma(\alpha))f(\alpha)
	\]
	for $\chi\in\widehat{\Gamma}$,
	$f\in C_c(G)$ and $\alpha\in G$.
\end{defi}

The fixed point algebra of $\tau$ can be calculated as follows.

\begin{prop}\label{prop: fixed point algebra of associated action}
	Let $G$ be a locally compact Hausdorff \'etale groupoid,
	$\Gamma$ be a discrete abelian group and $\sigma\colon G\to \Gamma$ be a continuous cocycle.
	In addition,
	let $\tau\colon\widehat{\Gamma}\curvearrowright C^*_r(G)$ denote the action induced from $\sigma$ in Definition \ref{defi: compact abelian group action induced from groupoid cocycle}.
	Then the fixed point algebra $C^*_r(G)^{\tau}$ coincides with $C^*_r(\ker\sigma)$.
\end{prop}

\begin{proof}
	Let $\delta\colon C^*_r(G)\to C^*_r(G)\otimes C^*_r(\Gamma)$ denote the coaction associated with $\sigma$.
	Then the fixed point algebra $C^*_r(G)^{\delta}$ of $\delta$ coincides with $C^*_r(G)^{\tau}$ by \cite[Remark 2.7]{Echterhoff_Quigg_1999}.
	In addition,
	we have $C^*_r(G)^{\delta}=C^*_r(\ker\sigma)$ by \cite[Lemma 6.3]{CARLSEN2021107923}.
	Therefore,
	we obtain $C^*_r(\ker\sigma)=C^*_r(G)^{\tau}$.
	\qed
	
\end{proof}

We characterize the condition that the action $\tau\colon\widehat{\Gamma}\curvearrowright C^*_r(G)$ in Definition \ref{defi: compact abelian group action induced from groupoid cocycle} becomes faithful in terms of the cocycle $\sigma\colon G\to \Gamma$.

\begin{prop}\label{prop: surjective of cocycle and faithful action}
	Let $G$ be a locally compact Hausdorff \'etale groupoid,
	$\Gamma$ be a discrete abelian group and $\sigma\colon G\to\Gamma$ be a continuous cocycle.
	Let $\tau\colon \widehat{\Gamma}\curvearrowright C^*_r(G)$ denote the action associated with $\sigma$ (see Definition \ref{defi: compact abelian group action induced from groupoid cocycle}).
	If $\sigma\colon G\to\Gamma$ is surjective,
	then $\tau\colon \widehat{\Gamma}\curvearrowright C^*_r(G)$ is faithful.
	Conversely,
	if $\tau$ is faithful and $\ker\sigma$ is topologically transitive,
	then $\sigma\colon G\to\Gamma$ is surjective.
\end{prop}

\begin{proof}
	First,
	assume that $\sigma\colon G\to\Gamma$ is surjective and $\chi\in\widehat{\Gamma}$ satisfies $\tau_{\chi}=\id_{C^*_r(G)}$.
	To show $\chi=1$,
	take $s\in\Gamma$.
	Since $\sigma$ is surjective,
	there exists $\alpha\in G$ such that $\sigma(\alpha)=s$.
	Take $f\in C_c(G)$ with $f(\alpha)=1$.
	Then we have
	\[
	\chi(s)=\chi(\sigma(\alpha))f(\alpha)=\tau_{\chi}(f)(\alpha)=f(\alpha)=1.
	\]
	Hence we obtain $\chi=1$ and $\tau$ is faithful.
	
	Next,
	assume that $\tau$ is faithful and $\ker\sigma$ is topologically transitive.
	Suppose that $\sigma$ is not surjective.
	Since $\ker\sigma$ is topologically transitive,
	the image $\sigma(\Gamma)$ is a subgroup of $\Gamma$ by \cite[Lemma 2.2.2]{Komura+2024}.
	Since $\sigma(\Gamma)$ is a proper subgroup of $\Gamma$,
	there exists $\chi\in\widehat{\Gamma}\setminus\{1\}$ such that $\chi|_{\sigma(\Gamma)}=1$.
	Then we have
	\[
	\tau_{\chi}(f)(\alpha)=\chi(\sigma(\alpha))f(\alpha)=f(\alpha)
	\]
	for all $\alpha\in G$ and $f\in C_c(G)$.
	Thus,
	we obtain $\tau_{\chi}=\id_{C^*_r(G)}$ and this contradicts to the condition that $\tau$ is faithful.
	Therefore $\sigma$ is surjective.
	\qed
\end{proof}

We have investigated the action $\tau\colon \widehat{\Gamma}\curvearrowright C^*_r(G)$ associated with a continuous cocycle $\sigma\colon G\to\Gamma$.
We give a characterization of actions induced by continuous cocycles here.
First,
if $\tau\colon \widehat{\Gamma}\curvearrowright C^*_r(G)$ is induced by a continuous cocycle $\sigma\colon G\to\Gamma$,
one can see that the fixed point algebra
\[
C^*_r(G)^{\tau}\defeq \bigcap_{\chi\in \widehat{\Gamma}}\{x\in C^*_r(G)\mid \tau_{\chi}(x)=x\}
\]
contains $C_0(G^{(0)})$.
Conversely,
if $G$ is effective,
such an action is obtained from a continuous cocycle as the following.

\begin{prop}\label{prop: cocycle corresponds to group action on abelian group}
	Let $G$ be a locally compact Hausdorff \'etale effective groupoid,
	$\Gamma$ be a discrete abelian group and $\tau\colon \widehat{\Gamma}\curvearrowright C^*_r(G)$ be a strongly continuous action.
	Assume that the fixed point algebra $C^*_r(G)^{\tau}$ contains $C_0(G^{(0)})$.
	Then there exists a continuous cocycle $\sigma\colon G\to\Gamma$ such that
	\[
	\tau_{\chi}(f)(\alpha)=\chi(\sigma(\alpha))f(\alpha)
	\]
	holds for all $f\in C_c(G)$, $\chi\in\widehat{\Gamma}$ and $\alpha\in G$.
\end{prop}

\begin{proof}
	By \cite[Remark 2.7]{Echterhoff_Quigg_1999},
	there exists a coaction $\delta\colon C^*_r(G)\to C^*_r(G)\otimes C^*_r(\Gamma)$ such that
	\[
	\tau_{\chi}(a)=(\id\otimes \chi)(\delta(a))
	\]
	holds for all $a\in C^*_r(G)$ and $\chi\in\widehat{\Gamma}$.
	Note that we may extends $\chi$ to the *-homomorphism $\chi\colon C^*_r(\Gamma)\to \C$ since we assume that $\Gamma$ is abelian and therefore amenable.
	By \cite[Corollary 2.1.9]{Komura+2024},
	there exists a cocycle $\sigma\colon G\to \Gamma$ such that
	\[
	\delta(f)=f\otimes s
	\]
	holds for all $s\in\Gamma$ and $f\in C_c(\sigma^{-1}(s))$.
	Now,
	one can check that
	\[\tau_{\chi}(f)(\alpha)=\chi(\sigma(\alpha))f(\alpha)\]
	holds for all $f\in C_c(G)$, $\chi\in\widehat{\Gamma}$ and $\alpha\in G$.
	Indeed,
	for $f\in C_c(\sigma^{-1}(s))$,
	$\chi\in\widehat{\Gamma}$ and $\alpha\in G$,
	we have
	\begin{align*}
	\tau_{\chi}(f)(\alpha)&=(\id\otimes\chi)(\delta(f))(\alpha) \\
	&=(\id\otimes\chi)(f\otimes s)(\alpha) \\
	&=\chi(s)f(\alpha)\\
	&= \chi(\sigma(\alpha))f(\alpha).
	\end{align*}
	Note that the last equation $\chi(s)f(\alpha)= \chi(\sigma(\alpha))f(\alpha)$ holds since we assume $f\in C_c(\sigma^{-1}(s))$ and the both term is $0$ if $s\not=\sigma(\alpha)$.
	Since the linear span of $\bigcup_{s\in\Gamma}C_c(\sigma^{-1}(s))$ is dense in $C^*_r(G)$,
	we obtain 
	\[\tau_{\chi}(f)(\alpha)=\chi(\sigma(\alpha))f(\alpha)\]
	for all $f\in C_c(G)$, $\chi\in\widehat{\Gamma}$ and $\alpha\in G$.
	\qed
\end{proof}

\begin{cor}\label{cor: galois group of fixed point algebra of compact abelian group}
	Let $G$ be a locally compact Hausdorff \'etale groupoid,
	$H$ be a compact abelian group and $\tau\colon H\curvearrowright C^*_r(G)$ be a strongly continuous action.
	Assume that $G$ is effective.
	In addition,
	assume that the fixed point algebra $C^*_r(G)^{\tau}$ is prime and contains $C_0(G^{(0)})$.
	Then the map
	\[
	\tau'\colon H/\ker\tau\ni q(h)\mapsto \tau_h\in\Aut_{C^*_r(G)^{\tau}}C^*_r(G),
	\]
	is an isomorphism,
	where $q\colon H\to H/{\ker\tau}$ denotes the quotient map and $h\in H$.	
\end{cor}

\begin{proof}
	First,
	we show the assertion under the assumption that $\tau$ is faithful.
	Note that $\ker\tau$ is trivial in this case.
	We may assume $H=\widehat{\Gamma}$ for some discrete abelian group $\Gamma$ by Pontryagin duality.
	By Proposition \ref{prop: cocycle corresponds to group action on abelian group},
	there exists a continuous cocycle $\sigma\colon G\to\Gamma$ such that
	\[
	\tau_{\chi}(f)(\alpha)=\chi(\sigma(\alpha))f(\alpha)
	\]
	holds for all $f\in C_c(G)$,
	$\chi\in\widehat{\Gamma}$ and $\alpha\in G$.
	Since we have $C^*_r(\ker\sigma)=C^*_r(G)^{\tau}$ and $C^*_r(G)^{\tau}$ is prime,
	$\ker\sigma$ is topologically transitive by Proposition \ref{prop: topologically transitive and prime}.
	In addition,
	since $\tau$ is faithful,
	$\sigma\colon G\to \Gamma$ is surjective by Proposition \ref{prop: surjective of cocycle and faithful action}.
	By Corollary \ref{cor: Galois group is dual of Gamma ab},
	the map
	\[
	\Psi\colon \widehat{\Gamma}\ni \chi \mapsto \tau_{\chi} \in\Aut_{C^*_r(G)^{\tau}}C^*_r(G)
	\]
	is an isomorphism as a topological group.
	This completes the proof of the corollary in case that $\tau$ is faithful.
	
	Next,
	we show the assertion for the general case.
	Put $\tau'\defeq \tau\circ q$.
	Then $\tau'$ defines a faithful action $\tau'\colon H/\ker\tau\curvearrowright C^*_r(G)$.
	From the above argument,
	the map
	\[
	\tau'\colon H/\ker\tau \to \Aut_{C^*_r(G)^{\tau'}}C^*_r(G).
	\]
	turns out to be an isomorphism.
	\qed
\end{proof}

The following corollary allows us to compute the fixed point algebra of the canonical action $\Aut_{C^*_r(\ker\sigma)} C^*_r(G)\curvearrowright C^*_r(G)$.

\begin{cor}\label{cor: fixed point algebra of the action arsing from group cocycle}
	Let $G$ be a locally compact Hausdorff \'etale groupoid,
	$\Gamma$ be a discrete group and $\sigma\colon G\to \Gamma$ be a surjective continuous cocycle.
	Assume that $G$ is effective and $\ker\sigma$ is topologically transitive.
	In addition,
	let $q\colon \Gamma\to\Gamma^{\ab}$ denote the quotient map and put $\sigma^{\ab}\defeq q\circ\sigma$.
	Then the fixed point algebra $C^*_r(G)^{\Aut_{C^*_r(\ker\sigma)}C^*_r(G)}$ of the canonical action \[\Aut_{C^*_r(\ker\sigma)}C^*_r(G)\curvearrowright C^*_r(G)\] coincides with $C^*_r(\ker\sigma^{\ab})$.
\end{cor}

\begin{proof}
	Let $\tau\colon \widehat{\Gamma^{\ab}}\curvearrowright C^*_r(G)$ denote the action induced by $\sigma^{\ab}\colon G\to\Gamma^{\ab}$.
	Since we have
	\[
	\Aut_{C^*_r(\ker\sigma)}C^*_r(G)=\Aut_{C^*_r(\ker\sigma^{\ab})}C^*_r(G)=\{\tau_{\chi}\in\Aut(G)\mid \chi\in\widehat{\Gamma^{\ab}}\}
	\]
	by Corollary \ref{cor: Galois group is dual of Gamma ab},
	we obtain
	\[
	C^*_r(G)^{\Aut_{C^*_r(\ker\sigma)}C^*_r(G)}=C^*_r(G)^{\tau}.
	\]
	In addition,
	by Proposition \ref{prop: fixed point algebra of associated action},
	we have
	\[
	C^*_r(G)^{\tau}=C^*_r(\ker\sigma^{\ab})
	\]
	and this completes the proof.
	\qed
\end{proof}

\begin{rem}
	Let $B\subset A$ be an inclusion of C*-algebras.
	Remark that Corollary \ref{cor: fixed point algebra of the action arsing from group cocycle} indicates that the fixed point algebra of the canonical action $\Aut_BA\curvearrowright A$ becomes larger than $B$ in general.
	Indeed,
	we will give an example of $G$ and $\sigma$ such that $C^*_r(\ker\sigma)\subsetneq C^*_r(\ker\sigma^{\ab})$ in Subsection \ref{subsection : Examples of coactions on Cuntz algebras and graph algebras}.
	This example also provides us an example of an inclusion of C*-algebras $B\subset A$ such that the fixed point algebra of the canonical action $\Aut_BA\curvearrowright A$ becomes larger than $B$ by Corollary \ref{cor: fixed point algebra of the action arsing from group cocycle}.

\end{rem}

\section{Examples and applications}\label{section: Examples and applications}

In this section,
we apply the results in the previous sections to examples.
In the first two subsections,
we investigate the Cuntz algebras.
In the last subsection,
we investigate C*-algebras associated with Deaconu-Renault systems and graph algebras.

\subsection{Examples of coactions on the Cuntz algebras}\label{subsection : Examples of coactions on Cuntz algebras and graph algebras}

In this subsection,
we calculate a concrete example of $\Aut_B\mathcal{O}_n$ for some C*-subalgebra $B\subset \mathcal{O}_n$ of the Cuntz algebra.
Then we point out that a map $B\mapsto \Aut_B\mathcal{O}_n$ is not injective.
Namely,
we will show that there exist C*-subalgebras $B_1, B_2\subset \mathcal{O}_n$ such that $\Aut_{B_1}\mathcal{O}_n=\Aut_{B_2}\mathcal{O}_n$ and $B_1\not=B_2$.
First,
we recall the groupoid model of the Cuntz algebras $\mathcal{O}_n$.
See \cite{Paterson2002} for details.
\begin{ex}\label{ex: Cuntz algebra}
	For $n\in\N$ with $n\geq 2$,
	let $P_n$ denote the polycyclic monoid of degree $n$.
	Recall that $P_n$ is the universal inverse semigroup generated by
	\[
	s_1, s_2,\dots,s_n,0,1
	\]
	which satisfies
	\begin{align*}
	s_i^*s_j=
	\begin{cases}
	1 & (i=j), \\
	0 & (i\not=j)
	\end{cases}
	\end{align*}
	for $i,j=1,2,\dots,n$.
	Put
	\[
	\Sigma\defeq \{1,2,\dots,n\}
	\]
	and let $\Sigma^*\defeq \bigcup_{n=0}^{\infty}\Sigma^n$ denote the set of all finite words on $\Sigma$.
	For $\mu\in\Sigma^*$,
	let $\lvert \mu\rvert $ denote the length of $\mu$ and define
	\[
	s_{\mu}\defeq s_{\mu_1}s_{\mu_2}\cdots s_{\mu_{\lvert \mu\rvert}}\in P_n.
	\]
	Then one can check that
	\[
	P_n=\{s_{\mu}s_{\nu}^*\in P_n\mid \mu,\nu\in\Sigma^*\}\cup\{0\}
	\]
	holds.

	Let $\Sigma^{\N}$ denote the set of infinite sequences on $\Sigma$.
	Note that $\Sigma^{\N}$ is a compact Hausdorff space with respect to the product topology.
	Define $\rho\colon P_n\curvearrowright \Sigma^{\N}$ by
	\[
	\rho_{s_{\mu}s_{\nu}^*}(\nu x)=\mu x,
	\]
	where $s_{\mu}s_{\nu}^*\in P_n$,
	$x\in\Sigma^{\N}$ and $\nu x$ denotes the concatenation of $\nu$ and $x$. 
	Note that $\rho_{s_{\mu}s_{\nu}^*}$ is a homeomorphism from $\nu\Sigma^{\N}$ to $\mu\Sigma^{\N}$,
	where $\nu\Sigma^{\N}\subset \Sigma^{\N}$ denotes the set of all infinite sequences which begin with $\nu$ and this is a compact open subset of $\Sigma^{\N}$.
	Put $G\defeq P_n\ltimes_{\rho}\Sigma^{\N}$.
	Then the following facts are known.

	\begin{enumerate}
		
		\item $G$ is a locally compact Hausdorff groupoid.
		In addition,
		$G$ is minimal and topologically principal (see \cite[Theorem 3.5, Proposition 5.1, 5.2]{Paterson2002}).
		
		\item Put $n\defeq \lvert\Sigma\rvert$.
		Then $C^*_r(G)$ is isomorphic to the Cuntz algebra $\mathcal{O}_{n}$ (see \cite[Corollary 3.9]{Paterson2002}).
		Indeed,
		let $S_i\in C_c(G)$ be the characteristic function on $[s_i, i\Sigma^{\N}]\subset G$ for $i=1,2,\dots,n$.
		Then $\{S_i\}_{i=1}^{n}$ generates $C^*_r(G)$ and satisfies
		\[
		S_i^*S_j=\delta_{i,j}1,\,\,\, \sum_{k=1}^{n}S_kS_k^*=1
		\]
		for all $i,j=1,2,\dots,n$.
	
		\item Let $F_n$ denote the free group generated by $t_1,t_2,\dots,t_n$.
		Then one can see that there exists a partial homomorphism $\theta\colon P_n^{\times}\to F_n$ such that $\theta(s_i)=t_i$ holds for all $i=1,2,\dots,n$.
		Hence,
		for any group $\Gamma$ and $w_1,w_2,\dots,w_n\in\Gamma$,
		there exists a partial homomorphism $\theta'\colon P_n^{\times}\to\Gamma$ such that $\theta'(s_i)=w_i$ holds for all $i=1,2,\dots,n$.
	\end{enumerate}

	Consider the cocycle $\sigma\colon G\to \Z$ defined by
	\[
	\sigma([s_{\mu}s_{\nu}^*, x])=\lvert\mu\rvert-\lvert\nu\rvert
	\]
	for $[s_{\mu}s_{\nu}^*, x]\in G$.
	Then $\sigma$ is surjective.
	Let $\tau\colon \T\curvearrowright C^*_r(G)$ be the action induced by $\sigma$ (see Definition \ref{defi: compact abelian group action induced from groupoid cocycle}).
	Then
	\[
	\tau_z(S_i)=zS_i
	\]
	holds for all $z\in \T$ and $i=1,2,\dots,n$.
	Hence $\tau$ coincides with the canonical gauge action of the Cuntz algebra.
	Now,
	one can see that $\ker\sigma$ is minimal and therefore topologically transitive.
	Hence,
	by Corollary \ref{cor: Galois group is dual of Gamma ab},
	we obtain the following proposition.
	We remark that the following proposition is already known in case that $n\not=\infty$.
	
	\begin{prop}[cf.\ {\cite[Proposition 4.4]{CONTI20122529}}]\label{prop: AutO_n^TO_n=T}
		For any $n\in\N_{\geq 2}$,
		consider the Cuntz algebra $\mathcal{O}_n$ and the gauge action $\tau\colon\T\curvearrowright \mathcal{O}_n$.
		Then
		\[\Aut_{\mathcal{O}_n^{\tau}}\mathcal{O}_n=\{\tau_z\in z\in \T\}(\simeq \T)\]
		holds.
	\end{prop}
	\begin{rem}\label{rem: corollary for Aut Oinfty is new}
		As already stated,
		Proposition \ref{prop: AutO_n^TO_n=T} is a known result.
		Indeed,
		one can deduce Proposition \ref{prop: AutO_n^TO_n=T} as a special case of \cite[Proposition 4.4]{CONTI20122529}.
		On the other hand,
		we will show
		\[
		\Aut_{\mathcal{O}_{\infty}^{\tau}}\mathcal{O}_{\infty}=\{\tau_z\in z\in \T\}(\simeq \T)
		\]
		in Corollary \ref{cor: autOinftyTOinfty is also T} and this seems to be a new result.
		We remark that we cannot immediately deduce
		\[
		\Aut_{\mathcal{O}_{\infty}^{\tau}}\mathcal{O}_{\infty}=\{\tau_z\in z\in \T\}(\simeq \T)
		\]
		from \cite[Proposition 4.4]{CONTI20122529}.
		Indeed,
		the proof of \cite[Proposition 4.4]{CONTI20122529} relies on a correspondence between unitary elements in graph algebras and certain *-endomorphisms on graph algebras,
		which holds under the assumption that the underlying graph is finite.
		This correspondence seems not to work for infinite graphs and therefore we cannot deduce
		\[
		\Aut_{\mathcal{O}_{\infty}^{\tau}}\mathcal{O}_{\infty}=\{\tau_z\in z\in \T\}(\simeq \T)
		\]
		from \cite[Proposition 4.4]{CONTI20122529},
		since the natural graph model of $\mathcal{O}_{\infty}$ is an infinite graph.
	\end{rem}
	
	Next,
	we consider another cocycle on $G$.
	Take a partial homomorphism $\theta\colon P_n^{\times}\to \mathfrak{S}_{n+1}$ such that
	\[
	\theta(s_i)=(i,i+1)
	\]
	holds for all $i=1,2,\cdots n$,
	where $\mathfrak{S}_{n+1}$ denotes the symmetric group of degree $n+1$ and $(i,i+1)$ denotes the adjacent transposition of $i$ and $i+1$ for $i=1,2,\cdots, n$.
	One can check that $\theta$ is surjective.
	Define a cocycle $\sigma\colon G\to \mathfrak{S}_{n+1}$ by
	\[
	\sigma([s_{\mu}s_{\nu}^*, x])\defeq\theta(s_{\mu}s_{\nu}^*).
	\]
	To investigate the inclusion $C^*_r(\ker\sigma)\subset C^*_r(G)$,
	we study the properties of $\sigma\colon G\to\mathfrak{S}_{n+1}$.
	
	\begin{prop}\label{prop: kersigma is minimal, where symmetric group grading}
		In the above notation,
		$\ker \sigma$ is minimal.
	\end{prop}

	\begin{proof}
		Take $x\in \Sigma^{\N}$ and $\mu\in\Sigma^*$.
		Then there exists $\nu\in\Sigma^*$ such that $\theta(s_{\nu})=\theta(s_{\mu}^*)$ since any elements in $\mathfrak{S}_{n+1}$ can be expressed as a product of adjacent transpositions.
		Now we have $[s_{\mu\nu}, x]\in\ker\sigma$ and $r([s_{\mu\nu}, x])=\mu\nu x$.
		Hence the orbit of $x$ by $\ker\sigma$ is dense in $\Sigma^\N$ and therefore $\ker\sigma$ is minimal.
		\qed
	\end{proof}
	
	By Proposition \ref{prop: kersigma is minimal, where symmetric group grading} and Corollary \ref{cor: Galois group is dual of Gamma ab},
	we obtain
	\[
	\Aut_{C^*_r(\ker\sigma)}C^*_r(G)\simeq \mathfrak{S}_{n+1}^{\ab}\simeq \Z/2\Z.
	\]
	For the same reason,
	we have
	\[
	\Aut_{C^*_r(\ker\sigma^{\ab})}C^*_r(G)=\Aut_{C^*_r(\ker\sigma)}C^*_r(G),
	\]
	where $\sigma^{\ab}\colon G\to\Z/2\Z$ denotes the composition of $\sigma$ and the quotient map $\mathfrak{S}_{n+1}\to \mathfrak{S}_{n+1}^{\ab}\simeq \Z/2\Z$.
	In addition,
	by \cite[Proposition 10.3.7]{asims},
	$C^*_r(\ker\sigma)$ and $C^*_r(\ker\sigma^{\ab})$ are simple since $\ker\sigma$ and $\ker\sigma^{\ab}$ are second-countable, effective and minimal.
	Therefore,
	we obtain an inclusion of simple C*-algebras $C^*_r(\ker\sigma)\subsetneq C^*_r(\ker\sigma^{\ab})\subsetneq C^*_r(G)$ such that
	\[
	\Aut_{C^*_r(\ker\sigma^{\ab})}C^*_r(G)=\Aut_{C^*_r(\ker\sigma)}C^*_r(G)
	\]
	holds.
	In particular,
	putting 
	\[
	\mathcal{B}\defeq \{B\subset C^*_r(G)\mid \text{$C_0(G^{(0)})\subset B$ and $B$ is a simple C*-subalgebra}\}
	\]
	and
	\[
	\mathcal{H}\defeq \{H\subset \Aut_{C_0(G^{(0)})} C^*_r(G)\mid \text{$H$ is a closed subgroup}\},
	\]
	we have observed that the map
	\[
	\Psi\colon\mathcal{B}\ni B\mapsto \Aut_BC^*_r(G)\in\mathcal{H}
	\]
	is not injective.

	Note that there exists conditional expectations $F\colon C^*_r(G)\to C^*_r(\ker\sigma)$ and $F'\colon C^*_r(G)\to C^*_r(\ker\sigma^{\ab})$
	defined by
	\[
	F(f)\defeq f|_{\ker\sigma},\,\,\, F'(f)\defeq f|_{\ker\sigma^{\ab}}
	\]
	for $f\in C_c(G)$ since $\ker\sigma, \ker\sigma^{\ab}\subset G$ are closed and we have \cite[Lemma 3.4]{BrownExelFuller2021}.
	In Proposition \ref{prop: Watatani index associated with coaction},
	we will observe $\mathrm{Ind}(F)=n!$ and $\mathrm{Ind}(F')=2$,
	where $\mathrm{Ind}(F)$ denotes the Watatani index of $F$.
	Therefore,
	the map
	\[
	\Psi\colon \mathcal{B}\ni B\mapsto \Aut_BC^*_r(G)\in\mathcal{H}
	\]
	is still not injective even if we restricts $\Psi$ to the set of C*-subalgebras of finite Watatani indices.
\end{ex}

In the last of this subsection,
we investigate the Watatani index of an inclusion $C^*_r(\ker\sigma)\subset C^*_r(G)$,
where $\sigma$ is a cocycle on $G$.
First,
we recall the definition of Watatani indices from \cite{watatani1990index}.

\begin{defi}
	Let $A$ be a C*-algebra,
	$B\subset A$ be a C*-subalgebra and $E\colon A\to B$ be a conditional expectation.
	A finite pairs $\{(u_i, v_i)\}_{i=1}^m\subset A\times A$ is called a quasi-basis if
	\[
	x=\sum_{i=1}^mu_iE(v_ix)=\sum_{i=1}^mE(xu_i)v_i
	\]
	holds for all $x\in A$.
	The Watatani index of $E$ is defined by
	\[
	\Ind E\defeq \sum_{i=1}^mu_iv_i.
	\]
\end{defi}

\begin{rem}
	The Watatani index $\Ind E$ does not depend on the choice of quasi-basis and is an element in the centre of $A$ by \cite[Proposition 1.2.8]{watatani1990index}.
\end{rem}
	
	Let $G$ be a locally compact Hausdorff \'etale groupoid and $G^{(0)}\subset H\subset G$ be a clopen subgroupoid.
	By \cite[Lemma 3.4]{BrownExelFuller2021},
	we have a conditional expectation $F\colon C^*_r(G)\to C^*_r(H)$ defined by
	\[
	F(f)\defeq f|_H
	\]
	for all $f\in C_c(G)$.
	The conditional expectation $F\colon C^*_r(G)\to C^*_r(\ker\sigma)$ in the following proposition is obtained in this way.
	
\begin{prop}\label{prop: Watatani index associated with coaction}
	Let $G$ be a locally compact Hausdorff \'etale groupoid,
	$\Gamma$ be a finite group and $\sigma\colon G\to \Gamma$ be a surjective continuous cocycle.
	In addition,
	let $F\colon C^*_r(G)\to C^*_r(\ker\sigma)$ denote the conditional expectation as above.
	If $G^{(0)}$ is compact and $\ker\sigma$ is minimal,
	then the Watatani index $\Ind F$ of $F$ is the order $\lvert \Gamma\lvert$ of $\Gamma$.
\end{prop}

First,
we prepare the following lemma to show Proposition \ref{prop: Watatani index associated with coaction}.

\begin{lem}\label{lem: lemma for watatani index}
	Consider the situation in Proposition \ref{prop: Watatani index associated with coaction}.
	Fix $s\in\Gamma$.
	For each $x\in G^{(0)}$,
	there exists a bisection $U\subset\sigma^{-1}(\{s\})$ with $x\in r(U)$.
\end{lem}

\begin{proof}
	Since $\sigma^{-1}(\{s\})$ is a non-empty open set and $\ker\sigma$ is minimal,
	there exists $\alpha\in\ker\sigma$ with $r(\alpha)=x$ and $d(\alpha)\in r(\sigma^{-1}(\{s\}))$.
	Take $\beta\in\sigma^{-1}(\{s\})$ with $d(\alpha)=r(\beta)$.
	Then we have $\alpha\beta\in\sigma^{-1}(\{s\})$ and $r(\alpha\beta)=x$.
	Now,
	take a bisection $U\subset G$ with $\alpha\beta\in U\subset \sigma^{-1}(\{s\})$ and this completes the proof.
	\qed
\end{proof}

\begin{proof}[\sc{Proof of Proposition \ref{prop: Watatani index associated with coaction}}]
		Fix $s\in\Gamma$.
		By Lemma \ref{lem: lemma for watatani index} and the compactness of $G^{(0)}$,
		there exists a family of finite bisections $\{U_i^s\}_{i\in F_s}$ such that $\bigcup_{i\in F_s}U_i^s=G^{(0)}$ and $U_i^s\subset \sigma^{-1}(\{s\})$ for each $i\in F_s$,
		where $F_s$ is a finite set.
		By the partition of unity,
		there exists $f_i^s\in C_c(r(U_i^s))$ for each $i\in F_s$ such that
		\[
		0\leq f_i^s\leq 1,\,\,\, \sum_{i\in F_s}f_i^s(x)=1
		\]
		for all $x\in G^{(0)}$.
		Put
		\[
		g_i^s\defeq  \sqrt{f_i^s\circ r|_{U_i^s}}.
		\]
		Then we have $g_i^s\in C_c(U_i^s)\subset C_c(\sigma^{-1}(\{s\}))$ for all $i\in F_s$.
		In addition,
		one can check that
		\[
		\sum_{i\in F_s}g_i^s {g_i^s}^*(x)=1
		\]
		holds for each $x\in G^{(0)}$ and $s\in\Gamma$.		
		Now,
		we show that $\{(g_i^s, {g_i^s}^*)\}_{s\in\Gamma, i\in F_s}$ is a quasi-basis for $F\colon C^*_r(G)\to C^*_r(\ker\sigma)$.
		Take $t\in\Gamma$ and $h\in C_c(\sigma^{-1}(\{t\}))$ arbitrarily.
		Then
		we have
		\begin{align*}
		\sum_{s\in \Gamma, i\in F_s} g_i^s F({g_i^s}^*h)=\sum_{i\in F_t}g_i^{t} {g_i^t}^*h=\bigg(\sum_{i\in F_t}g_i^{t} {g_i^t}^*\bigg)h=h.
		\end{align*}
		Since we have $C_c(G)=\bigoplus_{s\in \Gamma}C_c(\sigma^{-1}({t}))$ and $C_c(G)$ is dense in $C^*_r(G)$,
		we obtain
		\[
		\sum_{s\in\Gamma,i\in F_s}g_i^s F({g_i^s}^* a)=a
		\]
		for all $a\in C^*_r(G)$.
		Applying this formula to $a^*\in C^*_r(G)$ and taking the involution,
		we also obtain
		\[
		\sum_{s\in\Gamma,i\in F_s}F(a g_i^s){g_i^s}^*=a
		\]
		for all $a\in C^*_r(G)$.
		Hence $\{(g_i^s, {g_i^s}^*)\}_{s\in\Gamma, i\in F_s}$ is a quasi-basis and we obtain
		\[
		\Ind F=\sum_{s\in\Gamma,i\in F_s}g_i^s{g_i^s}^*=\sum_{s\in \Gamma} 1=\lvert \Gamma\rvert.
		\]
		This completes the proof.
		\qed
		
\end{proof}

\begin{rem}
	Under the assumptions in Proposition \ref{prop: Watatani index associated with coaction},
	$F\colon C^*_r(G)\to C^*_r(\ker\sigma)$ is the unique conditional expectation since the relative commutant $C^*_r(\ker\sigma)'\cap C^*_r(G)$ is trivial by Proposition \ref{prop: relative commutant is trivial if H is topo transitive} and we have \cite[Corollary 1.4.3]{watatani1990index}.
	In particular,
	although it is trivial,
	we have
	\[
	\Ind F=\min \{\Ind E\in [0,\infty)\mid E\in \mathcal{E}_0(C^*_r(G), C^*_r(\ker\sigma))\},
	\]
	where $\mathcal{E}_0(C^*_r(G), C^*_r(\ker\sigma))$ denotes the set of all conditional expectations $E\colon C^*_r(G)\to C^*_r(\ker\sigma)$ of index-finite type (in the present case, this is a singleton).
	We remark that the right hand side
	\[
	\min \{\Ind E\in [0,\infty)\mid E\in \mathcal{E}_0(C^*_r(G), C^*_r(\ker\sigma))\}
	\]
	is nothing but the index $[C^*_r(\ker\sigma),C^*_r(G)]_0$ of the inclusion $C^*_r(\ker\sigma)\subset C^*_r(G)$ defined in \cite[Definition 2.12.4]{watatani1990index}.
	
\end{rem}

\subsection{Groupoid model of the Cuntz algebra of infinite degree}\label{Subsection: Groupoid model of the Cuntz algebra of infinite degree}

We recall a groupoid model of the Cuntz algebra $\mathcal{O}_{\infty}$.
See \cite{Paterson2002} or \cite[Example 2.2.7]{Komura+2024} for details.

Put $\Sigma\defeq \N$ and $\Sigma^*\defeq \bigcup_{n\in\N}\Sigma^n$,
which is the set of all finite sequence on $\Sigma$.
Let $P_{\infty}$ denote the Polycyclic monoid of infinite degree.
Namely,
$P_{\infty}$ is a universal inverse semigroup defined by
\[
P_{\infty}\defeq \i<\{s_i\}_{i=1}^{\infty},0,1 \mid s_i^*s_j=\delta_{i,j}1>.
\]
Remark that
\[
P_{\infty}=\{s_{\mu}s_{\nu}^*\mid\mu,\nu\in\Sigma^*\}\cup\{0\}
\]
holds,
where $s_{\mu}\defeq s_{\mu_1}s_{\mu_2}\cdots s_{\mu_k}$ for $\mu\in\Sigma^*$ with $\lvert \mu\rvert=k$.
Let $X\defeq \Sigma^*\cup\Sigma^{\N}$ be the set of all finite or infinite sequences on $\Sigma$.
For $\mu\in\Sigma^*$ and a finite set $F\subset \Sigma^*$,
define $C_{\mu,F}\subset X$ to be the set of all sequences which begin with $\mu$ and do not begin with the elements in $F$.
Then a family of all $C_{\mu,F}$ forms an open basis of $X$ and $X$ is a compact Hausdorff space with respect to the topology generated by all $C_{\mu, F}$.
For $\mu,\nu\in\Sigma^*$,
define $\alpha_{s_{\mu}s_{\nu}^*}\colon C_{\nu,\emptyset}\to C_{\mu,{\emptyset}}$ by
\[
\alpha_{s_{\mu}s_{\nu}^*}(\nu x)=\mu x
\]
for $x\in X$.
Then we obtain the action $\alpha\colon P_{\infty}\curvearrowright X$.
Put $G\defeq P_{\infty}\ltimes_{\alpha}X$.
By \cite[Example 2.2.7]{Komura+2024},
the following properties hold:
\begin{enumerate}
	\item $G$ is a locally compact Hausdorff \'etale groupoid.
	In addition,
	$G$ is minimal and topologically principal.
	\item There exists a continuous cocycle $\sigma\colon G\to \Z$ defined by
	\[
	\sigma([s_{\mu}s_{\nu}^*, x])=\lvert \mu\rvert -\lvert \nu\rvert
	\]
	for $[s_{\mu}s_{\nu}^*, x]\in G$.
	\item $\ker\sigma\subset G$ is not minimal but topologically transitive.
	\item There exists a *-isomorphism $\varphi\colon C^*_r(G)\to \mathcal{O}_{\infty}$ such that
	\[
	\varphi(\chi_{[s_{\mu}s_{\nu}^*, C_{\nu,\emptyset}]})=S_{\mu}S_{\nu}^*
	\]
	holds.
	In particular,
	\begin{align*}
	&\varphi(C(G^{(0)}))=\overline{\Span}\{S_{\mu}S_{\mu}\in\mathcal{O}_{\infty}\mid \mu\in\Sigma^*\},\\
	&\varphi(C^*_r(\ker\sigma))=\overline{\Span}\{S_{\mu}S_{\nu}^*\in\mathcal{O}_{\infty}\mid \mu,\nu\in\Sigma*, \lvert\mu\rvert=\lvert\nu\rvert\}=\mathcal{O}_{\infty}^{\tau}
	\end{align*}
	holds,
	where $\mathcal{O}_{\infty}^{\tau}$ denotes the fixed point algebra of the gauge action $\tau\colon\T\curvearrowright \mathcal{O}_{\infty}$.
	\item Let $\tau'\colon \T\curvearrowright C^*_r(G)$ denotes the action induced by $\sigma$.
	Then $\varphi$ is a $\T$-equivariant *-isomorphism in the sense that $\varphi\circ \tau'_z=\tau_z\circ \varphi$ holds for all $z\in\T$.
\end{enumerate}
The above facts allow us to apply Corollary \ref{cor: galois group of fixed point algebra of compact abelian group} and we obtain the following corollary.

\begin{cor}\label{cor: autOinftyTOinfty is also T}
	Let $\tau\colon \T\curvearrowright \mathcal{O}_{\infty}$ denote the gauge action and $\mathcal{O}_{\infty}^{\T}$ denote the fixed point subalgebra of $\tau$.
	Then
	\[
	\Aut_{\mathcal{O}_{\infty}^{\T}}\mathcal{O}_{\infty}=\{\tau_z\in\Aut(\mathcal{O}_{\infty})\mid z\in\T\}(\simeq \T).
	\]
	holds.
\end{cor}

\begin{rem}
We give a few remarks about Corollary \ref{cor: autOinftyTOinfty is also T}.
First,
we note that Corollary \ref{cor: autOinftyTOinfty is also T} seems a new result as mentioned in Remark \ref{rem: corollary for Aut Oinfty is new}.

Second,
although one may expect that the Galois correspondence result between $\T$ and the inclusion $\mathcal{O}_{\infty}^{\T}\subset \mathcal{O}_{\infty}$,
this is not the case.
Indeed,
there exists an intermediate C*-subalgebra $B$ between $\mathcal{O}_{\infty}^{\T}$ and $\mathcal{O}_{\infty}$ which does not become a fixed point subalgebra of restricted action $\tau|_H\colon H\curvearrowright \mathcal{O}_{\infty}$ for any closed subgroup $H\subset \T$ as mentioned in \cite[Example 2.2.10]{Komura+2024}.
Note that the Galois correspondence result for $\mathcal{O}_n^{\T}\subset \mathcal{O}_n$ holds by \cite[Example 5.11]{Rrdam2023}.
Finally,
we remark that,
while $\mathcal{O}_{n}^{\T}$ is a UHF-algebra and hence simple,
$\mathcal{O}_{\infty}^{\T}$ is not simple.
Indeed,
if $\mathcal{O}_{\infty}^{\T}$ was simple,
$\ker\sigma$ should be minimal but it is not the case.

\end{rem}

\subsection{Deaconu-Renault systems} \label{subsection: Deaconu-Renault systems}

We apply our main theorems to C*-algebras associated with Deaconu-Renault systems.
Our aim in this subsection is to introduce and investigate the notion of flip eventual conjugacy,
which is a equivalence relation between Deaconu-Renault systems.
In Theorem \ref{theorem: characterization of flip eventually conjugate},
we will characterize flip eventual conjugacy in terms of \'etale groupoids and C*-algebras.
Before discussing flip eventual conjugacy,
we prepare a general theory about Deaconu-Renault systems and associated \'etale groupoids.
See \cite[Example 2.3.7]{asims} for details of \'etale groupoid associated with Deaconu-Renault systems.

A Deaconu-Renault system $(X,T)$ consists of a locally compact Hausdorff space $X$ and a local homeomorphism $T\colon X\to X$.
Note that $T\colon X\to X$ is a local homeomorphism if for all $x\in X$,
there exists an open neighbourhood $U\subset X$ of $x$ such that $T(U)$ is open in $X$ and the restriction $T|_U$ is a homeomorphism onto $T(U)$.
For simplicity,
we assume that $T$ is defined on the whole space $X$.
For a Deaconu-Renault system $(X,T)$,
the Deaconu-Renault groupoid $G(X,T)$ is defined as follows.
Put 
\begin{align*}
&G(X,T)\defeq \{(y,n-m,x)\in X\times \Z\times X\mid T^ny=T^mx\},\\
&G(X,T)^{(0)}\defeq \{(x,0,x)\in G(X,T)\mid x\in X\}.
\end{align*}
We identify $G(X,T)^{(0)}$ with $X$ via the bijection
\[
X\ni x\mapsto (x,0,x)\in G(X,T)^{(0)}.
\]
The domain map and range map $d,r\colon G(X,T)\to X$ are defined by
\[
d(y,n,x)\defeq x,\,\,\, r(y,n,x)\defeq y
\]
for $(y,n,x)\in G(X,T)$.
The product and inverse of $G(X,T)$ is defined by
\[
(z,n,y)\cdot (y,m,x)\defeq (z,n+m, x),\,\,\, (y,m,x)^{-1}\defeq (x,-m,y)
\]
for $(z,n,y),(y,m,x)\in G(X,T)$.
For open sets $U,V\subset X$ and $n,m\in \N$,
define
\[
Z(U,n,m,V)\defeq \{(y,n-m,x)\in G(X,T)\mid T^ny=T^mx\}.
\]
Then $G(X,T)$ is a locally compact Hausdorff \'etale groupoid with respect to the topology generated by a family $\{Z(U,n,m,V)\}_{U,V\subset X, n,m\in\N}$,
where $U$ and $V$ are taken around all open subsets of $X$.
In addition,
$G(X,T)$ has the canonical $\Z$-valued continuous cocycle defined by
\[
\sigma_X\colon G(X,T)\ni (y,n,x)\mapsto n\in\Z.
\]

The following fact about cocycles on $G(X,T)$ is elementary.
See \cite[Section 4.1]{RenaultAF-equovalence} for more details.
\begin{prop}[{\cite[Section 4.1]{RenaultAF-equovalence}}]\label{prop: cocycles on Deaconu-Renault systems.}
	Let $(X,T)$ be a Deaconu-Renault system and $H$ be a topological abelian group.
	For a continuous function $f\colon X\to H$,
	define $\sigma_f\colon G(X,T)\to H$ by
	\[
	\sigma_f(y,n-m,x)\defeq \sum_{i=0}^{n-1}f(T^i(y))-\sum_{j=0}^{m-1}f(T^j(x))
	\]
	for $(y,n-m,x)\in G(X,T)$ with $T^n(y)=T^m(x)$.
	Then $\sigma_f\colon G(X,T)\to H$ is a continuous cocycle.
	In addition,
	for a continuous cocycle $\sigma\colon G(X,T)\to H$,
	define $f_{\sigma}\colon X\to H$ by
	\[
	f_{\sigma}(x)\defeq \sigma(x,1,T(x))
	\]
	for $x\in X$.
	Then the assignment $f\mapsto \sigma_f$ is a bijection from the set of continuous functions $f\colon X\to H$ to the set of continuous cocycles $\sigma\colon G(X,T)\to H$.
	The inverse map is given by $\sigma\mapsto f_{\sigma}$.
\end{prop}
	
	For an \'etale groupoid associated with a Deaconu-Renault system,
	topological principality is equivalent to effectiveness.
	
\begin{prop} \label{prop: effective and topologically principal is equivalent for Deaconu-Renault}
	Let $(X,T)$ be a Deaconu-Renault system.
	Then $G(X,T)$ is effective if and only if $G(X,T)$ is topologically principal.
\end{prop}
\begin{proof}
	Since $G(X,T)$ is Hausdorff,
	$G(X,T)$ is effective if $G(X,T)$ is topologically principal by \cite[Proposition 3.6]{renault}.
	We show the converse and assume that $G(X,T)$ is effective.
	For $n,m\in\N$,
	put
	\[
	A_{n,m}\defeq \{x\in X\mid T^n(x)\not=T^m(x)\}.
	\]
	Then $A_{n,m}$ is open and dense in $X$ if $n\not=m$.
	Indeed,
	it is obvious that $A_{n,m}$ is open.
	Assume that there exists a non-empty open set $U\subset X$ such that $U\cap A_{n,m}=\emptyset$.
	Then, for $x\in U$,
	$(x,n-m,x)$ belongs to $\Iso(G(X,T))^{\circ}$.
	Since we assume that $G(X,T)$ is effective,
	we obtain $n=m$ and this is a contradiction.
	Therefore $A$ is dense in $X$.
	
	Now,
	put
	\[
	A\defeq \bigcap_{n\not=m} A_{n,m}.
	\]
	Then $A$ is dense in $X$ by Baire category theorem.
	In addition,
	one can check that
	\[
	d^{-1}(\{x\})\cap r^{-1}(\{x\})=\{x\}
	\]
	holds for all $x\in A$.
	Hence $G(X,T)$ is topologically principal.
	\qed
\end{proof}

\begin{defi}
	A Deaconu-Renault system $(X,T)$ is said to be topologically free if $G(X,T)$ is topologically principal.
	By Proposition \ref{prop: effective and topologically principal is equivalent for Deaconu-Renault},
	this is equivalent to the condition that $G(X,T)$ is effective.
\end{defi}

\subsubsection{Continuous orbit maps of Deaconu-Renault systems}

The notion of continuous orbit equivalence between Deaconu-Renault systems is introduced in \cite[Definition 8.1]{CARLSEN2021107923}.
This notion is a kind of equivalence relation between Deaconu-Renault systems.
In this subsection,
we introduce a slightly generalized notion which we call continuous orbit maps.
In this subsection,
 we explain how to characterize continuous orbit maps in terms of associated \'etale groupoids following  \cite[Section 8.1]{CARLSEN2021107923}.
We note that almost all of statements in this subsection is a special version in \cite[Section 8.1]{CARLSEN2021107923}.
In \cite[Section 8.1]{CARLSEN2021107923},
the authors treat general Deaconu-Renault systems.
In this subsection,
we mainly restrict our attention to topologically free Deaconu-Renault systems.
Instead of imposing this restriction,
we aim to simplify the proofs.
For this reason,
we give proofs for already known results in \cite[Section 8.1]{CARLSEN2021107923}.

\begin{defi}[{cf.\ \cite[Definition 8.1]{CARLSEN2021107923}}]
	Let $(X,T)$ and $(Y,S)$ be Deaconu-Renault systems.
	We say that a triplet $(l,k,h)$ is a continuos orbit map from $(X,T)$ to $(Y,S)$ if the following conditions hold:
	\begin{enumerate}
		\item $l,k\colon X\to \N$ and $h\colon X\to Y$ are continuous maps, and
		\item for all $x\in X$,
		\[
		S^{l(x)}(h(x))=S^{k(x)}(h(T(x)))
		\]
		holds.
	\end{enumerate}
	A 5-tuple $(l,k,l',k',h)$ is called a continuous orbit equivalence between $(X,T)$ and $(Y,S)$ if
	\begin{enumerate}
		\item $l,k\colon X\to \N$ and $l',k'\colon Y\to \N$ are continuous maps,
		\item $h\colon X\to Y$ is a homeomorphism,
		\item $(l,k,h)$ is a continuous orbit map from $(X,T)$ to $(Y,S)$, and
		\item $(l',k',h^{-1})$ is a continuous orbit map from $(Y,S)$ to $(X,T)$.
	\end{enumerate}
\end{defi}

The following proposition is a slight generalization of \cite[Lemma 8.8]{CARLSEN2021107923}.

\begin{prop}[{cf.\ \cite[Lemma 8.8]{CARLSEN2021107923}}]\label{prop: groupoid hom induced by continuous orbit equivalent}
	Let $(X,T)$ and $(Y,S)$ be Deaconu-Renault systems.
	Assume that $(l,k,h)$ is a continuous orbit map from $(X,T)$ to $(Y,S)$.
	Applying Proposition \ref{prop: cocycles on Deaconu-Renault systems.} to $l-k\colon X\to \Z$,
	define a continuous cocycle $\sigma_{l-k}\colon G(X,T)\to \Z$ by
	\[
	\sigma_{l-k}(y,n-m,x)\defeq \sum_{i=0}^{n-1}(l(T^i(y))-k(T^i(y)))-\sum_{j=0}^{m-1}(l(T^j(x))-k(T^j(x)))
	\]
	for $(y,n-m,x)\in G(X,T)$ with $T^n(y)=T^m(x)$.
	Then 
	\[(h(y), \sigma_{l-k}(y,n-m,x),h(x))\in G(Y,S)\]
	holds for all $(y,n-m,x)\in G(X,T)$ with $T^n(y)=T^m(x)$.
	In addition,
	the map
	\[
	\Phi\colon G(X,T)\ni (y,n-m,x)\mapsto (h(y), \sigma_{l-k}(y,n-m,x),h(x))\in G(Y,S)
	\]
	is a continuous groupoid homomorphism.
\end{prop}

\begin{proof}
	Using the formula
	\[
	S^{l(x)}(h(x))=S^{k(x)}(h(Tx))
	\]
	repeatedly,
	one can check that
	\[
	S^{\sum_{i=0}^{n-1}l(T^i(x))}(h(x))=S^{\sum_{i=0}^{n-1}k(T^i(x))}(h(T^{n}(x)))
	\]
	holds for all $x\in X$ and $n\in\N$.
	Fix $(y,n-m,x)\in G(X,T)$ with $T^{n}(y)=T^m(x)$ arbitrarily.
	Then we have
	\begin{align*}
	S^{\sum_{i=0}^{n-1}l(T^i(y))+\sum_{j=0}^{m-1}k(T^j(x))}(h(x))&=S^{\sum_{i=0}^{n-1}k(T^i(y))+\sum_{j=0}^{m-1}k(T^j(x))}(h(T^n(y))) \\
	&=S^{\sum_{i=0}^{n-1}k(T^i(y))+\sum_{j=0}^{m-1}k(T^j(x))}(h(T^m(x)))\\
	&=S^{\sum_{i=0}^{n-1}k(T^i(y))+\sum_{j=0}^{m-1}l(T^j(x))}(h(x)).
	\end{align*}
	Hence we obtain
	\[
	(h(y),\sigma_{l-k}(y,n-m,x), h(x))\in G(Y,S).
	\]
	
	Next,
	we show that $\Phi$ is a groupoid homomorphism.
	Take $\alpha\defeq (z,n,y)$ and $\beta\defeq (y,m,x)\in G(X,T)$ arbitrarily.
	Then $\Phi(\alpha)$ and $\Phi(\beta)$ are composable and we have
	\begin{align*}
	\Phi(\alpha)\Phi(\beta)&=(h(z),\sigma_{l-k}(\alpha),h(y))(h(y),\sigma_{l-k}(\beta),h(x))\\
	&=(h(z),\sigma_{l-k}(\alpha)+\sigma_{l-k}(\beta),h(x)) \\
	&=(h(z),\sigma_{l-k}(\alpha\beta),h(x))=\Phi(\alpha\beta).
	\end{align*}
	Thus $\Phi$ is a groupoid homomorphism.
				
	Finally,
	we show that $\Phi$ is continuous.
	Fix $(y_0,n-m,x_0)\in G(X,T)$,
	open sets $U,V\subset Y$ and $p,q\in\N$ such that 
	\[
	\Phi(y_0,n-m,x_0)=(h(y_0),\sigma_{l-k}(y_0,n-m,x_0), h(x_0))\in Z(U,p,q,V)
	\]
	and $T^{n}(y_0)=T^m(x_0)$.
	Note that we have $S^{p}(h(y))=S^{q}(h(x))$.
	Put 
	\begin{align*}
	i&\defeq \sum_{i=0}^{n-1}l(T^i(y_0))+\sum_{j=0}^{m-1}k(T^j(x_0)),\\
	j&\defeq \sum_{i=0}^{n-1}k(T^i(y_0))+\sum_{j=0}^{m-1}l(T^j(x_0)).
	\end{align*}
	Then we have $S^i(h(y_0))=S^{j}(h(x_0))$ and \[\sigma_{l-k}(y_0,n-m,x_0)=p-q=i-j.\]
	In case that $i>p$ holds,
	there exists an open neighbourhood $W\subset Y$ of $S^{p}(h(y_0))$ such that $S^{i-p}|_W$ is injective since $S$ is a local homeomorphism.
	If $i\leq p$,
	put $W\defeq Y$.
	Now, take open subsets $\widetilde{U},\widetilde{V}\subset X$ such that the followings hold:
	\begin{itemize}
		\item $y_0\in\widetilde{U}$ and $x_0\in \widetilde{V}$,
		\item $\widetilde{U}\subset h^{-1}(U)\cap (S^i\circ h)^{-1}(W)$ and $\widetilde{V}\subset h^{-1}(V)\cap (S^j\circ h)^{-1}(W)$,
		\item the formulae
		\begin{align*}
		i&=\sum_{i=0}^{n-1}l(T^i(y))+\sum_{j=0}^{m-1}k(T^j(x)), \\
		j&=\sum_{i=0}^{n-1}k(T^i(y))+\sum_{j=0}^{m-1}l(T^j(x))
		\end{align*}
		holds for all $y\in\widetilde{U}$ and $x\in\widetilde{V}$.
	\end{itemize}
We observe
\[
\Phi(Z(\widetilde{U},n,m,\widetilde{V}))\subset Z(U,p,q,V).
\]
Indeed,
for $(y,n-m,x)\in Z(\widetilde{U},n,m,\widetilde{V})$,
we have $S^i(h(y))=S^{j}(h(x))$.
We claim $S^{p}(h(y))=S^{q}(h(x))$.
This is obvious if $i\leq p$.
In case that $i>p$,
since $S^{p}(h(y))$ and $S^{q}(h(x))$ are contained in $W$ and $S^{i-p}=S^{j-q}$ is injective on $W$,
we obtain $S^{p}(h(y))=S^{q}(h(x))$.
Thus we have shown
\[
\Phi(y,n-m,x)=(h(y),p-q,h(x))\in Z(U,p,q, V).
\]
Hence $\Phi$ is continuous.
\qed	
\end{proof}

The next lemma is a variant of Proposition \ref{prop: all we need is effectiveness}.
Remark that we do not assume that $\Phi$ is an automorphism while we assume that $G$ is topologically principal.

\begin{lem} \label{lemma: groupoid endo which is identity on unit space is idendity}
	Let $G$ be a locally compact Hausdorff \'etale topologically principal groupoid.
	Assume that $\Phi\colon G\to G$ is a continuous groupoid homomorphism such that $\Phi|_{G^{(0)}}=\id_{G^{(0)}}$.
	Then $\Phi=\id_G$ holds.
\end{lem}
\begin{proof}
	Put
	\[
	A\defeq \{x\in G^{(0)}\mid d^{-1}(\{x\})\cap r^{-1}(\{x\})=\{x\}\}.
	\]
	Then $A$ is dense in $G^{(0)}$ since we assume that $G$ is topologically principal.
	Since $d\colon G\to G^{(0)}$ is an open map,
	$d^{-1}(A)$ is dense in $G$.
	Hence,
	to show $\Phi=\id_G$,
	it suffices to show $\Phi(\alpha)=\alpha$ for all $\alpha\in d^{-1}(A)$.
	Take $\alpha\in d^{-1}(A)$.
	Since we have
	\[d(\Phi(\alpha)^{-1})=r(\Phi(\alpha))=\Phi(r(\alpha))=r(\alpha),\]
	$\Phi(\alpha)^{-1}$ and $\alpha$ are composable.
	Since we have
	\[d(\Phi(\alpha)^{-1}\alpha)=r(\Phi(\alpha)^{-1}\alpha)=d(\alpha)\in A,
	\]
	we obtain $\Phi(\alpha)^{-1}\alpha=d(\alpha)$ and therefore $\Phi(\alpha)=\alpha$.
	\qed
\end{proof}

\begin{cor}[{\cite[Proposition 8.3]{CARLSEN2021107923}}]\label{Corollary: continuous orbit equivalence implies groupoid isomorphisms}
	Let $(X,T)$ and $(Y,S)$ be topologically free Deaconu-Renault systems.
	Assume that there exits a continuous orbit equivalence $(l,k,l',k',h)$ between $(X,T)$ and $(Y,S)$.
	Then the continuous groupoid homomorphisms
	\[
	\Phi\colon G(X,T)\ni (y,n-m,x)\mapsto (h(y), \sigma_{l-k}(y,n-m,x),h(x))\in G(Y,S)
	\]
	and 
	\[
	\Psi\colon G(Y,S)\ni (y,n-m,x)\mapsto (h^{-1}(y), \sigma_{l'-k'}(y,n-m,x),h^{-1}(x))\in G(Y,S).
	\] 
	induced by Proposition \ref{prop: groupoid hom induced by continuous orbit equivalent} are groupoid isomorphisms and $\Phi=\Psi^{-1}$ holds.
\end{cor}
\begin{proof}
	One can check that $\Psi\circ\Phi|_X=\id_X$ and $\Phi\circ\Psi|_Y=\id_Y$.
	Now we obtain $\Psi\circ\Phi=\id_{G(X,T)}$ and $\Phi\circ\Psi=\id_{G(Y,S)}$ by Lemma \ref{lemma: groupoid endo which is identity on unit space is idendity} and Proposition \ref{prop: effective and topologically principal is equivalent for Deaconu-Renault}.
	Hence we obtain $\Phi=\Psi^{-1}$ and,
	in particular,
	$\Phi$ is an isomorphism.
	\qed
\end{proof}

The following lemma is same as \cite[Lemma 8.4]{CARLSEN2021107923}.
We include a proof for readers' convenience.
\begin{lem}[{\cite[Lemma 8.4]{CARLSEN2021107923}}]\label{lemma: continuity of continuous orbit equivalence}
	Let $(X,T)$ be a Deaconu-Renault system.
	Define $\widetilde{l}\colon G(X,T)\to \N$ by
	\[
	\widetilde{l}(y,n,x)\defeq \min\{p\in\N\mid p\geq n, T^p(y)=T^{p-n}(x)\}.
	\]
	Then $\widetilde{l}\colon G(X,T)\to \N$ is a continuous function.
\end{lem}

\begin{proof}
	Fix $(y_0, n, x_0)\in G(X,T)$ and put $l_0\defeq \widetilde{l}(y_0,n, x_0)$.
	First,
	assume $l_0=\max\{0,n\}$.
	Take open neighbourhoods $U\subset X$ (resp.\ $V\subset X$) of $y_0$ (resp.\ $x_0$) respectively.
	If $l_0=0$,
	then we have $n\leq 0$ and $y_0=T^{-n}(x_0)$.
	Then $Z(U,0,-n,V)$ is an open neighbourhood of $(y_0,n,x_0)$ and $\widetilde{l}|_{Z(U,0,-n,V)}=0$.
	If $l_0=n$,
	then $Z(U,l_0, 0, V)$ is an open neighbourhood of $(y_0,n,x_0)$ and $\widetilde{l}|_{Z(U,0,-n,V)}=l_0$
	
	Next,
	assume $l_0>\max\{0,n\}$.
	Then we have $T^{l_0-1}(y_0)\not=T^{l_0-1-n}(x_0)$.
	There exists open neighbourhoods $U\subset X$ (resp.\ $V\subset X$) of $y_0$ (resp.\ $x_0$) such that $T^{l_0-1}(y)\not=T^{l_0-1-n}(x)$ holds for all $y\in U$ and $x\in V$.
	Then one can check that $\widetilde{l}|_{Z(U,l_0, l_0-n, V)}=l_0$.
	Hence, $\widetilde{l}$ is a locally constant function and therefore continuous.
	\qed
\end{proof}

The following corollary is a slight generalization of \cite[$(1)\Rightarrow(2)$ in Proposition 8.3]{CARLSEN2021107923}.

\begin{cor}[{cf.\ \cite[$(1)\Rightarrow(2)$ in Proposition 8.3]{CARLSEN2021107923}}]\label{cor: isom of DR groupoid induces continuous orbit equivalent}
	Let $(X,T)$ and $(Y,S)$ be Deaconu-Renault systems.
	Assume that there exists a continuous groupoid homomorphism $\Phi\colon G(X,T)\to G(Y,S)$.
	Then there exists a continuous orbit map $(l,k,\Phi|_X)$ from $(X,T)$ to $(Y,S)$ such that
	\[
	l(x)-k(x)=\sigma_Y(\Phi(x,1,Tx))
	\]
	holds for all $x\in X$.
	If $\Phi\colon G(X,T)\to G(Y,S)$ is an isomorphism,
	then there exists a continuous orbit equivalence $(l,k,l',k',\Phi|_X)$ between $(X,T)$ and $(Y,S)$ such that
	\begin{align*}
	l(x)-k(x)&=\sigma_Y(\Phi(x,1,Tx)) \\
	l'(y)-k'(y)&=\sigma_X(\Phi^{-1}(y,1,Sy))
	\end{align*}
	holds for all $x\in X$ and $y\in Y$.
\end{cor}

\begin{proof}
	Let $\widetilde{l}\colon G(X,T)\to\N$ denotes the continuous function in Lemma \ref{lemma: continuity of continuous orbit equivalence}.
	Define $l,k\colon X\to\N$ by
	\[l(x)\defeq \widetilde{l}(\Phi(x,1,T(x))), k(x)\defeq l(x)-\sigma(\Phi(x,1,Tx)).\]
	Note that $l$ and $k$ take values in $\N$ by the definition of $\widetilde{l}$ in Lemma \ref{lemma: continuity of continuous orbit equivalence}.
	Then one can check that $(l,k, \Phi|_X)$ is a continuous orbit map from $(X,T)$ to $(Y,S)$.
	If $\Phi$ is an isomorphism,
	apply the above argument to $\Phi^{-1}$ and then we obtain a continuous orbit equivalence between $(X,T)$ and $(Y,S)$.
	\qed
\end{proof}

\subsubsection{Flip eventual conjugacy of Deaconu-Renault systems}

In this subsection,
we introduce a flip eventual conjugacy,
which is a equivalence relation between Deaconu-Renault systems.
Then we characterize a flip eventual conjugacy in terms of \'etale groupoid and C*-algebras in Theorem \ref{theorem: characterization of flip eventually conjugate}.

\begin{defi}
	Let $(X,T)$ and $(Y,S)$ be Deaconu-Renault systems.
	Then $(X,T)$ and $(Y,S)$ are said to be flip eventually conjugate if there exists a continuous orbit equivalence $(l,k,l',k',h)$ between $(X,T)$ and $(Y,S)$ such that
	\[l-k=l'-k'=1 \text{ or } l-k=l'-k'=-1\]
	holds.
	In addition,
	$(X,T)$ and $(Y,S)$ are said to be eventually conjugate if there exists a continuous orbit equivalence $(l,k,l',k',h)$
	such that
	\[
	l-k=l'-k'=1
	\]
	holds.
\end{defi}
Obviously,
eventually conjugate Deaconu-Renault systems are flip eventually conjugate.
If the underlying space $X$ and $Y$ are compact,
we may take $l,l',k,k'$ as constant functions as the following:

\begin{prop}\label{prop: if X and Y are compact, cocycles are constant}
	Let $(X,T)$ and $(Y,S)$ be Deaconu-Renault systems such that $X$ and $Y$ are compact.
	Then $(X,T)$ and $(Y,S)$ are eventually conjugate if and only if there exists $m\in \N$ and a homeomorphism $h\colon X\to Y$
	such that
	\[
	S^{m+1}\circ h=S^{m}\circ h\circ T \text{ and } T^{m+1}\circ h^{-1}=T^{m}\circ h^{-1}\circ S
	\]
	holds.
	
	In addition,
	$(X,T)$ and $(Y,S)$ are flip eventually conjugate if and only if there exists $m\in\N$ and a homeomorphism $h\colon X\to Y$ such that
	\[
	S^{m+1}\circ h=S^{m}\circ h\circ T \text{ and } T^{m+1}\circ h^{-1}=T^{m}\circ h^{-1}\circ S
	\]
	or
	\[
	S^{m}\circ h=S^{m+1}\circ h\circ T \text{ and } T^{m}\circ h^{-1}=T^{m+1}\circ h^{-1}\circ S
	\]
	hold.
\end{prop}

\begin{proof}
	Assume that there exists $m\in \N$ and a homeomorphism $h\colon X\to Y$
	such that
	\[
	S^{m+1}\circ h=S^{m}\circ h\circ T \text{ and } T^{m+1}\circ h^{-1}=T^{m}\circ h^{-1}\circ S
	\]
	holds.
	Then,
	putting
	\[
	l=l'=m, k=k'=m+1,
	\]
	we obtain an eventual conjugacy $(l,k,l',k',h)$ between $(X,T)$ and $(Y,S)$.
	Conversely,
	assume that $(X,T)$ and $(Y,S)$ are eventually conjugate and let $(l,k,l',k',h)$ be a continuous orbit equivalence such that $l-k=l'-k'=1$.
	Since $X$ and $Y$ are compact,
	we may put
	\[
	m\defeq \max \{l(x)\in\N\mid x\in X\}\cup \{l'(y)\in\N\mid y\in Y\}.
	\]
	Then one can check
	\[
	S^{m+1}(h(x))=S^{m}(h(T(x))) \text{ and } T^{m+1}(h^{-1}(y))=T^{m}(h^{-1}(S(y)))
	\]
	holds for all $x\in X$ and $y\in Y$.
	
	The statement for flip eventual conjugacy is shown in the same way.
	\qed
\end{proof}

We give a characterisation of flip eventual conjugacy.

\begin{thm}\label{theorem: characterization of flip eventually conjugate}
	Let $(X,T)$ and $(Y,S)$ be topologically free Deaconu-Renault systems.
	Consider the following conditions.
	\begin{enumerate}
		\item $(X,T)$ and $(Y,S)$ are flip eventually conjugate,
		\item there exists an isomorphism $\Phi\colon G(X,T)\to G(Y,S)$ and $\tau\in\Aut(\Z)$ such that $\sigma_Y\circ \Phi=\tau\circ\sigma_X$ holds,
		\item there exists an isomorphism $\Phi\colon G(X,T)\to G(Y,S)$ such that $\Phi(\ker\sigma_X)=\ker\sigma_Y$,
		and
		\item there exists a *-isomorphism $\varphi\colon C^*_r(G(X,T))\to C^*_r(G(Y,S))$ such that $\varphi(C^*_r(\ker\sigma_X))=C^*_r(\ker\sigma_Y)$ and $\varphi(C_0(X))=C_0(Y)$.		
	\end{enumerate}
	Then (1)$\Leftrightarrow$(2)$\Rightarrow$(3)$\Leftrightarrow$(4) hold.
	If $\ker\sigma_X$ and $\ker\sigma_Y$ are topologically transitive,
	then (3)$\Rightarrow$(2) holds and hence all conditions are equivalent.
\end{thm}

\begin{proof}
	It is straightforward to check (3)$\Rightarrow$(4).
	We show (4)$\Rightarrow$(3).
	By \cite[Theorem 2.1.1]{komura2023homomorphisms},
	we obtain a groupoid isomorphism $\Phi\colon G(X,T)\to G(Y,S)$ and $c\in Z(G(X,T))$ such that
	\[
	\varphi(f)(\delta)=c(\Phi^{-1}(\delta))f(\Phi^{-1}(\delta))
	\]
	holds for all $\delta\in G(Y,S)$.
	Then one can check that $\Phi(\ker\sigma_X)=\ker\sigma_Y$ in the same way as the proof of Proposition \ref{prop: restricted Weyl group is restricted Weyl groupoid automorphisms}.
	
	Next,
	we show (1)$\Rightarrow$ (2).
	Let $(l,k,l',k',h)$ be a continuous orbit equivalence between $(X,T)$ and $(Y,S)$ such that
	\[l-k=l'-k'=1 \text{ or } l-k=l'-k'=-1.\]
	First,
	suppose $l-k=l'-k'=1$.
	Let $\Phi\colon G(X,T)\to G(Y,S)$ and $\Psi\colon G(Y,S)\to G(X,T)$ be the groupoid isomorphisms in Corollary \ref{Corollary: continuous orbit equivalence implies groupoid isomorphisms}.
	Namely,
	we have
	\begin{align*}
	\Phi(y,n,x) &=(h(y),\sigma_{l-k}(y,n,x), h(x))\\
	\Psi(y',m,x')&=(h^{-1}(y),\sigma_{l'-k'}(y',m,x'), h^{-1}(x))
	\end{align*}
	for all $(y,n,x)\in G(X,T)$ and $(y',m,x')\in G(Y,S)$.
	Note that we have $\Psi=\Phi^{-1}$ by Corollary \ref{Corollary: continuous orbit equivalence implies groupoid isomorphisms}.
	Then it follows that
	\[\sigma_Y(\Phi(y,n,x))=\sigma_{l-k}(y,n,x)=n=\sigma_X(y,n,x)\]
	for all $(y,n,x)\in G(X,T)$ from the definition of $\sigma_{l-k}$ in Proposition \ref{prop: groupoid hom induced by continuous orbit equivalent}.
	Hence we obtain $\sigma_Y\circ\Phi=\id_{\Z}\circ \sigma_X$.
	In case that $l-k=l'-k'=-1$,
	we obtain 
	\[\sigma_Y(\Phi(y,n,x))=\sigma_{l-k}(y,n,x)=-n=-\sigma_X(y,n,x).\]
	Hence we obtain $\sigma_Y\circ\Phi=(-\id_{\Z})\circ \sigma_X$ and have shown (1)$\Rightarrow$(2).
	
	We show (2)$\Rightarrow$(1).
	Assume that an isomorphism $\Phi\colon G(X,T)\to G(Y,S)$ satisfies $\sigma_Y\circ \Phi=\tau\circ \sigma_X$ for some $\tau\in\Aut(\Z)$.
	Note that $\tau=\pm \id_{\Z}$ since we have $\Aut(\Z)=\{\id_{\Z}, -\id_{\Z}\}$.
	By Corollary \ref{cor: isom of DR groupoid induces continuous orbit equivalent},
	we obtain a continuous orbit equivalence $(l,k,l',k',\Phi|_X)$ between $(X,T)$ and $(Y,S)$ such that
	\begin{align*}
	l(x)-k(x)&=\sigma_Y(\Phi(x,1,T(x))) \\
	l'(y)-k'(y)&=\sigma_X(\Phi^{-1}(y,1,S(y)))
	\end{align*}
	for all $x\in X$ and $y\in Y$.
	Now,
	we have
	\begin{align*}
	l(x)-k(x)&=\sigma_Y(\Phi(x,1,T(x)))=\tau(\sigma_X(x,1,T(x)))=\tau(1) \\
	l'(y)-k'(y)&=\sigma_X(\Phi^{-1}(y,1,S(y)))=\tau^{-1}(\sigma_Y(y,1,S(y)))=\tau^{-1}(1)
	\end{align*}
	for all $x\in X$ and $y\in Y$.
	If $\tau=\id_{\Z}$,
	we obtain $l-k=l'-k'=1$.
	If $\tau=-\id_{\Z}$,
	we obtain $l-k=l'-k'=-1$.
	Therefore,
	$(X,T)$ and $(Y,S)$ are flip eventually conjugate.
	
	Now,
	it remains to show (3)$\Rightarrow$(2) under the assumption that $\ker\sigma_X$ and $\ker\sigma_Y$ are topologically transitive.
	This follows from Corollary \ref{cor: ker preserving groupoid isom is groupoid equivalent isom}.
	\qed
\end{proof}

\begin{ex}
	We observe that (3)$\Rightarrow$(2) in Theorem \ref{theorem: characterization of flip eventually conjugate} does not hold in general.
	Put $X\colon \N$ and define $T\colon \N\to \N$ by $T(n)=n+1$ for $n\in\N$.
	Similarly,
	put $Y=\Z$ and define $S\colon \Z\to \Z$ by $S(n)=n+1$ for $n\in \Z$.
	Then $G(X,T)$ and $G(Y,S)$ are isomorphic.
	Indeed, they are isomorphic to the discrete equivalence relation $\N\times \N$.
	Since $\ker\sigma_X=X$ and $\ker\sigma_Y=Y$,
	$(X,T)$ and $(Y,S)$ satisfies (3) in Theorem \ref{theorem: characterization of flip eventually conjugate}.
	However,
	$(X,T)$ and $(Y,S)$ satisfy neither (1) nor (2).
	Indeed,
	if $(X,T)$ and $(Y,S)$ are flip eventually conjugate,
	then there exists a bijection $h\colon \N\to\Z$ such that
	\[
	S^{l(x)}(h(x))=S^{k(x)}(h(T(x)))
	\]
	holds for all $x\in X$.
	Then we have
	\[
	T(x)=h^{-1}(S^{l(x)-k(x)}(h(x)))=h^{-1}(S^{\pm 1}(h(x)))
	\]
	for all $x\in X$ since $S$ is invertible and $l-k=\pm 1$.
	Hence we obtain $T=h^{-1}\circ S^{\pm 1}\circ h$,
	which is a contradiction since $h^{-1}\circ S^{\pm 1}\circ h$ is a bijection and $T$ is not.
	
\end{ex}

\subsection{Restricted Weyl group of Deaconu-Renault system}\label{subsection: Restricted Weyl group of Deaconu-Renault system}

In this subsection,
we investigate the restricted Weyl group $\mathfrak{RW}_{G(X,T),\ker\sigma_X}$ associated with a topologically free Deaconu-Renault system $(X,T)$.
Our aim is to show that $\mathfrak{RW}_{G(X,T),\ker\sigma_X}$ is isomorphic to the group of the eventually conjugate automorphisms on $(X,T)$ under some assumptions (Corollary \ref{cor: eventually conjugate automorphism is isomorphic to restricted Weyl of DR system}).
It is worth to note that ``flip'' cannot occur if a Deaconu-Renault system $(X,T)$ is far from injective (Proposition \ref{prop: flip cannot occur if TU=X}).

Following \cite[Section 3]{ContiHongSzymaski} and \cite[Section 4]{CONTI20122529},
we define property (P) and the group of eventually conjugate automorphisms on a Deaconu-Renault system.

\begin{defi}
	Let $(X,T)$ be a Deaconu-Renault system such that $X$ is compact.
	We say that $h\in \Aut(X)$ has property (P) if there exists $m\in\N$ such that
	\[
	T^{m+1}\circ h=T^m\circ h\circ T
	\]
	holds.
	We say that $h\in \Aut(X)$ is an eventually conjugate automorphism on $(X,T)$ if the both of $h$ and $h^{-1}$ have property (P).
	Define the group $\mathfrak{A}_{(X,T)}$ of the eventually conjugate automorphisms on $(X,T)$ as
	\[
	\mathfrak{A}_{(X,T)}\defeq \{h\in\Aut(X)\mid \text{$h$ and $h^{-1}$ have property (P)}\}.
	\]
\end{defi}

\begin{rem}
	One can check that $\mathfrak{A}_{(X,T)}$ is actually a subgroup of $\Aut(X)$.
	Note that $h\in\mathfrak{A}_{(X,T)}$ holds if and only if there exists $n,m\in\N$ such that $(n,n-1,m,m-1,h)$ is a self-eventual conjugacy on $(X,T)$.
	We use this identification to apply Theorem \ref{theorem: characterization of flip eventually conjugate} to $h\in\mathfrak{A}_{(X,T)}$.
	In addition,
	if there exists a self-eventual conjugacy $(l,k,l',k',h)$ on $(X,T)$,
	we have $h\in\mathfrak{A}_{(X,T)}$ by Proposition \ref{prop: if X and Y are compact, cocycles are constant}. 
	Indeed,
	since we assume that $X$ is compact,
	we may put $n\defeq \max_{x\in X}l(x)$ and $m\defeq \max_{x\in X}l'(x)$.
	Then $(n,n-1,m,m-1,h)$ is also a self-eventual conjugacy on $(X,T)$.	 
\end{rem}

\begin{prop}\label{prop: G has no contractive bisection if d^1(x) is finite}
	Let $G$ be a locally compact Hausdorff \'etale groupoid.
	Assume that $d^{-1}(\{x\})$ is a finite set for all $x\in G^{(0)}$.
	Then there exists no compact open bisection $W\subset G$ such that $r(W)\subsetneq d(W)$.
\end{prop}

\begin{proof}
	Assume that there exists a compact open bisection $W\subset G$ such that $r(W)\subsetneq d(W)$.
	Let $S\defeq \chi_W\in C_c(G)$ denote the characteristic function on $W$.
	Then we have 
	\[SS^*=\chi_{r(W)}\lneq \chi_{d(W)}=S^*S.\]
	Since the each left regular representation $\lambda_x$ at $x\in G^{(0)}$ is a finite dimensional representation,
	we obtain $\lambda_x(SS^*)=\lambda_x(S^*S)$.
	Hence we obtain $SS^*=S^*S$ and hence $d(W)=r(W)$.
	This contradicts to $r(W)\subsetneq d(W)$.
	\qed
\end{proof}

\begin{lem}\label{lemma property of R_n in Deaconu-Renault}
	Let $(X,T)$ be a Deaconu-Renault system.
	For $n\in\N$,
	define
	\[
	R_n\defeq \{(y,0,x)\in G(X,T)\mid T^n(x)=T^n(x)\}.
	\]
	Then the followings hold:
	\begin{enumerate}
		\item $\ker\sigma_X=\bigcup_{n\in\N}R_n$,
		\item for all $x\in X$, $R_n$ is an open subgroupoid of $G(X,T)$ such that $X\subset R_n$, and
		\item if $T\colon X\to X$ is a proper map (i.e.\ $T^{-1}(K)$ is compact for all compact set $K\subset X$),
		then $d|_{R_n}^{-1}(\{x\})$ is a finite set for all $x\in X$.
	\end{enumerate}
\end{lem}

\begin{proof}
	(1) is straightforward.
	To show (2),
	check
	\[
	R_n=\bigcup_{U,V} Z(U,n,n,V),
	\]
	where the union of the right hand side is taken over all open sets $U,V\subset X$.
	To show (3),
	observe
	\[
	d|_{R_n}^{-1}(\{x\})=\{(y,0,x)\in G(X,T)\mid y\in T^{-n}(\{T^n(x)\})\},
	\]
	which is a finite set since we assume that $T$ is proper and locally homeomorphic.
	\qed
\end{proof}

\begin{prop}\label{prop: flip cannot occur if TU=X}
	Let $(X,T)$ be a Deaconu-Renault system.
	Assume that there exists a compact open set $U\subsetneq X$ such that $T(U)=X$ and $T|_U$ is injective (and hence $X$ is compact and $T$ is proper).
	Then there is no automorphism $\Phi\colon G(X,T)\to G(X,T)$ such that $\sigma_X\circ \Phi= (-\id_{\Z})\circ\sigma_X$.
\end{prop}

\begin{proof}
	Assume that there exists an automorphism $\Phi\colon G(X,T)\to G(X,T)$ such that $\sigma_X\circ \Phi= (-\id_{\Z})\circ\sigma_X$.
	Put
	\[
	\widetilde{W}\defeq Z(U,1,0, X)\subset G(X,T).
	\]
	Then $\widetilde{W}$ is a compact open bisection with $d(\widetilde{W})=X$,
	$r(\widetilde{W})=U$ and $\sigma_X(\widetilde{W})=1$.
	Put
	\[
	W\defeq \widetilde{W}\Phi(\widetilde{W}).
	\]
	Then $W\subset G(X,T)$ is a compact open bisection with $d(W)=X$ and $r(W)=U$.
	In addition,
	we have $W\subset \ker\sigma_X$ and hence $W\subset R_n$ for some $n\in \N$ by the compactness of $W$ and Lemma \ref{lemma property of R_n in Deaconu-Renault}.
	Now,
	Proposition \ref{prop: G has no contractive bisection if d^1(x) is finite} yields a contradiction since $d|_{R_n}^{-1}(x)$ is a finite set for all $x\in X$ by Lemma \ref{lemma property of R_n in Deaconu-Renault}.
	\qed
\end{proof}

\begin{prop}\label{prop: eventually conj auto isomorphic to sigma preserving groupoid auto}
	Let $(X,T)$ be a topologically free Deaconu-Renault system.
	Define a subgroup $\mathfrak{G}$ of $\Aut(G(X,T))$ by
	\[
	\mathfrak{G}\defeq \{\Phi\in\Aut(G(X,T))\mid \sigma_X\circ\Phi=\sigma_X\}.
	\]
	Then $\Phi|_X\in\mathfrak{A}_{(X,T)}$ holds for all $\Phi\in\mathfrak{G}$.
	In addition,
	\[\Psi\colon \mathfrak{G}\ni \Phi\mapsto \Phi|_{X}\in \mathfrak{A}_{(X,T)}\]
	is a group isomorphism.
\end{prop}

\begin{proof}
	Take $\Phi\in\mathfrak{G}$.
	Then we have $\Phi|_X\in\mathfrak{A}_{(X,T)}$ as shown in the proof of (2)$\Rightarrow$(1) in Theorem \ref{theorem: characterization of flip eventually conjugate}.
	Now,
	one chan check that $\Psi$ is a group homomorphism.
	To show that $\Psi$ is surjective,
	take $h\in\mathfrak{A}_{(X,T)}$.
	By Corollary \ref{Corollary: continuous orbit equivalence implies groupoid isomorphisms},
	there exists $\Phi\in\Aut(G(X,T))$ such that $\Phi|_X=h$.
	One can check that $\Phi\in\mathfrak{G}$ in the same way as the proof of (1)$\Rightarrow$(2) in Theorem \ref{theorem: characterization of flip eventually conjugate}.
	Hence $\Psi$ is surjective.
	It follows that $\Psi$ is injective from Proposition \ref{prop: all we need is effectiveness}.
	Therefore $\Psi$ is an isomorphism.
	\qed
\end{proof}

\begin{cor}\label{cor: eventually conjugate automorphism is isomorphic to restricted Weyl of DR system}
	Let $(X,T)$ be a topologically principal Deaconu-Renault system such that $\ker\sigma_X$ is topologically transitive.
	Assume that there exists a compact open set $U\subsetneq X$ such that $T(U)=X$ and $T|_U$ is injective (and hence $X$ is compact and $T$ is proper).
	Then
	\[
	\Aut(G(X,T);\ker\sigma_X)=\{\Phi\in\Aut(G(X,T))\mid \sigma_X\circ \Phi=\sigma_X\}
	\]
	holds.
	In particular,
	the groups $\mathfrak{RW}_{G(X,T),\ker\sigma_X}$, $\Aut(G(X,T);\ker\sigma_X)$ and $\mathfrak{A}_{(X,T)}$ are isomorphic to each others.
\end{cor}

\begin{proof}
	By (3)$\Rightarrow$(2) in Theorem \ref{theorem: characterization of flip eventually conjugate} and Proposition \ref{prop: flip cannot occur if TU=X},
	we have
	\[
	\Aut(G(X,T);\ker\sigma_X)=\{\Phi\in\Aut(G)\mid \sigma_X\circ \Phi=\sigma_X\}.
	\]
	Now,
	by Proposition \ref{prop: eventually conj auto isomorphic to sigma preserving groupoid auto},
	we obtain a group isomorphism
	\[
	\Psi\colon \Aut(G(X,T),\ker\sigma_X)\ni \Phi\to \Phi|_X\in \mathfrak{A}_{(X,T)}.
	\]
	By Corollary \ref{cor: restricted Weyl isomorphic to restricted groupoid Weyl},
	$\mathfrak{RW}_{G(X,T),\ker\sigma_X}$ and $\Aut(G(X,T);\ker\sigma_X)$ are isomorphic.
	Hence we complete the proof.
	\qed
\end{proof}

\subsection{Restricted Weyl group of graph algebras}\label{subsection: Restricted Weyl group of graph algebras}

	In this subsection,
	we aim to apply the results in the previous subsections to graph algebras.
	As a consequence,
	in Corollary \ref{cor: answer to open problem of Conti},
	we obtain an affirmative answer of the open problem mentioned under \cite[Theorem 4.13]{CONTI20122529}.
	Below \cite[Theorem 4.13]{CONTI20122529},
	the authors ask if the restricted Weyl group of graph algebras $C^*(E)$ is isomorphic to $\mathfrak{A}_{(E^{(\infty)},T_E)}$ under some assumptions,
	where $E^{(\infty)}$ denotes the infinite path space on a directed graph $E$ and $T_E$ denotes the shift on $E^{(\infty)}$.
	After observing that we may apply Corollary \ref{cor: eventually conjugate automorphism is isomorphic to restricted Weyl of DR system},
	we give an affirmative answer to this question in Corollary \ref{cor: answer to open problem of Conti}.
	
	First,
	we recall fundamental definitions and properties of graph algebras.
	See \cite{raeburn2005graph} and \cite{Paterson2002} for more details about graph algebras and groupoid models of graph algebras.
	We also refer to \cite{BROWNLOWE_CARLSEN_WHITTAKER_2017} for groupoid models of graph algebras.
	
	Let $E=(V,E,o,t)$ be a finite directed graph.
	Namely,
	$V$ and $E$ are finite sets and $o,t\colon E\to V$ are origin and target maps respectively.
	An element in $V$ and $E$ are called a vertex and edge respectively.
	For $l\in\N$,
	the set of all finite paths on $E$ of length $l$ is denoted by
	\[E^{(l)}\defeq \{(\mu_1,\mu_2,\cdots,\mu_l)\in E^l\mid t(\mu_i)=o(\mu_{i+1}) \text{ for all $i=1,\dots,l-1$} \}.\]
	The set of all finite paths on $E$ is denoted by $E^*\defeq \bigcup_{l\in\N} E^{(n)}$.
	For $\mu\in E^*$,
	$\lvert\mu\rvert\in\N$ denotes the length of $\mu$.
	Then the origin and target map are extended to the maps on $E^*$ by $o(\mu)\defeq o(\mu_1)$ and $t(\mu)\defeq t(\mu_{\lvert\mu\rvert})$ for $\mu\in E^*$. 
	A vertex $v\in V$ is called a sink if $o^{-1}(\{v\})=\emptyset$ holds.
	For simplicity,
	we treat a finite directed graph with no sink.
	See \cite{BROWNLOWE_CARLSEN_WHITTAKER_2017} for a general case.
	
	For a finite directed graph $E$ with no sink,
	the graph algebra $C^*(E)$ is defined to be the universal C*-algebra which is generated by mutually orthogonal projections $\{P_v\}_{v\in V}$ and partial isometries $\{S_e\}_{e\in E}$ such that
	\begin{enumerate}
		\item $S_e^*S_e=P_{t(e)}$ for all $e\in E$, and
		\item $P_v=\sum_{e\in o^{-1}(\{v\})}S_eS_e^*$ for all $v\in V$.
	\end{enumerate}
	The gauge action on $C^*(E)$ is an action $\tau\colon \T\curvearrowright C^*(E)$ defined by
	\[
	 \tau_z(P_v)=P_v, \tau_z(S_e)=zS_e
	\]
	for all $z\in \T$,
	$e\in E$ and $v\in V$.
	We define
	\[
	D_E\defeq \overline{\Span}\{S_{\mu}S_{\mu}^* \mid \mu\in E^*\}\subset C^*(E),
	\]
	where we put $S_{\mu}\defeq S_{\mu_1}S_{\mu_2}\cdots S_{\mu_{\lvert\mu\rvert}}$ for a finite path $\mu\in E^*$.
	Then $D_E$ is a commutative C*-subalgebra of $C^*(E)$ and $D_E\subset C^*(E)^{\tau}$ holds,
	where $C^*(E)^{\tau}$ denotes the invariant subalgebra of the gauge action.
	
	Now,
	we introduce a groupoid model of graph algebras.
	Let $E^{(\infty)}$ denote the set of infinite path on $E$, namely,
	\[
	E^{(\infty)}\defeq \{\{x_i\}_{i=0}^{\infty}\in E^{\N}\mid t(x_i)=o(x_{i+1})\text{ for all $i\in\N$}\}.
	\]
	Then $E^{(\infty)}$ is a compact Hausdorff space with respect to the relative topology of the product topology.
	For $\mu\in E^*$,
	let $C(\mu)$ denote the set of all infinite paths which begin with $\mu$ :
	\[
	C(\mu)\defeq \{\mu x\in E^{(\infty)}\mid x\in E^{(\infty)}, t(\mu)=o(x_0)\}.
	\]
	Note that $C(\mu)$ is a compact open set in $E^{(\infty)}$ and $\{C(\mu)\}_{\mu\in E^*}$ is an open basis of $E^{(\infty)}$.
	The shift map on $E^{(\infty)}$ is denoted by
	\[
	T_E\colon E^{(\infty)}\ni \{x_i\}_{i=0}^{\infty}\mapsto \{x_{i+1}\}_{i=0}^{\infty}\in E^{(\infty)}.
	\]
	Then $(E^{(\infty)}, T_E)$ is a Deaconu-Renault system.	
	By \cite[Proposition 2.2]{BROWNLOWE_CARLSEN_WHITTAKER_2017},
	there exists a *-isomorphism \[\pi\colon C^*(E)\to C^*_r(G(E^{(\infty)},T_E))\] such that $\pi(D_E)=C(E^{(\infty)})$ and $\pi(C^*(E)^{\tau})=C^*_r(\ker\sigma_{E^{(\infty)}})$.
		
	Following \cite[Proposition 2.3]{BROWNLOWE_CARLSEN_WHITTAKER_2017},
	we characterize the topological principality of $G(E^{(\infty)}, T_E)$ in terms of a directed graph $E$.
	A path $\mu\in E^*$ is called a cycle if $\lvert\mu\rvert\geq 1$ and $o(\mu)=t(\mu)$.
	An edge $e\in E$ is called an exit of a cycle $\mu$ if there exists $i\in\{1,2,\dots,\lvert \mu\rvert\}$ such that $o(e)=o(\mu_i)$ and $e\not=\mu_i$.
	A directed graph $E$ is said to have condition (L) if every cycle has an exit.
	We have the following characterization of topological principality of $G(E^{(\infty)},T_E)$.
	
	\begin{prop}[{\cite[Proposition 2.3]{BROWNLOWE_CARLSEN_WHITTAKER_2017}}] \label{prop: characterization that graph groupoid is topologically principal}
		Let $E$ be a finite directed graph with no sink.
		Then $G(E^{(\infty)}, T_E)$ is topologically principal if and only if $E$ has condition (L).
	\end{prop}
	
	Next,
	we characterize the topological transitivity of $\ker\sigma_{E^{(\infty)}}$.	
	
	\begin{prop}\label{prop: characterization that ker of graph groupoid is topologically transitive}
		Let $E$ be a finite directed graph with no sink.
		Then $\ker\sigma_{E^{(\infty)}}\subset G(E^{(\infty)},T_E)$ is topologically transitive if and only if the following condition holds : for all $v,w\in V$,
		there exists $\mu,\nu\in E^*$ with $\lvert\mu\rvert=\lvert\nu\rvert$ such that $v=o(\mu)$, $w=o(\nu)$ and $t(\mu)=t(\nu)$ hold.
	\end{prop}
	
	\begin{proof}
	First,
		we assume that $\ker\sigma_{E^{(\infty)}}\subset G(E^{(\infty)},T_E)$ is topologically transitive and take $v,w\in V$ arbitrarily.
		Since we assume that $E$ has no sink,
		$C(v)$ and $C(w)$ are non-empty open subsets of  $E^{(\infty)}$.
		Hence there exists $(y,0,x)\in \ker\sigma_{E^{(\infty)}}$ such that $y\in C(w)$ and $x\in C(v)$.
		Then there exists $l\in\N$ such that $T_E^l(y)=T_E^l(x)$.
		Putting $\mu\defeq x_0x_1x_2\cdots x_l$ and $\nu\defeq y_0y_1y_2\cdots y_l$,
		we obtain $\mu,\nu\in E^*$ with $\lvert\mu\rvert=\lvert\nu\rvert$ such that $v=o(\mu)$, $w=o(\nu)$ and $t(\mu)=t(\nu)$ hold.
		
		Next,
		we show the converse.
		To show that $\ker\sigma_{E^{(\infty)}}$ is topologically transitive,
		take $\mu,\nu\in E^*$ arbitrarily and show that there exists $\alpha\in \ker\sigma_{E^{(\infty)}}$ such that $r(\alpha)\in C(\nu)$ and $d(\alpha)\in C(\mu)$.
		We may assume $\lvert\nu\rvert \geq \lvert \mu\rvert$ without loss of generality.
		Then there exists $\eta\in E^*$ with $t(\mu)=o(\eta)$ and $\lvert\nu \rvert=\lvert\mu\eta\rvert$ since we assume that $E$ has no sink.
		Apply the assumption to $w\defeq t(\nu)$ and $v\defeq t(\eta)$,
		we obtain $\mu',\nu'\in E^*$ with $\lvert\mu'\rvert=\lvert\nu'\rvert$ such that $v=o(\mu')$, $w=o(\nu')$ and $t(\mu')=t(\nu')$ hold.
		In addition,
		take $x\in E^{(\infty)}$ with $t(\mu')=o(x_0)$.
		Put
		\[
		\alpha\defeq (\nu\nu'x,\lvert \nu\nu'\rvert-\lvert\mu\eta\mu'\rvert ,\mu\eta\mu'x)\in G(E^{(\infty)},T_E).
		\]
		Then one can check that $\alpha\in\ker\sigma_{E^{(\infty)}}$,
		$r(\alpha)\in C(\nu)$ and $d(\alpha)\in C(\mu)$.
		This completes the proof.
		\qed
	\end{proof}
	
	Now,
	we are ready to apply Corollary \ref{cor: eventually conjugate automorphism is isomorphic to restricted Weyl of DR system} to graph algebras.
	
	\begin{cor}\label{cor: restricted Weyl group of graph algebras}
		Let $E$ be a finite directed graph with no sink.
		Assume that $E$ satisfies condition (L) and the following condition : for all $v,w\in V$,
		there exists $\mu,\nu\in E^*$ with $\lvert\mu\rvert=\lvert\nu\rvert$ such that $v=o(\mu)$, $w=o(\nu)$ and $t(\mu)=t(\nu)$ hold.
		In addition,
		assume that $E$ has no source (i.e.\ $t^{-1}(v)$ is non-empty for all $v\in V$).
		Then the groups $\mathfrak{RW}_E=\mathfrak{RW}_{G(E^{(\infty)},T_E),\ker\sigma_{E^{(\infty)}}}$, $\Aut(G(E^{(\infty)},T_E);\ker\sigma_{E^{(\infty)}})$ and $\mathfrak{A}_{(E^{(\infty)},T_{E})}$ are isomorphic to each others.
	\end{cor}
	\begin{proof}
	By Proposition \ref{prop: characterization that graph groupoid is topologically principal} and Proposition \ref{prop: characterization that ker of graph groupoid is topologically transitive},
	$G(E^{(\infty)},T_E)$ is topologically principal and $\ker\sigma_{E^{(\infty)}}$ is topologically transitive.
	We check that there exists a compact open set $U\subsetneq E^{(\infty)}$ such that $T_E|_U$ is injective and $T_E(U)=E^{(\infty)}$.
	For each $v\in V$,
	take $e_v\in E$ with $t(e_v)=v$.
	Put $U\defeq \bigcup_{v\in V} C(e_v)\subset E^{(\infty)}$.
	Then $U$ is a compact open subset of $E^{(\infty)}$.
	In addition,
	one can check that $T_E(U)=E^{(\infty)}$ and $T_E(U)=E^{(\infty)}$.
	To show $U\subsetneq E^{(\infty)}$,
	we prepare the following lemma.
	
	\begin{lem*}
		There exists $w\in V$ such that $t^{-1}(\{w\})$ has at least two elements.
	\end{lem*}
	\begin{proof}
	Assume that $t^{-1}(\{w\})$ is a singleton for all $w\in V$.
	Since $E$ is finite and with no sink,
	there exists a cycle in $E$.
	Since we assume that $E$ satisfies condition (L),
	there exists $e,f\in E$ such that $o(e)=o(f)$ and $e\not=f$.
	By the assumption,
	there exists $\mu,\nu\in E^*$ such that $\lvert \mu\rvert=\lvert\nu\rvert$, $o(\mu)=t(e)$, $o(\nu)=t(f)$ and $t(\mu)=t(\nu)$.
	Applying the assumption that $t^{-1}(\{w\})$ is a singleton for all $w\in V$ recursively,
	we obtain $e\mu=f\nu$ and this contradicts to $e\not=f$.\qed
	\end{proof}

	Now,
	$U\subsetneq E^{(\infty)}$ follows from the previous lemma.
	Hence,
	we may apply Corollary \ref{cor: eventually conjugate automorphism is isomorphic to restricted Weyl of DR system} and this completes the proof of Corollary \ref{cor: restricted Weyl group of graph algebras}.
	\qed
	\end{proof}
	
	The assumptions in Corollary \ref{cor: restricted Weyl group of graph algebras} looks stronger than those in the open problem mentioned under \cite[Theorem 4.13]{CONTI20122529}.
	Indeed,
	in \cite[Theorem 4.13]{CONTI20122529},
	the authors assumed that 
	\begin{enumerate}
	\item $E$ is a finite directed graph without sinks and sources,
	\item $E$ satisfies condition (L), and
	\item the centre of $C^*(E)^{\tau}$ is trivial. \label{enumerate: assumption}
	\end{enumerate}
	Since,
	by Proposition \ref{prop: relative commutant is trivial if H is topo transitive},
	Proposition \ref{prop: characterization that graph groupoid is topologically principal} and Proposition \ref{prop: characterization that ker of graph groupoid is topologically transitive},
	one can deduce the above assumption (\ref{enumerate: assumption}) by the assumptions in Corollary \ref{cor: restricted Weyl group of graph algebras}.
	Remark that the other assumptions are common to ours.
	Hence our assumptions in Corollary \ref{cor: restricted Weyl group of graph algebras} implies those in \cite[Theorem 4.13]{CONTI20122529}.
	Note that the authors proved that the natural map $\mathfrak{RW}_E\to \mathfrak{A}_{(E^{(\infty)},T_{E})}$ yields an injective group homomorphism and the surjectivity is an open problem as mentioned under \cite[Theorem 4.13]{CONTI20122529}.

	In the rest of this subsection,
	we show that our assumption in Corollary \ref{cor: restricted Weyl group of graph algebras} can be relaxed to the above original assumptions in \cite[Theorem 4.13]{CONTI20122529}.
	As a result,
	we solve the open problem mentioned under \cite[Theorem 4.13]{CONTI20122529}.
	To do that,
	we prepare the following proposition about graph algebras.
	
	\begin{prop}\label{prop: in graph algebra, irreducible fixed subalgebra implies topological transitivity}
		Let $E$ be a finite directed graph with no sink.
		Assume that the relative commutant of $C^*(E)^{\tau}$ in $C^*(E)$ is trivial.
		Then,
		for all $v,w\in V$,
		there exists $\mu,\nu\in E^*$ with $\lvert\mu\rvert=\lvert\nu\rvert$ such that $v=o(\mu)$, $w=o(\nu)$ and $t(\mu)=t(\nu)$ hold.
	\end{prop}

	\begin{proof}
		We show the contraposition.
		For $v,w\in V$,
		we write $v\sim w$ if there exists $\mu,\nu\in E^*$ with $\lvert\mu\rvert=\lvert\nu\rvert$ such that $v=o(\mu)$, $w=o(\nu)$ and $t(\mu)=t(\nu)$ hold.
		By the assumption,
		there exists $v_0,w_0\in V$ such that $v_0\not\sim w_0$.
		Put
		\[
		F\defeq \{v\in V\mid v_0\sim v\}\subset V
		\]
		and $P\defeq \sum_{v\in F}P_v\in C^*(E)$.
		Then $P$ is contained in the relative commutant of $C^*(E)^{\tau}$.
		Indeed,
		take $\mu,\nu\in E^*$ with $\lvert \mu\rvert=\lvert\nu\rvert$ and $t(\mu)=t(\nu)$.
		If $o(\mu)\in F$,
		then one can check that
		\[
		PS_{\mu}S_{\nu}^*=S_{\mu}S_{\nu}^*P=S_{\mu}S_{\nu}^*
		\]
		holds.
		If $o(\mu)\not\in F$,
		then 
		\[PS_{\mu}S_{\nu}^*=S_{\mu}S_{\nu}^*P=0\]
		holds.
		Since $C^*(E)^{\tau}$ is the closed linear span of
		\[
		\{S_{\mu}S_{\nu}^*\mid \mu,\nu\in E^*, \lvert\mu\rvert=\lvert\nu\rvert\},
		\]
		we obtain $P\in (C^*(E)^{\tau})'$.
		Since we have $v_0\in F$ and $w_0\not\in F$,
		we obtain $P\not=0,1$ and $(C^*(E)^{\tau})'$ is non-trivial.
		This completes the proof.
		\qed
	\end{proof}
	
	Now,
	we relax our assumptions in Corollary \ref{cor: restricted Weyl group of graph algebras} to the original assumptions in the open problem mentioned under \cite[Theorem 4.13]{CONTI20122529}.
	
	\begin{cor}[{cf.\ \cite[Theorem 4.13]{CONTI20122529}}] \label{cor: answer to open problem of Conti}
		Let $E$ be a finite directed graph with no sink.
		Assume that $E$ satisfies condition (L) and the centre of $C^*(E)^{\tau}$ is trivial.
		In addition,
		assume that $E$ has no source (i.e.\ $t^{-1}(v)$ is non-empty for all $v\in V$).
		Then the groups 
		$\mathfrak{RW}_E=\mathfrak{RW}_{G(E^{(\infty)},T_E),\ker\sigma_{E^{(\infty)}}}$, $\Aut(G(E^{(\infty)},T_E);\ker\sigma_{E^{(\infty)}})$ and $\mathfrak{A}_{(E^{(\infty)},T_{E})}$ are isomorphic to each others.
	\end{cor}
	
	\begin{proof}
		Since we assume condition (L),
		$D_E$ is a masa in $C^*(E)$ by \cite[Theorem 5.2]{MR2188247}.
		Combining with $D_E\subset C^*(E)^{\tau}$,
		we obtain $(C^*(E)^{\tau})'\subset D_E\subset C^*(E)^{\tau}$,
		where $(C^*(E)^{\tau})'$ denotes the relative commutant of $C^*(E)^{\tau}$ in $C^*(E)$.
		Hence we obtain $(C^*(E)^{\tau})'=(C^*(E)^{\tau})'\cap C^*(E)^{\tau}$.
		Since we assume that the centre of $C^*(E)^{\tau}$ is trivial,
		the relative commutant $(C^*(E)^{\tau})'$ is also trivial.
		By Proposition \ref{prop: in graph algebra, irreducible fixed subalgebra implies topological transitivity},
		we may apply Corollary \ref{cor: restricted Weyl group of graph algebras} and this completes the proof.
		\qed
	\end{proof}

\bibliographystyle{plain}
\bibliography{bunken}

\end{document}